\documentclass[12pt]{amsart}
\usepackage{amssymb, a4wide, mathdots, url, graphicx}
\usepackage[usenames]{color}
\usepackage{enumerate}
\usepackage{mathtools}
\usepackage{xcolor}
\usepackage[hyperfootnotes=false, colorlinks, linkcolor={blue}, citecolor={magenta}, filecolor={blue}, urlcolor={blue}, plainpages=false, pdfpagelabels]{hyperref}
\usepackage{amsmath,amscd}
 
\newcommand{\Ad}{\operatorname{Ad}}

\newcommand{\Eisen}{\mathcal{E}}

\renewcommand{\Re}{\operatorname{Re}}

\newcommand{\A}{\mathbb{A}}

\newcommand{\C}{\mathbb{C}}

\newcommand{\Mat}{\operatorname{Mat}}
\newcommand{\Hom}{\operatorname{Hom}}

\newcommand{\Sym}{\operatorname{Sym}}

\newcommand{\M}{\mathcal{M}}

\newcommand{\wt}{\widetilde}

\newcommand{\Span}{\operatorname{Span}}

\newcommand{\bb}{\mathfrak{b}}
\newcommand{\GG}{\mathbf{G}}

\newtheorem{theorem}{Theorem}[section]
\newtheorem{lemma}[theorem]{Lemma}
\newtheorem{definition}[theorem]{Definition}
\newtheorem{proposition}[theorem]{Proposition}
\newtheorem{corollary}[theorem]{Corollary}
\theoremstyle{remark}
\newtheorem{remark}[theorem]{Remark}

\makeatletter
\@namedef{subjclassname@2020}{%
  \textup{2020} Mathematics Subject Classification}
\makeatother

\newcommand{\pr}{\operatorname{pr}}

\newcommand{\NN}{\mathcal{N}}

\DeclareMathOperator{\diag}{diag}

\DeclareMathOperator{\Ind}{Ind}
\DeclareMathOperator{\GL}{GL}

\DeclareMathOperator{\SO}{SO}

\DeclareMathOperator{\GSpin}{GSpin}

\DeclareMathOperator{\GSO}{GSO}

\DeclareMathOperator{\GSp}{GSp}

\DeclareMathOperator{\LL}{\mathcal{L}}

\begin{document}

\title[Rankin-Selberg integrals for $\mathrm{GSpin}$ groups with application]{Rankin-Selberg integrals for $\mathrm{GSpin}$ groups with application to the global Gan-Gross-Prasad conjecture}

\author{Pan Yan}
\address{Department of Mathematics, The University of Arizona, Tucson, AZ 85721, USA}
\email{panyan@arizona.edu}

\date{\today}


\begin{abstract}
We construct Rankin-Selberg integrals using Bessel models for a product of tensor product partial $L$-functions
\begin{equation*}
L^S(s,\pi\times\tau_1) L^S(s,\pi\times\tau_2)\cdots L^S(s,\pi\times\tau_r)
\end{equation*}
where $\pi$ is an irreducible cuspidal automorphic representation of a quasi-split $\GSpin$ group, and $\tau_1, \cdots, \tau_r$ are irreducible unitary cuspidal automorphic representations of $\GL_{k_1}, \cdots, \GL_{k_r}$ respectively. 
As an application, we prove one direction of the global Gan-Gross-Prasad conjecture for generic representations of quasi-split $\GSpin$ groups. 
\end{abstract}

\maketitle

\goodbreak 

\tableofcontents

\goodbreak

\section{Introduction}

\subsection{Rankin-Selberg integrals}
The first goal of this paper is to construct a family of global Rankin-Selberg integrals for a finite product of tensor product partial $L$-functions for quasi-split $\GSpin$ groups and general linear groups.

We begin by introducing some notation.   
Let $F$ be a number field and let $\A=\A_F$ be the ring of adeles of $F$. Fix a non-trivial additive character $\psi$ of $F\backslash \A$. 
Let $\GSpin_{m^\prime}$ be a quasi-split $\GSpin$ group defined over $F$, associated to a quadratic space $V$ of dimension $m^\prime$ with Witt index $\tilde{m}$, which fits into a short exact sequence
\begin{equation*}
\begin{CD}
1 @>>>  \GL_1 @>>> \GSpin_{m^\prime} @>\pr>> \SO_{m^\prime} @>>>  1. 
\end{CD}
\end{equation*}
Note that $\ker(\pr)=\GL_1$ lies in the center of $\GSpin_{m^\prime}$, and in particular, if $m^\prime>2$, then $\ker(\pr)=\GL_1$ is the identity component of the center of $\GSpin_{m^\prime}$.
Take $\ell<\tilde{m}$ to be a non-negative integer. Let $N_{\ell}$ be the unipotent radical of the standard parabolic subgroup $Q_{\ell}$ of $\GSpin_{m^\prime}$, whose Levi subgroup, $L_{\ell}$, is isomorphic to $(\GL_1)^{\ell}\times\GSpin_{m^\prime-2\ell}$. The elements of $N_{\ell}$ can be represented by upper triangular unipotent matrices as in \eqref{eq-unipotent-Nl}. Let $\psi_{\ell, a}$ be a character of $N_{\ell}(\A)$ associated to $a\in F^\times$ which is trivial on $N_{\ell}(F)$, defined by \eqref{eq-psi-l-a}. Given an automorphic form $\phi$ on $\GSpin_{m^\prime}(\A)$, we will consider its Fourier coefficient of Bessel type (or the Gelfand-Graev coefficient) along $N_{\ell}$ with respect to $\psi_{\ell,a}$, given by 
\begin{equation*}
\phi^{N_{\ell}, \psi_{\ell, a}}(g):=\int_{N_{\ell}(F)\backslash N_{\ell}(\A)}\phi(ug)\psi_{\ell, a}^{-1}(u)du, \quad g\in \GSpin_{m^\prime}(\A).
\end{equation*}
The connected component of the identity of the stabilizer of $\psi_{\ell,a}$ in $L_{\ell}(\A)$ is a subgroup $G(\A)$, where $G$ is of the form $\GSpin_{m^\prime-2\ell-1}$.
Given a cuspidal automorphic form $\phi^\prime$ of $G(\A)$, such that $\phi(z)\phi^\prime(z)=1$ for all $z\in C^0_G(\A)$ where $C_G^0=\ker(\pr)=\GL_1$, we define the Bessel period for the pair $(\phi, \phi^\prime)$ by
\begin{equation}
\label{intro-eq-Bessel-period}
\mathcal{B}^{\psi_{\ell,a}}(\phi, \phi^\prime):=	 \int_{ C^0_G(\A) G(F)\backslash G(\A)} \phi^{N_{\ell}, \psi_{\ell, a}}(g) \phi^\prime(g)dg. 
\end{equation}

Let $k$ be an integer with $1\le k \le \tilde{m}$, and let $P_k$ be the standard parabolic subgroup of $\GSpin_{m^\prime}$, whose Levi part is isomorphic to $\GL_k\times \GSpin_{m^\prime-2k}$.
Take $\tau$ to be an irreducible unitary generic automorphic representation of $\GL_k(\A)$, of the following isobaric type:
\begin{equation}
\label{intro-eq-tau}
\tau=\tau_1\boxtimes \tau_2 \boxtimes \cdots \boxtimes \tau_{r},
\end{equation}
with $k=k_1+\cdots + k_r$, where each $\tau_i$ is an irreducible unitary cuspidal automorphic representation of $\GL_{k_i}(\A)$ for each $1\le i \le r$. 
Let $\sigma$ be an irreducible automorphic representation of $G_0(\A)=\GSpin_{m^\prime-2k}(\A)$. We denote by $\omega_\sigma$ the restriction of the central character of $\sigma$ to $C_{G_0}^0(\A)=\GL_1(\A)$.  Given a section 
$f_{\tau,\sigma, s}\in \Ind_{P_{k}(\A)}^{H(\A)}(\tau|\cdot|^{s-\frac{1}{2}} \otimes \sigma)$, 
we define the Eisenstein series
\begin{equation*}
E(g, f_{\tau,\sigma, s})=\sum_{\gamma\in P_k(F)\backslash H(F)} f_{\tau,\sigma, s}(\gamma g), \quad g\in \GSpin_{m^\prime}(\A).
\end{equation*}
Let $\pi$ be an irreducible cuspidal automorphic representation of $G(\A)$, such that $\omega_\pi(z) \omega_{\sigma}(z)=1$ for all $z\in C_G^0(\A)$, where $\omega_\pi$ is the restriction of the central characters of $\pi$ to $C_G^0(\A)=\GL_1(\A)$. Let  $\phi_\pi\in V_\pi$ be a non-zero cusp form. 
The global zeta integral we consider in this paper is defined to be the following Bessel period
\begin{equation}
\label{intro-eq-global-integral}
\mathcal{Z}(\phi_\pi, f_{\tau,\sigma, s}):=\mathcal{B}^{\psi_{\ell,a}}(E(\cdot, f_{\tau,\sigma, s}), \phi_\pi)= \int_{C_G^0(\A) G(F)\backslash G(\A)}  E^{N_\ell, \psi_{\ell,a}}(g, f_{\tau,\sigma,s}) \phi_\pi(g) dg.
\end{equation}
Due to the rapid decrease of the cusp form $\phi_\pi$, the integral $\mathcal{Z}(\phi_\pi, f_{\tau,\sigma, s})$ converges absolutely away from the poles of the Eisenstein series. 

Let  $\sigma^\prime$ be given as in \eqref{eq-sigma-prime-1} for $k>\ell$ and in \eqref{eq-sigma-prime-2} for $k\le \ell$.
If $k>\ell$, we assume that the pair $(\pi, \sigma^\prime)$  has a non-zero Bessel period. If $k\le \ell$, we assume the pair $(\sigma^\prime, \pi)$ has a non-zero Bessel period.

\begin{theorem}[Theorem~\ref{thm-global-unfolding-kgreaterl}, Theorem~ \ref{thm-global-unfolding-klel}, Theorem~\ref{thm-local-unramified-computation}]
\label{intro-thm-main}
Let the notation be as above. For decomposible vectors, the integral $\mathcal{Z}(\phi_\pi, f_{\tau,\sigma, s})$ has an Euler product factorization
\begin{equation*}
\mathcal{Z}(\phi_\pi, f_{\tau,\sigma, s})=\prod_{\nu} \mathcal{Z}_{\nu}(\phi_{\pi_{\nu}},  f_{\mathcal{W}(\tau_{\nu}),\sigma_{\nu},s})
\end{equation*}  
where the local zeta integral $\mathcal{Z}_{\nu}(\phi_{\pi_{\nu}},  f_{\mathcal{W}(\tau_{\nu}),\sigma_{\nu},s})$ is given in \eqref{eq-local-zeta-integral-k-larger-than-l} when $k>\ell$ and in \eqref{eq-local-zeta-integral-k-le-l} when $k\le \ell$ respectively. Moreover,  for $\mathrm{Re}(s)\gg 0$, at a finite place $\nu$ when all data are unramified, the local unramified zeta integral has the following expression 
\begin{equation*}
	\mathcal{Z}_\nu(v_{\pi_\nu}^0,	f_{\mathcal{W}(\tau_\nu),\sigma_\nu, s}^0)
	 =  \frac{L(s,\pi_\nu \times \tau_\nu)}{L(s+\frac{1}{2},\sigma_\nu  \times \tau_\nu \otimes\omega_{\pi_\nu}) L\left(2s, \tau_\nu, \rho \otimes \omega_{\pi_\nu}\right)}
\end{equation*}
where $v_{\pi_\nu}^0$, $f_{\mathcal{W}(\tau_\nu),\sigma_\nu, s}^0$ are spherical vectors suitably normalized, and
\begin{equation*}
\rho=
\begin{cases} 
\wedge^2, & 
\mbox{ if $G$ is an odd $\GSpin$ group,}
\\
\Sym^2, & 
\mbox{ if $G$ is an even $\GSpin$ group.}
\end{cases} 
\end{equation*}
\end{theorem}

For special orthogonal groups and unitary groups, analogous Rankin-Selberg integrals were studied by Jiang and Zhang in \cite{JiangZhang2014, JiangZhang2020Annals}.
When $\pi$ and $\sigma$ have trivial central characters, they can be viewed as cuspidal representations of $\SO_{m^\prime-2\ell-1}(\A)$ and $\SO_{m^\prime-2k}(\A)$ respectively. In this case, 
the integral representation in Theorem~\ref{intro-thm-main} recovers the construction of Jiang and Zhang \cite{JiangZhang2014} for quasi-split special orthogonal groups. 
For quasi-split orthogonal groups, a special family of such global integrals was first investigated by Ginzburg, Piatetski-Shapiro, and Rallis \cite{GinzburgPSRallis1997}. 

We remark that, in general, the family of Bessel models interpolates between the Whittaker model and the spherical model, and thus serves as a unifying generalization of both. 
In our construction, the representation $\pi$ is not required to be globally generic, and $\tau$ is not required to be cuspidal. When $\pi$ is globally generic, $\tau$ is cuspidal and $P_k$ is a Siegel parabolic subgroup of $\GSpin_{m^\prime}$, such construction appears in a recent work of Asgari, Cogdell, and Shahidi \cite{AsgariCogdellShahidi2024}; see Remark~\ref{remark-construction-of-ACS}. Their construction has applications to describing the image of the Langlands functorial transfer of globally generic cuspidal representations of $\GSpin$ groups to the general linear groups (see \cite{AsgariShahidi2014}),  as well as to the local descent theory for $\GSpin$ groups (see \cite{KaplanLauLiu2023}). 

When $\ell=0$, $m^\prime=2n+2$, and $k=n$, our integral generalizes the Rankin-Selberg integral of the $L$-function for $\SO_{2n+1}\times\GL_n$ constructed in \cite[Appendix]{BumpFriedbergFurusawa1997}. In particular, when $n=2$, in view of the isomorphism $\GSpin_5\cong\mathrm{GSp}_4$, our integral further generalizes the integral representation of Furusawa for $\GSp_4\times \GL_2$ \cite{Furusawa1993}.

The main difference in the $\GSpin$ case, compared to the special orthogonal case, lies not only in the more delicate analysis required for the global unfolding arguments, but also in the appearance of the twisted symmetric/exterior square $L$-functions of $\tau$. A similar phenomenon is observed in the work of Asgari, Cogdell, and Shahidi \cite{AsgariCogdellShahidi2024}.

\begin{remark}
In this paper, we only considered the global integral formed by pairing the Bessel coefficient of an Eisenstein series along $N_{\ell}$ with respect to $\psi_{\ell,a}$ with a cusp form.
There should also exist a ``dual" integral to $\mathcal{Z}(\phi_\pi, f_{\tau,\sigma, s})$. Here, by ``dual" integral we mean one obtained by switching the role of cusp form and Eisenstein series, i.e.,  one formed by pairing the Bessel coefficient of a cusp form along $N_{\ell}$ with respect to $\psi_{\ell,a}$ with an Eisenstein series. We will consider the dual integral representation in our upcoming work. For generic representations such dual integrals are constructed in \cite[\S 4]{AsgariCogdellShahidi2024} and \cite[\S 5]{AsgariCogdellShahidi2024}. For the case of split special orthogonal groups, such a duality is discussed in \cite[\S 1]{Soudry2017Israel}, and aspects of the local theory are developed in \cite{Soudry2018JNT}. Additionally, a family of related global integral for the general linear groups using the Bessel coefficient of a cusp form paired with an Eisenstein series is constructed in \cite{YanZhang2023}.
\end{remark}

\subsection{Application to the global Gan-Gross-Prasad conjecture}
\label{subsection-intro-GGP}
In this section we describe an application of Theorem~\ref{intro-thm-main} to  proving one direction of the global Gan-Gross-Prasad conjecture for generic representations of quasi-split $\GSpin$ groups, namely, the implication from the non-vanishing of the Bessel period to the non-vanishing of the central value of the relevant $L$-function.

\begin{theorem}[Theorem~\ref{thm-global-GGP}]
\label{intro-thm-global-GGP}
Let $k>\ell$.
Let $\pi$, $\sigma$ be irreducible generic cuspidal automorphic representations of $\GSpin_{m^\prime-2\ell-1}(\A)$ and $\GSpin_{m^\prime-2k}(\A)$ respectively such that $\omega_\pi \omega_\sigma=1$. 
If the Bessel period $\mathcal{B}^{\psi_{k-\ell-1},-a}(\phi_\pi, \phi_\sigma)$ is non-zero for some choice of data $\phi_\pi\in V_\pi, \phi_\sigma\in V_\sigma$, then the tensor product $L$-function $L(s,\sigma\times\pi)$ is holomorphic and non-zero at $s=\frac{1}{2}$.
\end{theorem}

We remark that the Bessel period $\mathcal{B}^{\psi_{k-\ell-1},-a}(\phi_\pi, \phi_\sigma)$ appearing in Theorem~\ref{intro-thm-global-GGP} is defined analogously as in \eqref{intro-eq-Bessel-period}, but for the pair of groups $(\GSpin_{m^\prime-2\ell-1}, \GSpin_{m^\prime-2k})$.

The global Gan-Gross-Prasad conjecture was first proposed by Gross and Prasad in \cite{GrossPrasad1992, GrossPrasad1994} for special orthogonal groups, and later was formulated by Gan, Gross and Prasad in \cite{GanGrossPrasad2012} in full generality for all classical groups including the metaplectic groups. 
In recent years, there has been much progress towards this conjecture for classical groups, see \cite{ZhangWei2014, Liu2014, Xue2014, FurusawaMorimoto2017, JiangZhang2020Annals, BPLZZ2021, BPCZ2022} for a few examples.
Recently Emory in \cite{Emory2020} formulated a refinement of this conjecture for the pair of $\GSpin$ groups $(\GSpin_{n+1}, \GSpin_{n})$, and proved it completely for $n=2, 3$  and in certain cases for $n=4$. Emory's approach relies on exceptional isomorphisms between low-rank $\GSpin$ groups and other classical groups, which allows one to reduce the conjecture to the already known case on the classical groups.  

We point out that implication from the non-vanishing of the Bessel period to the non-vanishing of the central value of the relevant $L$-function has been proved by Jiang and Zhang \cite[Theorem 5.7]{JiangZhang2020Annals} for special orthogonal groups and unitary groups in full generality, using the Rankin-Selberg method together with the endoscopic classification for these groups established by Arthur \cite{Arthur2013}, Mok \cite{Mok2015}, and Kaletha, Minguez, Shin and White \cite{KalethaMinguezShinWhite}. 
However, the endoscopic classification for $\GSpin$ groups is not available. As a result, our Theorem~\ref{intro-thm-global-GGP} is limited to generic representations. 
To remedy the lack of endoscopic classification, we use the Langlands functorial transfer for generic representations of $\GSpin$ groups established by Asgari and Shahidi \cite{AsgariShahidi2006, AsgariShahidi2014}, and combine it with the Rankin-Selberg integral in Theorem~\ref{intro-thm-main} to prove Theorem~\ref{intro-thm-global-GGP}. The proof is given in Section~\ref{section-GGP}.

For generic representations of classical groups, Theorem~\ref{intro-thm-global-GGP} was previously considered by Ginzburg, Jiang and Rallis \cite{GinzburgJiangRallis2004, GinzburgJiangRallis2005, GinzburgJiangRallis2009}, using the Rankin-Selberg method and the Arthur truncation method.

\subsection{Application to the reciprocal non-vanishing of Bessel periods}
 
 As another application of Theorem~\ref{intro-thm-main}, we prove a reciprocal non-vanishing result of Bessel periods, which relates the non-vanishing of certain Bessel periods for two different pairs of $\GSpin$ groups. 
 
We still take $k>\ell$ as in Section~\ref{subsection-intro-GGP}. Let $\sigma$ be an irreducible generic cuspidal automorphic representation of $\GSpin_{m^\prime-2\ell-1}(\A)$, and 
let $\tau=\tau_1\boxtimes \tau_2 \boxtimes \cdots \boxtimes \tau_{r}$ be an irreducible unitary generic isobaric automorphic representation of $\GL_k(\A)$ associated to distinct $\tau_1, \cdots, \tau_r$, such that $\tau_i$ is $\omega_\sigma^{-1}$-self-dual (see Definition~\ref{defn-twisted-self-dual}) for each $1\le i \le r$.
In Proposition~\ref{prop-GGP-EisensteinSeriesPoles}, we will prove that the Eisenstein series $E(\cdot, f_{\tau,\sigma, s})$ can possibly have a pole at $s=1$ of order at most $r$.
When the Eisenstein series $E(\cdot, f_{\tau,\sigma, s})$ has a pole at $s=1$ of order $r$, we denote by $\Eisen_{\tau\otimes\sigma}$ the $r$-th iterated residue at $s=1$ of $E(\cdot, f_{\tau,\sigma, s})$.   
 
\begin{theorem}[Theorem~\ref{thm-GGP-Reciprocal}]
\label{intro-thm-GGP-Reciprocal}	
Let $k>\ell$.
Let $\pi$, $\sigma$ be irreducible generic cuspidal automorphic representations of $\GSpin_{m^\prime-2\ell-1}(\A)$ and $\GSpin_{m^\prime-2k}(\A)$ respectively such that $\omega_\pi \omega_\sigma=1$. Let $\Pi$ be the functorial transfer of $\pi$ and let $\tau=\Pi\otimes\omega_\pi^{-1}$. Assume that the residue $\Eisen_{\tau\otimes\sigma^\prime}$ is non-zero. Then the Bessel period $\mathcal{B}^{\psi_{\ell},a}$ for $(\Eisen_{\tau\otimes\sigma^\prime}, \pi)$ is non-zero for some choice of data if and only if the Bessel period $\mathcal{B}^{\psi_{k-\ell-1},-a}$ for the pair $(\pi, \sigma)$ is non-zero for some choice of data. 
\end{theorem}
 
\begin{remark}
For special orthogonal groups and unitary groups, the reciprocal non-vanishing of Bessel periods is established in \cite[Theorem 5.3]{JiangZhang2020Annals}.	
\end{remark}

\subsection{Organization of the paper}

We now give a brief overview about the organization of the remainder of the paper. In Section~\ref{section-preliminaries}, we review basic facts about the $\GSpin$ groups, their subgroups, and the relevant Bessel periods. The definition of the global zeta integrals is given in Section~\ref{subsection-definition-global-zeta-integral}, and the unfolding computations are carried out in Sections~\ref{subsection-unfolding}, \ref{subsection-unfolding-kgm}, \ref{subsection-unfolding-klm}. In particular, we divide the computation into two cases, depending on whether $k>\ell$ or $k\le \ell$, and complete the analysis for each case in Sections~\ref{subsection-unfolding-kgm} and \ref{subsection-unfolding-klm} respectively. The results of the unfolding computation are summarized in Theorem~\ref{thm-global-unfolding-kgreaterl} and Theorem~\ref{thm-global-unfolding-klel}. In Section~\ref{section-local-zeta-integrals}, we carry out the local unramified computation of the local zeta integral. Our proof is to reduce the local unramified computation to the corresponding case of special orthogonal groups. In Section~\ref{section-GGP}, we establish sone analytic properties of the local zeta integrals, and the non-vanishing of the local zeta integrals at finite ramified and archimedean places which is one of the main technical results. We then prove Theorem~\ref{intro-thm-GGP-Reciprocal} in Section~\ref{subsection-GGP-Reciprocal-converse} and Theorem~\ref{intro-thm-global-GGP} in Section~\ref{subsection-GGP-proof}.

\subsection*{Acknowledgements}
I would like to thank Hang Xue for suggesting that I consider the global Gan-Gross-Prasad conjecture for the $\GSpin$ groups. This paper would not have come into existence without his suggestion and insightful discussions. I also thank Dihua Jiang for valuable discussions during my visit to the University of Minnesota. I thank Mahdi Asgari for sending me the paper \cite{AsgariCogdellShahidi2024}, which was very helpful in the preparation of this work. I am grateful to Jim Cogdell for his constant support over the years.  
The author has been partially supported by an AMS-Simons Travel Grant, and a summer grant from the Department of Mathematics at the University of Arizona.

\section{Preliminaries}
\label{section-preliminaries}

\subsection{The $\GSpin$ group}
Let $F$ be a number field and let $\A=\A_F$ be the ring of ad{\`e}les of $F$.
Let $V$ be a quadratic space over $F$ of dimension $m^\prime$ equipped with a non-degenerate quadratic form $q_V$. Let $\SO_{m^\prime}=\SO(V)$ be the special orthogonal group of $V$ associated to the quadratic form $q_V$.
Let $\tilde{m}$ be the Witt index of $V$. Then
\begin{equation*}
\tilde{m} = 
\begin{cases}
m & \mbox{ if } \dim(V) = 2m+1, \\
m & \mbox{ if } \dim(V) = 2m \mbox{ with } \SO(V) \mbox{ split, } \\
m-1 & \mbox{ if } \dim(V) = 2m \mbox{ with } \SO(V) \mbox{ quasi-split non-split. } 
\end{cases}
\end{equation*}
We take $V^+$ to be a maximal totally isotropic subspace of $V$ and $V^-$ to be its dual, so that $V$ has the following decomposition
\begin{equation*}
V=V^+\oplus V_0 \oplus V^-,	
\end{equation*}
where $V_0=(V^+\oplus V^-)^\perp$ denotes the anisotropic kernel of $V$. We choose a basis $\{e_1, e_2, \cdots, e_{\tilde{m}}\}$ of $V^+$ and a basis $\{e_{-1}, e_{-2}, \cdots, e_{-\tilde{m}}\}$ of $V^-$ so that $q_V(e_i, e_{-j})=\delta_{i,j}$ for all $1\le i, j\le \tilde{m}$. The anisotropic kernel $V_0$ is at most two dimensional. More specifically, if $\dim_F V$ is even, then $\dim_F V_0$ is either 0 or 2, and if $\dim_F V$ is odd, then $\dim_F V_0=1$. 
When $\dim_F V_0=2$ (i.e., $\SO(V)$ is quasi-split non-split, $m^\prime=2m$, , and $\tilde{m}=m-1$), we choose an orthogonal basis $\{e_0^{(1)}, e_0^{(2)}\}$ of $V_0$ so that
\begin{equation*}
q_{V_0}( e_0^{(1)}, e_0^{(1)})=1, \quad	q_{V_0}( e_0^{(2)}, e_0^{(2)})=-c,
\end{equation*}
where $c\in F^\times$ is a non-square and $q_{V_0}=q_V|_{V_0}$. When $\dim_F V_0=1$ (i.e., $\SO(V)$ is split, $m^\prime=2m+1$, , and $\tilde{m}=m$), we choose an anisotropic basis $\{e_0\}$ for $V_0$. We put the basis in the following order:
\begin{equation*}
e_1, e_2, \cdots, e_{\tilde{m}}, e_0^{(1)}, e_0^{(2)},  	e_{-\tilde{m}}, \cdots, e_{-2}, e_{-1}, \quad \textrm{ if } \dim_F V_0=2, 
\end{equation*}
\begin{equation*}
e_1, e_2, \cdots, e_{\tilde{m}}, e_0,  	e_{-\tilde{m}}, \cdots, e_{-2}, e_{-1}, \quad \textrm{ if } \dim_F V_0=1, 
\end{equation*}
\begin{equation*}
e_1, e_2, \cdots, e_{\tilde{m}},	e_{-\tilde{m}}, \cdots, e_{-2}, e_{-1}, \quad \textrm{ if } \dim_F V_0=0.
\end{equation*}
Then the corresponding standard flag of $V$ with respect to the given order defines the Borel subgroup $B_{\SO(V)}=T_{\SO(V)}\ltimes N_{\SO(V)}$ of $\SO(V)$, where $T_{\SO(V)}$ is a maximal torus. 

Let $H=\GSpin_{m^\prime}=\GSpin(V)$ be the $\GSpin$ cover of $\SO(V)$. Then we have a projection map
\begin{equation}
\label{eq-pr}
\pr:\GSpin(V)\to \SO(V)	
\end{equation} 
whose kernel lies in the center $C_{\GSpin(V)}$ and is isomorphic to $\GL_1$, and thus we get an exact sequence
\begin{equation}
\label{eq-GSpin-SO-exact-sequence}
\begin{CD}
1 @>>>  \GL_1 @>>> \GSpin(V) @>\pr>> \SO(V) @>>>  1
\end{CD}	
\end{equation}
Now $B_{\GSpin(V)}=\pr^{-1}(B_{\SO(V)})$ is a Borel subgroup of $\GSpin(V)$, $B_{\GSpin(V)}=T_{\GSpin(V)}\ltimes N_{\GSpin(V)}$, where $T_{\GSpin(V)}$ is a maximal torus. 
Note that the projection $\pr:\GSpin(V)\to \SO(V)$ induces an isomorphism of unipotent varieties, so we may specify unipotent elements or subgroups by their images under $\pr$. This defines coordinates for any unipotent element or subgroup of $\GSpin(V)$, which we use when defining characters. Thus we may write $u_{i,j}$ for the $(i,j)$ entry of $\pr(u)$. Moreover, the Weyl group $W_{\GSpin(V)}$ of $\GSpin(V)$ is isomorphic to the Weyl group $W_{\SO(V)}$ of $\SO(V)$.

Note that the $\GL_1$ in \eqref{eq-GSpin-SO-exact-sequence} is the identity component $C_{\GSpin(V)}^0$ of the center $C_{\GSpin(V)}$ when $\dim V>2$. When $\dim V=2$, the group $\GSpin(V)$ is abelian and hence $C_{\GSpin(V)}^0$ is larger than $\GL_1$. However, as a convention in this paper, we still denote $C_{\GSpin(V)}^0=\ker(\pr)=\GL_1$ even when $\dim V=2$.

We denote by $N$ the spinor norm, which is a homomorphism
\begin{equation*}
N: \GSpin(V)\to \GL_1.	
\end{equation*}
Note that $N(z)=z^2$ for $z\in C_{\GSpin(V)}^0$.

If we have an inclusion $(W, q_W)\subset (V, q_V)$ of quadratic spaces with $q_W=q_V|_{W}$, then we have an inclusion $\GSpin(W)\subset \GSpin(V)$ of $\GSpin$ groups, and a commutative diagram
\begin{equation*}
\begin{CD}
1 @>>>  C_{\GSpin(W)}^0 @>>> \GSpin(W) @>>> \SO(W) @>>>  1\\
@. @| @VVV @VVV @. \\
1 @>>>  C_{\GSpin(V)}^0 @>>> \GSpin(V) @>>> \SO(V) @>>>  1 
\end{CD}
\end{equation*}
where the vertical arrows are inclusions.

For the structure of the $\GSpin$ groups defined using based root datum, we refer the reader to \cite{AsgariShahidi2006, HundleySayag2016, AsgariCogdellShahidi2024}.
We will simply write $\GSpin_{m^\prime}$ for a quasi-split $\GSpin$ group associated to a quadratic space of dimension $m^\prime$. Following the convention in \cite{AsgariCogdellShahidi2024}, if we need to emphasize that the $\GSpin_{m^\prime}$ is quasi-split but non-split associated to the square class of an element $a\in F^\times$ in $F^\times/(F^\times)^2$, we may write $\GSpin_{m^\prime}^a$.

\subsection{Parabolic subgroups and Bessel subgroups}

Let $\ell$ be an integer such that $1\le \ell<m$ if $\dim(V)=2m+1$ or $1\le \ell <m-1$ if $\dim(V)=2m$. Let $V_{\ell}^{\pm}$ be the totally isotropic subspace generated by $\{e_{\pm 1}, e_{\pm 2}, \cdots, e_{\pm \ell}\}$, and $P_{\ell}=M_{{\ell}}\ltimes U_{{\ell}}$ the standard maximal parabolic subgroup of $\GSpin(V)$ which stabilizes $V_{\ell}^+$. Then the Levi subgroup $M_{{\ell}}$ is isomorphic to $\GL(V_\ell^+) \times \GSpin(W_\ell)$, where
\begin{equation}
\label{eq-W-l}
	W_\ell=(V_\ell^+\oplus V_\ell^-)^{\perp}.
\end{equation}
We fix the isomorphism between $\GL(V_\ell^+) \times \GSpin(W_\ell)$ and $M_{\ell}$ as in \cite[\S 1]{CaiFriedbergKaplan2024}. 
 For the Siegel parabolic subgroup of $\GSpin(V)$, we follow the convention in \cite{AsgariCogdellShahidi2024}.

We denote by $Q_\ell=L_{{\ell}}\ltimes N_{{\ell}}$ the standard maximal parabolic subgroup of $\GSpin(V)$ stabilizing the following maximal flag of isotropic subspaces:
\begin{equation}
\label{eq-parobolic-Ql-flag}
0\subset V_1^+ \subset V_2^+ \subset \cdots \subset 	V_{\ell}^+.
\end{equation}
Then the Levi subgroup $L_{\ell}\cong (\GL_1)^\ell \times \GSpin(W_\ell)$ and the unipotent radical is given by 
\begin{equation}
\label{eq-unipotent-Nl}
N_{\ell} = 	\left\{ u=\begin{pmatrix} z & y &x\\ &I_{m^{\prime}-2\ell} &y^\prime\\ &&z^*\end{pmatrix}: z\in Z_{\ell} \right\}
\end{equation}
where $Z_{\ell}$ is the standard maximal upper-triangular unipotent subgroup of $\GL_{\ell}=\GL(V_{\ell}^+)$. 

Let $\psi:F\backslash \A\to \C$ be a fixed non-trivial additive character. 
For $a\in F^\times$, let $w_0$ be the anisotropic vector in $W_{\ell}$ defined by
\begin{equation}
\label{eq-w0}
w_0=	e_{\tilde{m}}+(-1)^{m^\prime+1} \frac{a}{2} e_{-\tilde{m}}
\end{equation}
and let $\psi_{\ell, a}$ be the character of $N_{\ell}(\A)$ defined by
\begin{equation}
\label{eq-psi-l-a}
\psi_{\ell, a}(u)=	\psi(\sum^{\ell-1}_{i=1}z_{i,i+1}+y_{\ell, \tilde{m}-\ell}+(-1)^{m^\prime+1}\frac{a}{2}y_{\ell, m^\prime-\tilde{m}-\ell+1})
\end{equation}
for $u\in N_{\ell}(\A)$ as in \eqref{eq-unipotent-Nl}. 
 The Levi subgroup $L_{\ell}$ acts on the set of the characters of $N_{\ell}(\A)$. Each orbit for this action contains a character of the form $\psi_{\ell, a}$, for $a\in F^\times$. We let $G$ be the connected component of the identity of the stabilizer, i.e., 
 \begin{equation*}
 G=(L_{\ell}^{\psi_{\ell, a}})^0\cong \GSpin_{m^\prime-2\ell-1}.	
 \end{equation*}
 Then the $\ell$-th Bessel subgroup of $\GSpin(V)$ is
\begin{equation*}
R_{\ell, a}=G \ltimes N_{\ell}.
\end{equation*}

The Bessel subgroup can also be defined similarly when $\ell=0$. When $\ell=0$, the unipotent subgroup $N_0$ is trivial and the Bessel subgroup is simply $R_{0, a}=G=\GSpin_{m^\prime-1}$.

\subsection{Bessel period}

Let $\phi$ be an automorphic form on $H(\A)=\GSpin_{m^\prime}(\A)$. We define the Bessel coefficient (or Gelfand-Graev coefficient) of $\phi$ by
\begin{equation}
\label{eq-Bessel-coef}
\phi^{N_{\ell}, \psi_{\ell, a}}(g):=\int_{N_{\ell}(F)\backslash N_{\ell}(\A)}\phi(ug)\psi_{\ell, a}^{-1}(u)du. 	
\end{equation}
This defines an automorphic function on $G(\A)$. Let $\phi^\prime$ be a cuspidal automorphic form on $G(\A)$, such that
\begin{equation}
\label{eq-Bessel-period-char-condition}
\phi(z)\phi^\prime(z)=1, \quad \forall z\in 	C^0_{G}(\A),
\end{equation}
where we recall that $C^0_G=\GL_1$ is the identity component of the center $C_G$ of $G$ if $m^\prime-2\ell-1>2$.
We define a Bessel period for the pair $(\phi, \phi^\prime)$ by
\begin{equation}
\label{eq-Bessel-period}
\mathcal{B}^{\psi_{\ell,a}}(\phi, \phi^\prime):=	 \int_{ C^0_G(\A) G(F)\backslash G(\A)} \phi^{N_{\ell}, \psi_{\ell, a}}(g) \phi^\prime(g)dg
\end{equation}
and we also call this a Bessel period for the pair $(H, G)$ with respect to the character $\psi_{\ell,a}$.

In this paper, we will apply the Bessel period to certain Eisenstein series that we will discuss in Section~\ref{subsection-definition-global-zeta-integral}.

\section{The global zeta integrals}
\label{section-global-zeta-integrals}
\subsection{Definition of the global zeta integral}
\label{subsection-definition-global-zeta-integral}
Let $k$ be an integer with $1\le k \le  \tilde{m}$. Recall that the parabolic subgroup $P_k\subset H$ has a Levi subgroup isomorphic to $\GL_k\times \GSpin_{m^\prime-2k}$. 
Let $\tau$ be an irreducible unitary generic automorphic representation of $\GL_k(\A)$, of the following isobaric type:
\begin{equation}
\label{eq-tau}
\tau=\tau_1\boxtimes \tau_2 \boxtimes \cdots \boxtimes \tau_{r},
\end{equation}
with $k=k_1+\cdots + k_r$, where $\tau_i$ is an irreducible unitary cuspidal automorphic representation of $\GL_{k_i}(\A)$ for each $1\le i \le r$. 
Let $\sigma$ be an irreducible automorphic representation of $G_0(\A)=\GSpin_{m^\prime-2k}(\A)$ (note that we do not assume $\sigma$ is cuspidal) and we denote by $\omega_\sigma$ the restriction of the central character of $\sigma$ to $C_{G_0}^0(\A)=\A^\times$. 
Given a section $f_{\tau,\sigma, s}$ in the normalized induced representation
\begin{equation*}
\Ind_{P_{k}(\A)}^{H(\A)}(\tau|\cdot|^{s-\frac{1}{2}} \otimes \sigma),
\end{equation*}
we form the Eisenstein series
\begin{equation}
\label{eq-Eis}
E(g, f_{\tau,\sigma, s})=\sum_{\gamma\in P_k(F)\backslash H(F)} f_{\tau,\sigma, s}(\gamma g), \quad g\in H(\A).
\end{equation}
Let $\pi$ be an irreducible cuspidal automorphic representation of $G(\A)$, and let $\phi_\pi\in V_\pi$ be a non-zero cusp form.
We assume that 
\begin{equation}
\label{eq-assumption-central-char-G-G0}
\omega_\pi(z) \omega_{\sigma}(z)=1, \quad \forall z\in 	C^0_{G}(\A),
\end{equation}
where $\omega_\pi$ is the restriction of the central characters of $\pi$ to $C_G^0(\A)=\A^\times$. 
The global zeta integral we consider is defined to be the following Bessel period
\begin{equation}
\label{eq-global-integral}
\mathcal{Z}(\phi_\pi, f_{\tau,\sigma, s}):= \int_{C_G^0(\A) G(F)\backslash G(\A)}  E^{N_\ell, \psi_{\ell,a}}(g, f_{\tau,\sigma,s})\phi_\pi(g)dg.
\end{equation}
Because of the assumption~\eqref{eq-assumption-central-char-G-G0}, the integrand is invariant under $C_G^0(\A)$.
\begin{remark}
\label{remark-construction-of-ACS}
When $P_k$ is a Siegel parabolic subgroup of $H$, such global zeta integrals with generic $\pi$ and cuspidal $\tau$ have been studied by Asgari, Cogdell and Shahidi in \cite{AsgariCogdellShahidi2024}. These are the Case B global integrals defined in \cite[\S 4]{AsgariCogdellShahidi2024}. Hence we assume from now on that $k<\tilde{m}$. 
\end{remark}

\begin{lemma}
\label{lemma-global-integral-convergence}
The global zeta integral $\mathcal{Z}(\phi_\pi, f_{\tau,\sigma, s})$ converges absolutely and uniformly in vertical strips in $\mathbb{C}$ away from the possible poles of the Eisenstein series, and hence it defines a meromorphic function on $\mathbb{C}$ with possible poles at the locations where the Eisenstein series has poles. 
\end{lemma}

\begin{proof}
This follows due to the rapid decrease (mod $C_G^0$) of $\phi_\pi$, the moderate growth (mod $C_G^0$) of the Eisenstein series $E(\cdot, f_{\tau,\sigma,s})$, and the compactness of $N_{\ell}(F)\backslash N_{\ell}(\A)$. 
\end{proof}

\subsection{The unfolding computation}
\label{subsection-unfolding}
The goal of the rest of Section~\ref{section-global-zeta-integrals}  is to carry out the global unfolding computation for the global zeta integral $\mathcal{Z}(\phi_\pi, f_{\tau,\sigma, s})$. The main results are summarized in the following two theorems, which address the case $k>\ell$ and the case $k\le \ell$ respectively.

\begin{theorem}
\label{thm-global-unfolding-kgreaterl}
Assume $k>\ell$ and set $\beta=k-\ell$.
For $\Re(s)\gg 0$, the integral $\mathcal{Z}(\phi_\pi, f_{\tau,\sigma, s})$ unfolds to 
\begin{equation}
\int_{R_{\ell,\beta-1}^{\eta_2}(\A)\backslash G(\A) } \mathcal{B}^{\psi_{\beta-1,-a}}( \pi(g) \phi_\pi, \mathcal{J}_{\ell,a}(R( \epsilon_{0,\beta} \eta_2 g) {f_{\mathcal{W}(\tau,\psi_{Z_k,a}^{-1}),\sigma,s}})) dg,
\end{equation}
where  
\begin{itemize}
\item $\epsilon_{0, \beta}$ and $\eta_2$ are given in 	\eqref{eq-double-coset-epsilon-alpha-beta} and \eqref{eq-eta-kgreaterl} respectively,
\item $R_{\ell,\beta-1}^{\eta_2}$ is the Bessel subgroup of $G$ given in \eqref{eq-Bessel-subgroup-of-G},
\item $Z_k$ is the maximal unipotent subgroup of $\GL_k\subset P_k\subset \GSpin_{m^\prime}$ and $\psi_{Z_k,a}$ is the generic character of $Z_k$ given in \eqref{eq-psi-Zka},
\item $f_{\mathcal{W}(\tau,\psi_{Z_k,a}^{-1}),\sigma,s}\in \Ind_{P_{k}(\A)}^{H(\A)}(\mathcal{W}(\tau,\psi_{Z_k,a}^{-1})|\cdot|^{s-\frac{1}{2}} \otimes \sigma)$ is given by  \eqref{eq-f-upper-Zk}, where $\mathcal{W}(\tau,\psi_{Z_k,a}^{-1})$ is the $\psi_{Z_k,a}^{-1}$-Whittaker model of $\tau$,
\item $R$ denotes the right translation,
\item $\mathcal{J}_{\ell,a}$ denotes taking certain Fourier coefficient given in \eqref{eq-unfolding-J-l-a}, 
\item $\mathcal{B}^{\psi_{\beta-1,-a}}$ denotes the Bessel period for the pair $(G, G_0)$ with respect to the character $\psi_{\beta-1,-a}$ defined analogously as in \eqref{eq-psi-l-a}.
\end{itemize}
\end{theorem}

\begin{theorem}
\label{thm-global-unfolding-klel}
Assume $k\le  \ell$.
For $\Re(s)\gg 0$, the integral $\mathcal{Z}(\phi_\pi, f_{\tau,\sigma, s})$ unfolds to 
\begin{equation}
\int_{N_{\ell}^{\eta_1}(\A)\backslash N_{\ell}(\A)}  \mathcal{B}^{\psi_{m^\prime-2k;\ell-k,a}} ( R(\epsilon_{0,0} \eta_1   u) f_{\mathcal{W}(\tau,\psi),\sigma,s}, \phi_\pi) \psi_{\ell,a}^{-1}(u)du ,
\end{equation}
where 
\begin{itemize}
\item $\epsilon_{0, 0}$ and $\eta_1$ are given in 	\eqref{eq-double-coset-epsilon-alpha-beta} and \eqref{eq-eta-klel},
\item $N_{\ell}^{\eta_1}$ is the unipotent subgroup of $N_{\ell}$ given in \eqref{eq-Nl-eta1},
\item $\mathcal{B}^{\psi_{m^\prime-2k;\ell-k,a}}$ denotes the Bessel period for the pair $(G_0, G)$ with respect to the character $\psi_{m^\prime-2k;\ell-k,a}$, which is the restriction of $\psi_{\ell,a}$ to the subgroup $N_{k,\ell-k}$ defined in \eqref{eq-N-k-l-k}. 
\end{itemize}
\end{theorem}

The remainder of Section~\ref{section-global-zeta-integrals} is devoted to the proofs of Theorem~\ref{thm-global-unfolding-kgreaterl} and Theorem~\ref{thm-global-unfolding-klel}.

We start with the following double coset decomposition for $P_k(F)\backslash H(F)/P_\ell(F)$. Recall that $\tilde{m}$ is the Witt index of $(V, q_V)$ defining $H$, and $m=[\frac{m^\prime}{2}]$. We divide the discussion into the following two cases, given in Lemma~\ref{lemma-double-coset-case1} and Lemma~\ref{lemma-double-coset-case2} below:
\begin{itemize}
\item Case (1): $H$ is not a $F$-split even $\GSpin$ group,\\
\item Case (2): $H$ is a $F$-split even $\GSpin$ group.
\end{itemize}

\begin{lemma}
\label{lemma-double-coset-case1}
(\textbf{Case (1)})
Suppose $H$ is not a $F$-split even $\GSpin$ group. The double coset representatives for $P_k(F)\backslash H(F)/P_\ell(F)$ are given by $\epsilon_{\alpha, \beta}\in  W_H\cong W_{\SO_{m^\prime}}$ for each pair of non-negative integers $(\alpha, \beta)$ in
\begin{equation}
\label{eq-double-coset-index-set-Ekl}
\mathcal{E}_{k,\ell}=\{ (\alpha, \beta): 0\le \alpha \le \beta \le k \text{ and } k\le \ell +\beta-\alpha \le \tilde{m} \},	
\end{equation}
with 
\begin{equation}
\label{eq-double-coset-epsilon-alpha-beta}
	\pr(\epsilon_{\alpha,\beta})	 =  w_q^{k-\beta}  a_{\alpha, \beta}
\end{equation}
where  
\begin{equation}
\label{eq-a-alpha-beta}
 a_{\alpha, \beta} = 
\begin{pmatrix}
I_{\alpha} &&&&&&&\\
&0&0&I_{\beta-\alpha} &0&0&0&0&\\
&0&0&0&0&0&0&I_{k-\beta}&\\
&I_{\ell+\beta-\alpha-k} &0&0&0&0&0&0&\\
&0&0&0&I_{m^\prime-2(\ell+\beta-\alpha)} &0&0&0&\\
&0&0&0&0&0&I_{\ell+\beta-\alpha-k}&0&\\
&0&I_{k-\beta}&0&0&0 &0&0&\\
&0&0 &0&0&I_{\beta-\alpha}&0&0&\\
&&&&&&&&I_\alpha 
\end{pmatrix},
\end{equation}
and $w_q$ is given in \cite[pp. 70-71]{GinzburgRallisSoudry2011} which is an auxiliary Weyl group element.
\end{lemma}

\begin{proof}
By \cite[\S 4.2.1]{GinzburgRallisSoudry2011}, we know
\begin{equation*}
\SO_{m^\prime}(F)= \bigsqcup_{(\alpha,\beta)\in \mathcal{E}_{k,\ell}} \pr(P_k(F))\pr(\epsilon_{\alpha,\beta}) \pr(P_{\ell}(F))
\end{equation*}
where the disjoint union is taken over $(\alpha, \beta)$ in the set \eqref{eq-double-coset-index-set-Ekl}.
Note that $W_{H}\cong W_{\SO_{m^\prime}}$ and that $\ker (\pr) = C_H^0 \subset P_k$	. Hence  we have
\begin{equation*}
	H(F)= \bigsqcup_{(\alpha,\beta)\in \mathcal{E}_{k,\ell}} P_k(F)\epsilon_{\alpha, \beta} P_{\ell}(F).
\end{equation*}
\end{proof}

In the next case when $H$ is the $F$-split even $\GSpin$ group, the set of pairs $(\alpha, \beta)$ is also given in \eqref{eq-double-coset-index-set-Ekl}. We still define $a_{\alpha, \beta}$ as in \eqref{eq-a-alpha-beta} and $\epsilon_{\alpha, \beta}$ as in \eqref{eq-double-coset-epsilon-alpha-beta}, but now we take $w_q=\diag(I_{m-1}, w_q^0, I_{m-1})$ with $w_q^0=\begin{pmatrix}0&1\\1&0\end{pmatrix}$ as in \cite[\S 4.2.2]{GinzburgRallisSoudry2011}.

\begin{lemma}
\label{lemma-double-coset-case2}
(\textbf{Case (2)})
Suppose $H$ is a $F$-split even $\GSpin$ group (i.e., $\tilde{m}=m$ and $m^\prime=2m$).
The coset representatives for $P_k(F)\backslash H(F)/P_\ell(F)$ are indexed by the set \eqref{eq-double-coset-index-set-Ekl} given as follows. When $\ell+\beta-\alpha<m$, there is one double coset representative $\epsilon_{\alpha, \beta}\in  W_H\cong W_{\SO_{m^\prime}}$ given by \eqref{eq-double-coset-epsilon-alpha-beta}. 
When $\ell+\beta-\alpha=m$, there are two double coset representatives $\epsilon_{\alpha, \beta}$ and $\tilde{\epsilon}_{\alpha, \beta}$, where $\pr(\tilde{\epsilon}_{\alpha, \beta})=w_q \pr({\epsilon}_{\alpha, \beta}) w_q^{-1}$.
\end{lemma}

\begin{proof}
The proof is similar to the proof of Lemma~\ref{lemma-double-coset-case1}.	
\end{proof}

We use the notation $\tilde{\mathcal{E}}_{k,\ell}$ to denote a set of the double coset representatives for $P_k(F)\backslash H(F)/P_\ell(F)$ described in Lemma~\ref{lemma-double-coset-case1} and Lemma~\ref{lemma-double-coset-case2}.
Using this decomposition, for $\mathrm{Re}(s)\gg 0$, we can unfold the Eisenstein series to obtain
\begin{equation*}
\begin{split}
E^{N_\ell, \psi_{\ell,a}}(g, f_{\tau,\sigma,s}) &= \int_{N_{\ell}(F)\backslash N_{\ell}(\A)} E(ug, f_{\tau,\sigma,s})\psi_{\ell,a}^{-1}(u)du \\
&= \sum_{\epsilon_{\alpha,\beta} \in \tilde{\mathcal{E}}_{k,\ell} } \int_{N_{\ell}(F)\backslash N_{\ell}(\A)} \sum_{\delta\in P_{\ell}^{\epsilon_{\alpha,\beta} }(F)\backslash P_{\ell}(F)} f_{\tau,\sigma,s}(\epsilon_{\alpha,\beta}  \delta u g) \psi_{\ell,a}^{-1}(u)du,
\end{split}
\end{equation*}
where $P_{\ell}^{\epsilon_{\alpha,\beta}}:=\epsilon_{\alpha,\beta} ^{-1}P_k \epsilon_{\alpha,\beta}\cap P_{\ell}$ is the stabilizer in $P_{\ell}$.

In the case of $\SO_{m^\prime}$, the stabilizers $P_{\ell}^{\epsilon_{\alpha, \beta}}$ and  $P_{\ell}^{\tilde{\epsilon}_{\alpha, \beta}}$ are explicitly given in \cite[\S 4.3.1, \S 4.3.2]{GinzburgRallisSoudry2011} (see also \cite[(3.5)-(3.6)]{JiangZhang2014}). If we write  $P_{\ell}^{\epsilon_{\alpha, \beta}}=M_{\ell}^{\epsilon_{\alpha, \beta}} \ltimes U_{\ell}^{\epsilon_{\alpha, \beta}} $ (resp. $P_{\ell}^{\tilde{\epsilon}_{\alpha, \beta}}=M_{\ell}^{\tilde{\epsilon}_{\alpha, \beta}} \ltimes U_{\ell}^{\tilde{\epsilon}_{\alpha, \beta}} $) with $M_{\ell}^{\epsilon_{\alpha, \beta}}=\epsilon_{\alpha, \beta}^{-1}P_k \epsilon_{\alpha, \beta}\cap M_{\ell}$ (resp. $M_{\ell}^{\tilde{\epsilon}_{\alpha, \beta}}=\tilde{\epsilon}_{\alpha, \beta}^{-1}P_k \tilde{\epsilon}_{\alpha, \beta}\cap M_{\ell}$), then the unipotent subgroup $U_{\ell}^{\epsilon_{\alpha, \beta}}$ (resp. $U_{\ell}^{\tilde{\epsilon}_{\alpha, \beta}}$)  agrees with \cite{GinzburgRallisSoudry2011} via the projection map, and the $\GSpin_{m^\prime}/\SO_{m^\prime}$ difference is in $M_{\ell}^{\epsilon_{\alpha, \beta}}$ (resp. $M_{\ell}^{\tilde{\epsilon}_{\alpha, \beta}}$).

To continue, we further consider the double coset decomposition $P_{\ell}^{\epsilon_{\alpha,\beta}}(F)\backslash P_{\ell}(F) / R_{\ell,a}(F)$ for $\epsilon_{\alpha,\beta}\in \tilde{\mathcal{E}}_{k,\ell}$, where $R_{\ell, a}=G\ltimes N_{\ell}$. In the case of $\SO_{m^\prime}$, the representatives for these double cosets are computed in \cite[(5.2)]{GinzburgRallisSoudry2011} (see also \cite[(3.7)]{JiangZhang2014}, which are of the form
\begin{equation*}
\begin{pmatrix}
 \epsilon & & \\ & \gamma &\\ && \epsilon^*	
 \end{pmatrix} \in \SO_{m^\prime},
\end{equation*}
where $\epsilon$ runs through a set of representatives for the quotient of the Weyl groups 
\begin{equation*}
	W_{\GL_{\alpha}}\times W_{\GL_{\ell+\beta-\alpha-k}}\times W_{\GL_{k-\beta}}\backslash W_{\GL_{\ell}},
\end{equation*}
and $\gamma$ runs through a set of representatives for $P_{w,\SO_{m^\prime}}^\prime\backslash \SO(W_{\ell}) / \mathrm{Stab}_{L_{\ell,\SO_{m^\prime}}}(\psi_{\ell, a})$, where 
\begin{equation*}
P_{w, \SO_m^\prime}^\prime:= \SO(W_{\ell})\cap \pr(\epsilon_{\alpha,\beta}^{-1})P_{k,\SO_{m^\prime}} \pr(\epsilon_{\alpha,\beta}) \ \ \text{ for }w=\pr(\epsilon_{\alpha,\beta}) \text{ or } w=\pr(\tilde{\epsilon}_{\alpha,\beta}),
\end{equation*}
${L_{\ell,\SO_{m^\prime}}}\cong (\GL_1)^{\ell}\times \SO_{m^\prime-2\ell}$  is the Levi subgroup of the standard parabolic subgroup of $\SO_{m^\prime}$ stabilizing \eqref{eq-parobolic-Ql-flag}, $P_{k,\SO_{m^\prime}}$ is the standard parabolic subgroup of $\SO_{m^\prime}$ stabilizing $V_k^+$, and we recall that $W_\ell=(V_\ell^+\oplus V_\ell^-)^{\perp}$. 

\begin{remark}
\label{remark-parabolic-subgroup}
Note that $P_{w, \SO_m^\prime}^\prime$ is a maximal parabolic subgroup of $\SO(W_{\ell})$ as shown in \cite[pp. 563-564]{JiangZhang2014}. More specifically, we have the following:
\begin{itemize}
\item 	When $\SO_{m^\prime}$ is not the $F$-split even special orthogonal group, $P_{w, \SO_m^\prime}^\prime$ is the parabolic subgroup of $\SO(W_{\ell})$ preserving the $(\beta-\alpha)$-dimensional totally isotropic subspace $V_{\ell, \beta-\alpha}^+$ of $W_{\ell}$, where
\begin{equation}
\label{eq-subspace-V-lt}
V_{\ell, t}^{\pm}=\Span\{e_{\pm (\ell+ 1)}, \cdots, e_{\pm (\ell+t)} \}
\end{equation}
for $1\le t\le m-\ell$.
\item Assume $\SO_{m^\prime}$ is the $F$-split even special orthogonal group. When $\ell+\beta-\alpha<m$, $P_{w, \SO_m^\prime}^\prime$ is the parabolic subgroup of $\SO(W_{\ell})$ preserving the $(\beta-\alpha)$-dimensional totally isotropic subspace $V_{\ell, \beta-\alpha}^+$ of $W_{\ell}$. When $\ell+\beta-\alpha=m$ and $w=\pr(\epsilon_{\alpha,\beta})$, $P_{w, \SO_m^\prime}^\prime$ is the parabolic subgroup of $\SO(W_{\ell})$ preserving the $(m-\ell)$-dimensional totally isotropic subspace $V_{\ell, m-\ell}^+$ of $W_{\ell}$. When $\ell+\beta-\alpha=m$ and $w=\pr(\tilde{\epsilon}_{\alpha,\beta})$, $P_{w, \SO_m^\prime}^\prime$ is the parabolic subgroup of $\SO(W_{\ell})$ preserving the $(m-\ell)$-dimensional subspace $w_q  V_{\ell,m-\ell}^+$ of $W_{\ell}$.
\end{itemize}
\end{remark}
By a similar argument as in the proof of Lemma~\ref{lemma-double-coset-case1}, we have the following lemma. 

\begin{lemma}
\label{lemma-double-coset-second}
The representatives for 	$P_{\ell}^{\epsilon_{\alpha,\beta}}(F)\backslash P_{\ell}(F) / R_{\ell,a}(F)$ are given by $\{\eta_{\epsilon,\gamma} \}$ which are determined by
\begin{equation*}
\pr(\eta_{\epsilon,\gamma}) = \begin{pmatrix}
 \epsilon & & \\ & \gamma &\\ && \epsilon^*	
 \end{pmatrix} \in \SO_{m^\prime}(F)
 \end{equation*}
 as above.
\end{lemma}

We denote by $\mathcal{N}_{\alpha,\beta, \ell, a}$ the above set of representatives of 	$P_{\ell}^{\epsilon_{\alpha,\beta}}(F)\backslash P_{\ell}(F) / R_{\ell,a}(F)$ given in Lemma~\ref{lemma-double-coset-second}. Then we further obtain
\begin{equation}
\label{eq-unfolding-ENpsi-1}
\begin{split}
E^{N_\ell, \psi_{\ell,a}}(g, f_{\tau,\sigma,s}) 
= \sum_{\epsilon_{\alpha,\beta}\in \tilde{\mathcal{E}}_{k,\ell} } \sum_{\eta\in \mathcal{N}_{\alpha,\beta, \ell, a}} \int_{N_{\ell}(F)\backslash N_{\ell}(\A)}     \sum_{\delta\in R_{\ell,a}^{\eta}(F)\backslash R_{\ell,a}(F)}f_{\tau,\sigma,s}(\epsilon_{\alpha,\beta} \eta \delta u g) \psi_{\ell,a}^{-1}(u)du,	
\end{split}
\end{equation}
where 
$$R_{\ell,a}^{\eta}:=R_{\ell,a}\cap \eta^{-1} P_{\ell}^{\epsilon_{\alpha,\beta}} \eta.$$ 
Since $R_{\ell, a}=G\ltimes N_{\ell}$, we have
\begin{equation*}
	R_{\ell,a}^{\eta}=(G\cap \eta^{-1} M_{\ell}^{\epsilon_{\alpha,\beta}} \eta) \cdot (N_{\ell}\cap \eta^{-1} P_{\ell}^{\epsilon_{\alpha,\beta}} \eta) 
\end{equation*}
and 
\begin{equation*}
	R_{\ell,a}^{\eta}\backslash R_{\ell,a} = \left( (G\cap \eta^{-1} M_{\ell}^{\epsilon_{\alpha,\beta}} \eta)\backslash G\right) \cdot \left( (N_{\ell}\cap \eta^{-1} P_{\ell}^{\epsilon_{\alpha,\beta}} \eta) \backslash N_{\ell} \right).  
\end{equation*}
Denote 
\begin{equation}
\label{eq-G-eta-N-l-eta}
G^{\eta}:=G\cap \eta^{-1} M_{\ell}^{\epsilon_{\alpha,\beta}} \eta, \quad N_{\ell}^\eta:=N_{\ell}\cap \eta^{-1} P_{\ell}^{\epsilon_{\alpha,\beta}} \eta,
\end{equation}
so that
\begin{equation*}
	R_{\ell,a}^{\eta}=G^\eta\cdot N_{\ell}^\eta.
\end{equation*}
For fixed $\epsilon_{\alpha,\beta}$ and $\eta$, the inner integration in \eqref{eq-unfolding-ENpsi-1} becomes
\begin{equation*}
	\int_{N_{\ell}(F)\backslash N_{\ell}(\A)}     \sum_{\delta\in G^{\eta}(F) \backslash G(F)} \sum_{\delta^\prime \in N_{\ell}^{\eta}(F)  \backslash N_{\ell}(F)} f_{\tau,\sigma,s}(\epsilon_{\alpha,\beta} \eta \delta^\prime \delta  u g) \psi_{\ell,a}^{-1}(u)du.
\end{equation*}
When $\mathrm{Re}(s)\gg 0$, we can interchange the $du$-integration with the sum over $\delta$, by the absolute convergence of the integral. Note that any modulus character will evaluate to $1$ on $\delta$ since $\delta\in G(F)$. Also, $\delta$ stabilizes the character $\psi_{\ell,a}$. After interchanging, we may further collapse the $du$-integration with the sum over $\delta^\prime$, to obtain
\begin{equation*}
\begin{split}
E^{N_\ell, \psi_{\ell,a}}(g, f_{\tau,\sigma,s}) 
= \sum_{\epsilon_{\alpha,\beta}\in \tilde{\mathcal{E}}_{k,\ell} } \sum_{\eta\in \mathcal{N}_{\alpha,\beta, \ell, a}} \sum_{\delta\in G^{\eta}(F)\backslash G(F)}
 \int_{N_{\ell}^{\eta}(F)\backslash N_{\ell}(\A)}     f_{\tau,\sigma,s}(\epsilon_{\alpha,\beta} \eta \delta u g) \psi_{\ell,a}^{-1}(u)du.
 \end{split}
 \end{equation*}
Hence, for $\mathrm{Re}(s)\gg 0$, $E^{N_\ell, \psi_{\ell,a}}(g, f_{\tau,\sigma,s})$ is equal to
 \begin{equation}
 \label{eq-unfolding-ENpsi-2}
    \sum_{\epsilon_{\alpha,\beta}\in \tilde{\mathcal{E}}_{k,\ell} } \sum_{\eta\in \mathcal{N}_{\alpha,\beta, \ell, a}} \sum_{\delta\in G^{\eta}(F)\backslash G(F)} \int_{N_{\ell}^{\eta}(\A)\backslash N_{\ell}(\A)}   \int_{N_{\ell}^{\eta}(F)\backslash N_{\ell}^{\eta}(\A)}    f_{\tau,\sigma,s}(\epsilon_{\alpha,\beta} \eta \delta u^\prime u g) \psi_{\ell,a}^{-1}(u^\prime u)du^\prime du.
\end{equation}

\begin{lemma}
\label{lemma-unfolding-vanishing-alpha=0}
Suppose either one of the following two conditions holds:
\begin{enumerate}
\item 	$\alpha>0$,\\
\item $\alpha=0$, $\beta>\max\{k-\ell, 0\}$ and  $\gamma w_0$ is not orthogonal to $V_{\ell,\beta}^-$ for \\ $\gamma\in P_{w,\SO_{m^\prime}}^\prime\backslash \SO(W_{\ell}) / \mathrm{Stab}_{L_{\ell,\SO_{m^\prime}}}(\psi_{\ell, a})$.
\end{enumerate}
Then there exists a unipotent $F$-subgroup $S\subset N_{\ell}^\eta$ such that 
\begin{itemize}
\item $\psi_{\ell,a}$ is non-trivial on $S(\A)$, and \\
\item $\epsilon_{\alpha,\beta} \eta S \eta^{-1} \epsilon_{\alpha,\beta}^{-1}\subset U_k$, the unipotent radical of $P_k$.
\end{itemize}
Moreover, we have
\begin{equation}
\label{eq-lemma-unfolding-vanishing-alpha=0}
\int_{N_{\ell}^{\eta}(\A)\backslash N_{\ell}(\A)}   \int_{N_{\ell}^{\eta}(F)\backslash N_{\ell}^{\eta}(\A)}    f_{\tau,\sigma,s}(\epsilon_{\alpha,\beta} \eta \delta u^\prime u g) \psi_{\ell,a}^{-1}(u^\prime u)du^\prime du=0.	
\end{equation}
\end{lemma}

\begin{proof}
The existence of such a unipotent subgroup $S$ in $\SO_{m^\prime}$ is proved in \cite[Lemma 3.1 and Lemma 3.2]{JiangZhang2014}. However, this is a unipotent subgroup, and the unipotent varieties for $\SO_{m^\prime}$ and $H$ are the same. Hence, the existence of such a unipotent subgroup remains true for $H=\GSpin_{m^\prime}$.

Now let $S\subset N_{\ell}^{\eta}$ be such a unipotent subgroup. Then for $s\in S(\A)$, we have
\begin{equation*}
\begin{split}
 f_{\tau,\sigma,s}(\epsilon_{\alpha,\beta} \eta    s u^\prime u g) =  f_{\tau,\sigma,s}\left(( \epsilon_{\alpha,\beta} \eta s \eta^{-1} \epsilon_{\alpha,\beta}^{-1}) \epsilon_{\alpha,\beta} \eta u^\prime u g \right) = f_{\tau,\sigma,s}\left( \epsilon_{\alpha,\beta} \eta u^\prime u g \right)
\end{split}
\end{equation*}
and
\begin{equation*}
\int_{S(F)\backslash S(\A)} \psi_{\ell,a}^{-1}(s)ds=0.	
\end{equation*}
Thus we obtain \eqref{eq-lemma-unfolding-vanishing-alpha=0}.
\end{proof}

By combining \eqref{eq-unfolding-ENpsi-2} with Lemma~\ref{lemma-unfolding-vanishing-alpha=0}, we obtain the following.

\begin{proposition}
\label{prop-unfolding-1}
For $\mathrm{Re}(s)\gg 0$, we have  
 \begin{equation}
 \label{eq-unfolding-ENpsi-3}
 \begin{split}
 &E^{N_\ell, \psi_{\ell,a}}(g, f_{\tau,\sigma,s})\\
 =&   \sum_{\epsilon_{\beta}\in \tilde{\mathcal{E}}_{k,\ell}^0 } \sum_{\eta\in \mathcal{N}_{\beta, \ell, a}^0} \sum_{\delta\in G^{\eta}(F)\backslash G(F)} \int_{N_{\ell}^{\eta}(\A)\backslash N_{\ell}(\A)}   \int_{N_{\ell}^{\eta}(F)\backslash N_{\ell}^{\eta}(\A)}    f_{\tau,\sigma,s}(\epsilon_{\beta} \eta \delta u^\prime u g) \psi_{\ell,a}^{-1}(u^\prime u)du^\prime du,
\end{split}
\end{equation}
where
\begin{itemize}
\item $\epsilon_\beta=\epsilon_{0,\beta}\in \tilde{\mathcal{E}}_{k,\ell}^0$, which is the subset of $\tilde{\mathcal{E}}_{k,\ell}$ consisting of elements with $\alpha=0$,
\item $\eta\in \mathcal{N}_{\beta, \ell, a}^0$, which is the subset of $\mathcal{N}_{\alpha,\beta, \ell, a}$ consisting of elements with $\alpha=0$,  $\pr(\eta)=\begin{pmatrix}
 \epsilon & & \\ & \gamma &\\ && \epsilon^*	
 \end{pmatrix}$, where $\epsilon=\begin{pmatrix} & I_{\ell+\beta-k}\\ I_{k-\beta} &\end{pmatrix}$, and have the property that if $\beta>\max\{k-\ell, 0\}$, then $\gamma w_0$ is orthogonal to $V_{\ell,\beta}^-$ for $\gamma\in P_{w,\SO_{m^\prime}}^\prime\backslash \SO(W_{\ell}) / \mathrm{Stab}_{L_{\ell,\SO_{m^\prime}}}(\psi_{\ell, a})$.	
\end{itemize}
\end{proposition}

Therefore, by Proposition~\ref{prop-unfolding-1}, we obtain that for $\mathrm{Re}(s)\gg 0$, 
\begin{equation*}
\begin{split}
	&\mathcal{Z}(\phi_\pi, f_{\tau,\sigma, s})\\ =& \int_{C_G^0(\A) G(F)\backslash G(\A)} \phi_\pi(g) E^{N_\ell, \psi_{\ell,a}}(g, f_{\tau,\sigma,s})dg\\
	=&  \int_{C_G^0(\A) G(F)\backslash G(\A)} \phi_\pi(g) \sum_{\epsilon_\beta, \eta, \delta} \int_{N_{\ell}^{\eta}(\A)\backslash N_{\ell}(\A)}   \int_{N_{\ell}^{\eta}(F)\backslash N_{\ell}^{\eta}(\A)}    f_{\tau,\sigma,s}(\epsilon_{\beta} \eta \delta u^\prime u g) \psi_{\ell,a}^{-1}(u^\prime u)du^\prime du dg\\
	=& \sum_{\epsilon_\beta, \eta} \int_{C_G^0(\A) G(F)\backslash G(\A)} \sum_{\delta} \phi_\pi(\delta g) \int_{N_{\ell}^{\eta}(\A)\backslash N_{\ell}(\A)}   \int_{N_{\ell}^{\eta}(F)\backslash N_{\ell}^{\eta}(\A)}    f_{\tau,\sigma,s}(\epsilon_{\beta} \eta u^\prime u \delta  g) \psi_{\ell,a}^{-1}(u^\prime u)du^\prime du dg
\end{split}
\end{equation*}
where the conditions for the summations in $\epsilon_{\beta}$, $\eta$, and $\delta$ are given in Proposition~\ref{prop-unfolding-1}. Now we collapse the $dg$-integration and the summation over $\delta$ to obtain that $\mathcal{Z}(\phi_\pi, f_{\tau,\sigma, s})$ is equal to 
\begin{equation}
\label{eq-unfolding-integral1}
 \sum_{\epsilon_\beta, \eta} \int_{C_G^0(\A) G^\eta(F)\backslash G(\A)}  \phi_\pi( g) \int_{N_{\ell}^{\eta}(\A)\backslash N_{\ell}(\A)}   \int_{N_{\ell}^{\eta}(F)\backslash N_{\ell}^{\eta}(\A)}    f_{\tau,\sigma,s}(\epsilon_{\beta} \eta u^\prime u   g) \psi_{\ell,a}^{-1}(u^\prime u)du^\prime du dg.
\end{equation}

In the next lemma, we will use the cuspidality of $\phi_\pi$ to show the vanishing of some terms in \eqref{eq-unfolding-integral1}. 

\begin{lemma}
\label{lemma-unfolding-vanishing-cuspidal}
Let $\alpha=0$, $\eta=\eta_{\epsilon,\gamma}\in \mathcal{N}_{\beta, \ell, a}^0$ with $\pr(\eta)=\begin{pmatrix}
 \epsilon & & \\ & \gamma &\\ && \epsilon^*	
 \end{pmatrix}$, where	$\epsilon=\begin{pmatrix} & I_{\ell+\beta-k}\\ I_{k-\beta} &\end{pmatrix}$ and $\gamma\in P_{w,\SO_{m^\prime}}^\prime\backslash \SO(W_{\ell}) / \mathrm{Stab}_{L_{\ell,\SO_{m^\prime}}}(\psi_{\ell, a})$. If the stabilizer $G^{\eta}$ is a proper maximal parabolic subgroup of $G$, then
 \begin{equation}
 \label{eq-unfolding-vanishing-cuspidal}
 \int_{C_G^0(\A) G^\eta(F)\backslash G(\A)}  \phi_\pi( g) \int_{N_{\ell}^{\eta}(\A)\backslash N_{\ell}(\A)}   \int_{N_{\ell}^{\eta}(F)\backslash N_{\ell}^{\eta}(\A)}    f_{\tau,\sigma,s}(\epsilon_{\beta} \eta u^\prime u   g) \psi_{\ell,a}^{-1}(u^\prime u)du^\prime du dg=0.	
 \end{equation}
\end{lemma}

\begin{proof}
Assume $G^\eta$ is a proper maximal parabolic subgroup of $G$.
Write $G^{\eta}=M^\prime\cdot U^\prime$, where $U^\prime$ is the unipotent radical of $G^\eta$. Then the integral in \eqref{eq-unfolding-vanishing-cuspidal} is equal to
 \begin{equation*}
 \begin{split}
 \int_{C_G^0(\A) M^\prime(F)U^\prime(\A) \backslash G(\A)}  \int_{U^\prime(F)\backslash U^\prime(\A)} \phi_\pi( n g) \int_{N_{\ell}^{\eta}(\A)\backslash N_{\ell}(\A)}   \int_{N_{\ell}^{\eta}(F)\backslash N_{\ell}^{\eta}(\A)}    f_{\tau,\sigma,s}(\epsilon_{\beta} \eta u^\prime u    n g) \psi_{\ell,a}^{-1}(u^\prime u)du^\prime du dn dg.
 \end{split}	
 \end{equation*}
Since $G$ normalizes $N_{\ell}$ and stabilizes $\psi_{\ell, a}$, the above integral is equal to
 \begin{equation*}
 \begin{split}
 \int_{C_G^0(\A) M^\prime(F)U^\prime(\A) \backslash G(\A)}  \int_{U^\prime(F)\backslash U^\prime(\A)} \phi_\pi( n g) \int_{N_{\ell}^{\eta}(\A)\backslash N_{\ell}(\A)}   \int_{N_{\ell}^{\eta}(F)\backslash N_{\ell}^{\eta}(\A)}    f_{\tau,\sigma,s}(\epsilon_{\beta} \eta n u^\prime u      g) \psi_{\ell,a}^{-1}(u^\prime u)du^\prime du dn dg.
 \end{split}	
 \end{equation*}
 Note that $f_{\tau,\sigma,s}$ is left-invariant under $M_k(F)U_k(\A)$, and hence is left-invariant with respect to the conjugation by $\epsilon_{\beta} \eta$ on the unipotent radical $U^\prime(\A)$. Thus the above integral is equal to
  \begin{equation*}
 \begin{split}
 \int_{C_G^0(\A) M^\prime(F)U^\prime(\A) \backslash G(\A)}  \int_{U^\prime(F)\backslash U^\prime(\A)} \phi_\pi( n g) dn \int_{N_{\ell}^{\eta}(\A)\backslash N_{\ell}(\A)}   \int_{N_{\ell}^{\eta}(F)\backslash N_{\ell}^{\eta}(\A)}    f_{\tau,\sigma,s}(\epsilon_{\beta} \eta  u^\prime u      g) \psi_{\ell,a}^{-1}(u^\prime u)du^\prime du dg.
 \end{split}	
 \end{equation*}
The inner integral
\begin{equation*}
	\int_{U^\prime(F)\backslash U^\prime(\A)} \phi_\pi( n g) dn =0
\end{equation*}
by the cuspidality of $\phi_\pi$. Thus we obtain Lemma~\ref{lemma-unfolding-vanishing-cuspidal}.
\end{proof}

We will combine \eqref{eq-unfolding-integral1} with Lemma~\ref{lemma-unfolding-vanishing-cuspidal} to reduce the terms in the summation in \eqref{eq-unfolding-integral1}. Now we collect some facts about the double coset space $P_{w,\SO_{m^\prime}}^\prime\backslash \SO(W_{\ell}) / \mathrm{Stab}_{L_{\ell,\SO_{m^\prime}}}(\psi_{\ell, a})$ from \cite[Proposition 4.4]{GinzburgRallisSoudry2011}. Recall that $P_{w,\SO_{m^\prime}}^\prime$ is a maximal parabolic subgroup of $\SO(W_{\ell})$. We denote by $|P_{w,\SO_{m^\prime}}^\prime\backslash \SO(W_{\ell}) / \mathrm{Stab}_{L_{\ell,\SO_{m^\prime}}}(\psi_{\ell, a})|$ the number of representatives.

\begin{lemma}\cite[Proposition 4.4]{GinzburgRallisSoudry2011}
\label{lemma-GRS-double-coset}
Let $X$ be a non-trivial totally isotropic subspace of $W_{\ell}$ and let $P$ be the maximal parabolic subgroup of $\SO(W_{\ell})$ preserving $X$. Then we have the following.
\begin{enumerate}
\item If $\dim X<\mathrm{Witt}(W_{\ell})$, then $|P\backslash \SO(W_{\ell})/ \mathrm{Stab}_{L_{\ell,\SO_{m^\prime}}}(\psi_{\ell, a})|=2$.
\item Assume that $\dim X=\mathrm{Witt}(W_{\ell})$ and $\mathrm{Witt}(w_0^\perp\cap W_{\ell})=\dim X$. 
\begin{enumerate}
\item If $\dim W_{\ell}\ge 2\dim X+2$, then $|P\backslash \SO(W_{\ell})/ \mathrm{Stab}_{L_{\ell,\SO_{m^\prime}}}(\psi_{\ell, a})|=2$.
\item If $\dim W_{\ell}= 2\dim X+1$, then $|P\backslash \SO(W_{\ell})/ \mathrm{Stab}_{L_{\ell,\SO_{m^\prime}}}(\psi_{\ell, a})|=3$.
\end{enumerate}
\item Assume that $\dim X=\mathrm{Witt}(W_{\ell})$ and $\mathrm{Witt}(w_0^\perp\cap W_{\ell})=\dim X-1$. Then \\ $|P\backslash \SO(W_{\ell})/ \mathrm{Stab}_{L_{\ell,\SO_{m^\prime}}}(\psi_{\ell, a})|=1$.
\item If $\dim W_{\ell}=2\dim X$, then $\mathrm{Witt}(w_0^\perp\cap W_{\ell})=\dim X-1$, and, in particular,\\ $|P\backslash \SO(W_{\ell})/ \mathrm{Stab}_{L_{\ell,\SO_{m^\prime}}}(\psi_{\ell, a})|=1$.
\end{enumerate}
\end{lemma}

We combine Remark~\ref{remark-parabolic-subgroup} and Lemma~\ref{lemma-GRS-double-coset} to obtain the following description on the number of representatives $|P_{w,\SO_{m^\prime}}^\prime\backslash \SO(W_{\ell}) / \mathrm{Stab}_{L_{\ell,\SO_{m^\prime}}}(\psi_{\ell, a})|$. 
Note that we are only interested in the situation where $\alpha=0$.

\begin{lemma}
\label{lemma-double-coset2}
\begin{enumerate}
\item Suppose $\ell+\beta<\tilde{m}$. Then 
\begin{equation*}
	| P_{w,\SO_{m^\prime}}^\prime\backslash \SO(W_{\ell}) / \mathrm{Stab}_{L_{\ell,\SO_{m^\prime}}}(\psi_{\ell, a}) | =2.
\end{equation*}

\item Suppose $\ell+\beta=\tilde{m}$. Then we have the following.

\begin{enumerate}
\item If $\SO(W_{\ell})$ is an odd special orthogonal group, then
\begin{equation*}
	| P_{w,\SO_{m^\prime}}^\prime\backslash \SO(W_{\ell}) / \mathrm{Stab}_{L_{\ell,\SO_{m^\prime}}}(\psi_{\ell, a}) | = \begin{cases}
3  & \text{ if } \mathrm{Witt}(w_0^\perp \cap W_{\ell})=\tilde{m}-\ell,  \\
 1	 & \text{ if } \mathrm{Witt}(w_0^\perp \cap W_{\ell})=\tilde{m}-\ell-1.
 \end{cases}
\end{equation*}

\item 	If $\SO(W_{\ell})$ is a $F$-quasi-split non-split even special orthogonal group (with $\dim V_0=2$), then
\begin{equation*}
	| P_{w,\SO_{m^\prime}}^\prime\backslash \SO(W_{\ell}) / \mathrm{Stab}_{L_{\ell,\SO_{m^\prime}}}(\psi_{\ell, a}) | = \begin{cases}
2  & \text{ if } \mathrm{Witt}(w_0^\perp \cap W_{\ell})=\tilde{m}-\ell,  \\
 1	 & \text{ if } \mathrm{Witt}(w_0^\perp \cap W_{\ell})=\tilde{m}-\ell-1.
 \end{cases}
\end{equation*}

\item If $\SO(W_{\ell})$ is a $F$-split even special orthogonal group. Then 
\begin{equation*}
	| P_{w,\SO_{m^\prime}}^\prime\backslash \SO(W_{\ell}) / \mathrm{Stab}_{L_{\ell,\SO_{m^\prime}}}(\psi_{\ell, a}) | =2.
\end{equation*}
\end{enumerate}
\end{enumerate}
\end{lemma}

\begin{proof}
These can be checked case by case. We prove the case when $\SO(W_{\ell})$ is a $F$-quasi-split non-split even orthogonal group. In this case, $m^\prime=2m$, $\tilde{m}=m-1$, $\mathrm{Witt}(W_{\ell})=\tilde{m}-\ell$, and $P_{w,\SO_{m^\prime}}^\prime$ preserves a $\beta$-dimensional subspace. If $\ell+\beta<\tilde{m}$, then $\beta<\tilde{m}-\ell=\mathrm{Witt}(W_{\ell})$, so we can apply Lemma~\ref{lemma-GRS-double-coset} (1) to conclude that 
\begin{equation*}
	| P_{w,\SO_{m^\prime}}^\prime\backslash \SO(W_{\ell}) / \mathrm{Stab}_{L_{\ell,\SO_{m^\prime}}}(\psi_{\ell, a}) | =2.
\end{equation*}
If $\ell+\beta=\tilde{m}$ and $\mathrm{Witt}(w_0^\perp \cap W_{\ell})=\tilde{m}-\ell$, then $\mathrm{Witt}(w_0^\perp \cap W_{\ell})=\beta$ and $\dim W_{\ell}=2m-2\ell=2\dim X+2$, and we apply Lemma~\ref{lemma-GRS-double-coset} (2)(a) to conclude that 
\begin{equation*}
	| P_{w,\SO_{m^\prime}}^\prime\backslash \SO(W_{\ell}) / \mathrm{Stab}_{L_{\ell,\SO_{m^\prime}}}(\psi_{\ell, a}) | =2.
\end{equation*}
If $\ell+\beta=\tilde{m}$ and $\mathrm{Witt}(w_0^\perp \cap W_{\ell})=\tilde{m}-\ell-1$, then $\beta=\mathrm{Witt}(W_{\ell})$, $\mathrm{Witt}(w_0^\perp \cap W_{\ell})=\beta-1$, and we apply Lemma~\ref{lemma-GRS-double-coset} (3) to conclude that 
\begin{equation*}
	| P_{w,\SO_{m^\prime}}^\prime\backslash \SO(W_{\ell}) / \mathrm{Stab}_{L_{\ell,\SO_{m^\prime}}}(\psi_{\ell, a}) | =1.
\end{equation*}
The other cases can be proved similarly and we omit the details. 
\end{proof}

We now consider the summands in \eqref{eq-unfolding-integral1} corresponding to $\beta$ with $\max\{k-\ell, 0\}<\beta<\tilde{m}-\ell$. In this case, $P_{w,\SO_{m^\prime}}^\prime\backslash \SO(W_{\ell}) / \mathrm{Stab}_{L_{\ell,\SO_{m^\prime}}}(\psi_{\ell, a})$ has two representatives $\gamma_1$ and $\gamma_2$ such that  $\gamma_1 w_0$ is orthogonal to $V_{\ell,\beta}^-$ and $\gamma_2w_0$ is not orthogonal to $V_{\ell,\beta}^-$. Moreover, for the representative $\eta$ with $\pr(\eta)=\begin{pmatrix}
 \epsilon & & \\ & \gamma_1 &\\ && \epsilon^*	
 \end{pmatrix}$, the stabilizer $G^{\eta}$ is a proper maximal parabolic subgroup of $G$, which preserves the isotropic subspace $w_q^{k-\beta}V_{\ell,\beta}^+\cap w_0^\perp$.  Thus, all the summands in this case are equal to zero by Lemma~\ref{lemma-unfolding-vanishing-cuspidal} (for $\gamma_1$) and Lemma~\ref{lemma-unfolding-vanishing-alpha=0} (for $\gamma_2$).  

Next, we consider the summands in \eqref{eq-unfolding-integral1} corresponding to $\beta$ with $\beta=\tilde{m}-\ell$. Note that we have $\beta>\max\{0,k-\ell\}$. We have the following cases:
\begin{itemize}
\item If $\SO(W_{\ell})$ is a $F$-split even special orthogonal group, then by Lemma~\ref{lemma-double-coset-case2} there are two double coset representatives $\epsilon_{0,\beta}$ and $\tilde{\epsilon}_{0,\beta}$ for $P_k(F)\backslash H(F)/P_{\ell}(F)$ corresponding to the pair $(0,\beta)$. For these two representatives, their stabilizers $P_{\ell}^{\epsilon_{0,\beta}}$ and $P_{\ell}^{\tilde{\epsilon}_{0,\beta}}$ both preserve a maximal isotropic subspace of $W_{\ell}$, and the double cosets 
$$P_{\epsilon_{0,\beta},\SO_{m^\prime}}^\prime\backslash \SO(W_{\ell}) / \mathrm{Stab}_{L_{\ell,\SO_{m^\prime}}}(\psi_{\ell, a}) \text{ and } P_{\tilde{\epsilon}_{0,\beta},\SO_{m^\prime}}^\prime\backslash \SO(W_{\ell}) / \mathrm{Stab}_{L_{\ell,\SO_{m^\prime}}}(\psi_{\ell, a})$$ 
both consist of one element with their stabilizers $G^\eta$ being a proper maximal parabolic subgroup of $G$ in both cases. By Lemma~\ref{lemma-unfolding-vanishing-cuspidal}, the corresponding summands are equal to zero. 
\item If $\SO(W_{\ell})$ is either a $F$-quasi-split non-split even special orthogonal group or an odd special orthogonal group, with $\mathrm{Witt}(w_0^\perp \cap W_{\ell})=\mathrm{Witt}(W_{\ell})-1$, then the double coset space $P_{w,\SO_{m^\prime}}^\prime\backslash \SO(W_{\ell}) / \mathrm{Stab}_{L_{\ell,\SO_{m^\prime}}}(\psi_{\ell, a})$ has only one element, and its stabilizer $G^\eta$ is a proper maximal parabolic subgroup of $G$. Hence the corresponding summand vanishes by Lemma~\ref{lemma-unfolding-vanishing-cuspidal}.
\item If $\SO(W_{\ell})$ is an odd special orthogonal group with $\mathrm{Witt}(w_0^\perp \cap W_{\ell})=\mathrm{Witt}(W_{\ell})$, then $P_{w,\SO_{m^\prime}}^\prime\backslash \SO(W_{\ell}) / \mathrm{Stab}_{L_{\ell,\SO_{m^\prime}}}(\psi_{\ell, a})$ consists of three elements, given explicitly in \cite[(4.33)]{GinzburgRallisSoudry2011}. In this case, two stabilizers $G^{\eta}$ are a proper maximal parabolic subgroup of $G$, and the third representative $\gamma$ satisfies the condition that $\gamma w_0$ is not orthogonal to $V_{\ell,\beta}^-$. By Lemma~\ref{lemma-unfolding-vanishing-cuspidal} and Lemma~\ref{lemma-unfolding-vanishing-alpha=0}, all the corresponding summands are zero. 
\item If $\SO(W_{\ell})$ is a $F$-quasi-split non-split even special orthogonal group, with $\mathrm{Witt}(w_0^\perp \cap W_{\ell})=\mathrm{Witt}(W_{\ell})$, then $P_{w,\SO_{m^\prime}}^\prime\backslash \SO(W_{\ell}) / \mathrm{Stab}_{L_{\ell,\SO_{m^\prime}}}(\psi_{\ell, a})$ consists of two elements, and their stabilizers are similar to the case $\beta<\tilde{m}-\ell$ as discussed above. Hence by Lemma~\ref{lemma-unfolding-vanishing-cuspidal} and Lemma~\ref{lemma-unfolding-vanishing-alpha=0} the corresponding summands are equal to zero. 
\end{itemize}
Therefore, all the summands corresponding to $\beta=\tilde{m}-\ell$ are equal to zero.

We have proved that, for $\mathrm{Re}(s)\gg 0$, $\mathcal{Z}(\phi_\pi, f_{\tau,\sigma, s})$ is equal to 
\begin{equation}
\label{eq-unfolding-integral2}
 \sum_{\epsilon_\beta, \eta} \int_{C_G^0(\A) G^\eta(F)\backslash G(\A)}  \phi_\pi( g) \int_{N_{\ell}^{\eta}(\A)\backslash N_{\ell}(\A)}   \int_{N_{\ell}^{\eta}(F)\backslash N_{\ell}^{\eta}(\A)}    f_{\tau,\sigma,s}(\epsilon_{\beta} \eta u^\prime u   g) \psi_{\ell,a}^{-1}(u^\prime u)du^\prime du dg,
\end{equation}
where the summation is subject to the conditions in Proposition~\ref{prop-unfolding-1}, with additional conditions that $\beta=\max\{k-\ell,0\}$, and $\eta$ a representative in $P_{\ell}^{\epsilon_{\beta}}(F)\backslash P_{\ell}(F) / R_{\ell,a}(F)$ with $\pr(\eta)=\begin{pmatrix}
 \epsilon & & \\ & \gamma &\\ && \epsilon^*	
 \end{pmatrix}$, where $G^\eta$ is not a proper maximal parabolic subgroup of $G$. In fact, the representative $\eta$ is uniquely determined by $\beta=\max\{k-\ell,0\}$. More specifically, we have the following determination of $\eta$:
 \begin{itemize}

\item If $k>\ell$, then $\beta=k-\ell$, and $\eta$ is determined by 
$\pr(\eta)=\diag(\epsilon, \gamma, \epsilon^*)$ with $\epsilon=I_{\ell}$, and $\gamma\in P_{w,\SO_{m^\prime}}^\prime\backslash \SO(W_{\ell}) / \mathrm{Stab}_{L_{\ell,\SO_{m^\prime}}}(\psi_{\ell, a})$. Since $\ell+\beta=k<\tilde{m}$, the double coset space $P_{w,\SO_{m^\prime}}^\prime\backslash \SO(W_{\ell}) / \mathrm{Stab}_{L_{\ell,\SO_{m^\prime}}}(\psi_{\ell, a})$ has two elements by Lemma~\ref{lemma-double-coset2}.
Thus, the double coset space $P_{\ell}^{\epsilon_{\beta}}(F)\backslash P_{\ell}(F) / R_{\ell,a}(F)$ gives two representatives. One representative has the property that $\gamma=I_{m^\prime-2\ell}$ and the corresponding stabilizer $G^\eta$ is a proper maximal parabolic subgroup, and hence does not contribute to the integral by Lemma~\ref{lemma-unfolding-vanishing-cuspidal}. The other representative $\eta_2$ is uniquely determined by 
\begin{equation}
\label{eq-eta-kgreaterl}
	\pr(\eta_2)=\begin{pmatrix}
 \epsilon & & \\ & \gamma &\\ && \epsilon^*	
 \end{pmatrix}, \quad \epsilon=I_{\ell}, \quad \gamma=\begin{pmatrix} &I_{k-\ell} & & &\\ I_{\tilde{m}-k} &&&& \\ &&I_{m^\prime-2\tilde{m}}&&\\ &&&&I_{\tilde{m}-k}\\&&&I_{k-\ell} &\end{pmatrix}.
\end{equation}

 \item If $k\le \ell$, then 	$\beta=0$. In this case, the representative $\eta=\eta_1$ is determined by
\begin{equation}
\label{eq-eta-klel}
	\pr(\eta_1)=\begin{pmatrix}
 \epsilon & & \\ & \gamma &\\ && \epsilon^*	
 \end{pmatrix}, \quad \gamma=I_{m^\prime-2\ell}, \quad \epsilon=\begin{pmatrix} &I_{\ell-k}\\ I_k&\end{pmatrix}.
\end{equation}
Moreover, we have $G^{\eta_1}=G$.
 \end{itemize}

We summarize the above computation in the following proposition.

\begin{proposition}
\label{prop}
Let the notation be as above. Assume $\mathrm{Re}(s)\gg 0$.
\begin{enumerate}

\item If $k>\ell$, then $\beta=k-\ell$ and the global integral $\mathcal{Z}(\phi_\pi, f_{\tau,\sigma, s})$ is equal to 
\begin{equation}
\label{eq-unfolding-integral3-k-greater-than-l}
\int_{C_G^0(\A) G^{\eta_2}(F)\backslash G(\A)}  \phi_\pi( g) \int_{N_{\ell}^{\eta_2}(\A)\backslash N_{\ell}(\A)}   \int_{N_{\ell}^{\eta_2}(F)\backslash N_{\ell}^{\eta_2}(\A)}    f_{\tau,\sigma,s}(\epsilon_{0,\beta} \eta_2 u^\prime u   g) \psi_{\ell,a}^{-1}(u^\prime u)du^\prime du dg,
\end{equation}
with $\eta_2$ determined by \eqref{eq-eta-kgreaterl}.

\item If $k\le \ell$, then $\beta=0$ and the global integral $\mathcal{Z}(\phi_\pi, f_{\tau,\sigma, s})$ is equal to 
\begin{equation}
\label{eq-unfolding-integral3-klel}
\int_{C_G^0(\A) G(F)\backslash G(\A)}  \phi_\pi( g) \int_{N_{\ell}^{\eta_1}(\A)\backslash N_{\ell}(\A)}   \int_{N_{\ell}^{\eta_1}(F)\backslash N_{\ell}^{\eta_1}(\A)}    f_{\tau,\sigma,s}(\epsilon_{0,0} \eta_1 u^\prime u   g) \psi_{\ell,a}^{-1}(u^\prime u)du^\prime du dg,
\end{equation}
with $\eta_1$ determined by \eqref{eq-eta-klel}. 
\end{enumerate}
\end{proposition}

In Section~\ref{subsection-unfolding-kgm} and Section~\ref{subsection-unfolding-klm} below, we will treat the case $k>\ell$ and the case $k\le \ell$ separately. We will prove Theorem~\ref{thm-global-unfolding-kgreaterl} and Theorem~\ref{thm-global-unfolding-klel} in Section~\ref{subsection-unfolding-kgm} and Section~\ref{subsection-unfolding-klm} respectively.

\subsection{Euler product factorization: the case of $k>\ell$}
\label{subsection-unfolding-kgm}
In this section we complete the unfolding computation for \eqref{eq-unfolding-integral3-k-greater-than-l} and prove Theorem~\ref{thm-global-unfolding-kgreaterl}.
We assume $k>\ell$, $\beta=k-\ell$. Recall that $N_{\ell}^{\eta_2}=N_{\ell}\cap \eta_2^{-1}P_{\ell}^{\epsilon_{0,\beta}} \eta_2$ is the stabilizer in $N_{\ell}$, which is explicitly given by
$$
N_{\ell}^{\eta_2}=\left\{ \begin{pmatrix}
z&0&0&0&y&0&0\\
&I_{\tilde{m}-k}&&&&&0\\
&&I_{k-\ell}&&&&y'\\
&&&I_{m^\prime-2\tilde{m}}&&&0\\
&&&&I_{k-\ell}&&0\\
&&&&&I_{\tilde{m}-k}&0\\
&&&&&&z^{*}
\end{pmatrix} \mid z\in Z_{\ell} \right\}.
$$
We conjugate the unipotent $N_{\ell}^{\eta_2}$ by $\epsilon_{0,\beta} \eta_2$ to obtain
\begin{equation}
\label{eq-Z-l-prime}
	\epsilon_{0,\beta} \eta_2 N_{\ell}^{\eta_2} (\epsilon_{0,\beta} \eta_2)^{-1}= \left\{ \hat{z}_{\ell}^\prime= \begin{pmatrix} I_{k-\ell} & y \\ &z \end{pmatrix}^\wedge \mid z\in Z_\ell \right\} = \hat{Z}_\ell^\prime
\end{equation}
where for $g\in \GL(V_{k}^+)$, we denote by $\hat{g}$ the lift of $g$ into the Levi $\GL(V_k^+)\times \GSpin(W_k)$ of $P_k$. 
For $\hat{z}^\prime= \begin{pmatrix} I_{k-\ell} & y \\ &z \end{pmatrix}^\wedge\in \hat{Z}_\ell^\prime$ with $z=(z_{i,j})$ and $y=(y_{i,j})$, the character  $\psi_{\ell,a} ( \eta_2^{-1}\epsilon_{0,\beta}^{-1}\hat{z}^\prime \epsilon_{0,\beta}\eta_2)$ is given by
\begin{equation}
\psi_{\ell,a}^{-1}( \eta_2^{-1}\epsilon_{0,\beta}^{-1}\hat{z}^\prime \epsilon_{0,\beta}\eta_2)=	\psi(z_{1,2}+\cdots + z_{\ell-1, \ell}+(-1)^{m^\prime+1}\frac{a}{2}y_{k-\ell, 1})
\end{equation}
and so on the group $\hat{Z}_\ell^\prime$ we set
\begin{equation}
\psi_{Z_{\ell}^\prime,a}	(z^\prime)=\psi( (-1)^{m^\prime+1}\frac{a}{2}z^\prime_{k-\ell, k-\ell+1} +z^\prime_{k-\ell+1,k-\ell+2}+\cdots + z^\prime_{k-1,k} ), \quad z^\prime=(z^\prime_{i,j}).
\end{equation}
If we set
\begin{equation}
\label{eq-f-Z-l-prime}
\begin{split}
f_{\tau,\sigma,s}^{Z_{\ell}^\prime, \psi^{-1}_{Z_{\ell}^\prime,a}	}(g) &= \int_{Z_{\ell}^\prime(F)\backslash  Z_{\ell}^\prime(\A)} f_{\tau,\sigma,s}(\hat{z}^\prime     g)  \psi_{Z_{\ell}^\prime,a}	(\hat{z}^\prime) d\hat{z}^\prime \\
&=  \int_{Z_{\ell}^\prime(F)\backslash  Z_{\ell}^\prime(\A)} f_{\tau,\sigma,s}(\hat{z}^\prime     g) \psi_{\ell,a}^{-1}( \eta_2^{-1}\epsilon_{0,\beta}^{-1}\hat{z}^\prime \epsilon_{0,\beta}\eta_2) d\hat{z}^\prime,
\end{split}
\end{equation}
then by \eqref{eq-unfolding-integral3-klel}, the global integral is equal to
\begin{equation}
\label{eq-unfolding-integral3-klel-2}
\mathcal{Z}(\phi_\pi, f_{\tau,\sigma, s})=\int_{C_G^0(\A) G^{\eta_2}(F)\backslash G(\A)}  \phi_\pi( g) \int_{N_{\ell}^{\eta_2}(\A)\backslash N_{\ell}(\A)}       f_{\tau,\sigma,s}^{Z_{\ell}^\prime, \psi^{-1}_{Z_{\ell}^\prime,a}	}(\epsilon_{0,\beta} \eta_2  u   g) \psi_{\ell,a}^{-1}( u)  du dg.
\end{equation}

Next, we turn to the subgroup $G^{\eta_2}=G\cap \eta_2^{-1} M_{\ell}^{\epsilon_{0,\beta}} \eta_2$, as defined in \eqref{eq-G-eta-N-l-eta}. 
Similar to the decomposition in \cite[(3.33)]{JiangZhang2014}, we also have a decomposition
\begin{equation}
\label{eq-G-eta2}
G^{\eta_2}=(\GL(V_{\tilde{m}-\beta,\beta-1}^+)\times \GSpin(\pr(\eta_2)^{-1}W_k)  ) \ltimes 	U_{\beta-1,\eta_2},
\end{equation}
where $V_{\tilde{m}-\beta,\beta-1}^+$ is defined in 
\eqref{eq-subspace-V-lt}, $W_k=(V_k^+\oplus V_k^{-})^\perp$ is defined in \eqref{eq-W-l}, and $U_{\beta-1,\eta_2}$ is the unipotent radical of the stabilizer $G^{\eta_2}$, as described in \cite[page 573]{JiangZhang2014}. More specifically, we have
\begin{equation*}
V_{\tilde{m}-\beta,\beta-1}^+=	\Span\{e_{\tilde{m}-k+\ell+1}, \cdots, e_{\tilde{m}-1} \} =\pr(\eta_2)^{-1}V_{\ell,\beta}^+\cap w_0^\perp=\pr(\eta_2)^{-1}V_{\ell,\beta-1}^+,
\end{equation*}
and
\begin{equation*}
U_{\beta-1,\eta_2}=\eta_2^{-1} U_{\beta,m^\prime-2\ell}^{\eta_2} \eta_2,
\end{equation*}
where elements of $U_{\beta,m^\prime-2\ell}^{\eta_2}$ have the form
\begin{equation}\label{eq-U-beta-mprime-2l-eta2}
\begin{pmatrix}
I_{\ell}&&&&&&\\
&I_{\beta-1}&d_{1}&u&v_{1}&v&\\
&&1&0&0&v'_{1}&\\
&&&I_{m^\prime-2k}&0&u'&\\
&&&&1&d'_{1}&\\
&&&&&I_{\beta-1}&\\
&&&&&&I_{\ell}
\end{pmatrix}.
\end{equation}
Denote
\begin{equation}
\label{eq-Z-ell-beta-1}
Z_{\ell,\beta-1}:=\left\{
\begin{pmatrix}
I_{\ell}&&&&\\
&d&&&\\
&&I_{m^\prime-2k+2}&&\\
&&&d^{*}&\\
&&&&I_{\ell}
\end{pmatrix}\mid d\in Z_{\beta-1}
\right\}	
\end{equation}
and 
\begin{equation*}
	Z_{\ell,\beta-1}^{\eta_2}:=\eta_2^{-1}Z_{\ell,\beta-1}\eta_2.
\end{equation*}
Then $Z_{\ell,\beta-1}^{\eta_2}$ is a maximal unipotent subgroup $Z_{\ell,\beta-1}^{\eta_2}:=\eta_2^{-1}Z_{\ell,\beta-1}\eta_2$ of $\GL(V_{\tilde{m}-\beta,\beta-1}^+)$.

Write
\begin{equation}
\label{eq-N-l-beta-1}
	N_{\ell, \beta-1}^{\eta_2}:=Z_{\ell,\beta-1}^{\eta_2} U_{\beta-1,\eta_2}.
\end{equation}
 This is a unipotent subgroup of $G$ of the type as defined in \eqref{eq-unipotent-Nl} with the corresponding character $\psi_{\beta-1,-a}$ on $N_{\ell,\beta-1}^{\eta_2}(\A)$ defined similarly as in \eqref{eq-psi-l-a}. The corresponding connected component of the identity of the stabilizer $(G^{\psi_{\beta-1,-a}})^0$ is equal to $\GSpin(\pr(\eta_2)^{-1}W_k)$. Hence
 $(G^{\psi_{\beta-1,-a}})^0=\eta_2^{-1} \GSpin(W_k)\eta_2 =\eta_2^{-1} G_0\eta_2$ and we denote
 \begin{equation*}
 G_0^\prime :=	\eta_2^{-1} G_0 \eta_2.
 \end{equation*}
 The elements of $N_{\ell, \beta-1}^{\eta_2}$ have the form
\begin{equation}
\label{eq-N-l-beta-1-eta2}
 (\epsilon_{0,\beta}\eta_2)^{-1} \begin{pmatrix}
d&d_1&0&u&0&v_1&v\\
&1&&&&&v_1^\prime\\
&&I_{\ell}&&&&0\\
&&&I_{m^\prime-2k}&&&u^\prime\\
&&&&I_{\ell}&&0\\
&&&&&1&d_1^\prime\\
&&&&&&d^*
\end{pmatrix}(\epsilon_{0,\beta}\eta_2) ,
\end{equation}
where $d\in Z_{\beta-1}$.
We remark that $Z_{\ell,\beta-1}^{\eta_2}$ is the set of all elements of the form \eqref{eq-N-l-beta-1-eta2} with all entries 0 except $d$. Let $Z_{\beta,\eta_2}$ be the subgroup of $N_{\ell, \beta-1}^{\eta_2}$ consisting of elements of the form \eqref{eq-N-l-beta-1-eta2} with all entires 0 except $d$ and $d_1$, and let $C_{\beta-1,\eta_2}$ be the subgroup of $N_{\ell, \beta-1}^{\eta_2}$ consisting of elements of the form \eqref{eq-N-l-beta-1-eta2} with $d=I_{\beta-1}$ and $d_1=0$. Then we have
\begin{equation*}
	N_{\ell, \beta-1}^{\eta_2}= Z_{\beta,\eta_2} C_{\beta-1,\eta_2}.
\end{equation*}
Let $R_{\ell,\beta-1}^{\eta_2}$ be the Bessel subgroup of $G$ given by
\begin{equation}
\label{eq-Bessel-subgroup-of-G}
	R_{\ell,\beta-1}^{\eta_2}= {G_0^\prime} \cdot N_{\ell, \beta-1}^{\eta_2}.
\end{equation}

Similar to \cite[p. 575]{JiangZhang2014}, we have
\begin{equation}
\label{eq-quotient-space-isom}
C_{\beta-1,\eta_2}\backslash G^{\eta_2}\cong P_{\beta}^1\times  {G_0^\prime},
\end{equation}
where
\begin{equation*}
P_{\beta}^1=\left\{ \begin{pmatrix} d & *\\0&1\end{pmatrix}\mid d\in \GL_{\beta-1} \right\} \subset \GL_{\beta}	
\end{equation*}
is the mirabolic subgroup of $\GL_\beta$. Now we return to the expression \eqref{eq-unfolding-integral3-klel-2}, and consider the inner integral
\begin{equation}
\label{eq-Phi}
	\Phi(g):=\int_{N_{\ell}^{\eta_2}(\A)\backslash N_{\ell}(\A)}       f_{\tau,\sigma,s}^{Z_{\ell}^\prime, \psi^{-1}_{Z_{\ell}^\prime,a}	}(\epsilon_{0,\beta} \eta_2  u   g) \psi_{\ell,a}^{-1}( u)  du .
\end{equation}

\begin{lemma}
\label{lemma-Phi-left-invariant-under-Cbetaeta}
The function $\Phi(g)$ is left-invariant under $C_{\beta-1,\eta_2}(\A)$.	
\end{lemma}

\begin{proof}
Note that
\begin{equation*}
\epsilon_{0,\beta}\eta_2  C_{\beta-1,\eta_2} \eta_2^{-1}\epsilon_{0,\beta}^{-1}= 	\left\{ \begin{pmatrix}
I_{\beta-1}&&&u&&v_1&v\\
&1&&&&&v_1^\prime\\
&&I_{\ell}&&&&\\
&&&I_{m^\prime-2k}&&&u^\prime\\
&&&&I_{\ell}&&\\
&&&&&1&\\
&&&&&&I_{\beta-1}
\end{pmatrix}\right\},
\end{equation*}
and for any $\hat{z}^\prime\in \hat{Z}_{\ell}^\prime$, we have
\begin{equation*}
\hat{z}^\prime 	\left( \epsilon_{0,\beta}\eta_2  C_{\beta-1,\eta_2} \eta_2^{-1}\epsilon_{0,\beta}^{-1} \right) \hat{z}^{\prime -1} \subset U_k.
\end{equation*}
It follows from the left-invariance of $f_{\tau,\sigma,s}$ under $U_k(\A)$  that, for any $c\in C_{\beta-1,\eta_2}(\A)$, we have
\begin{equation*}
\begin{split}
	 f_{\tau,\sigma,s}^{Z_{\ell}^\prime, \psi^{-1}_{Z_{\ell}^\prime,a}	}(\epsilon_{0,\beta} \eta_2  c    g) &= \int_{Z_{\ell}^\prime(F)\backslash  Z_{\ell}^\prime(\A)} f_{\tau,\sigma,s}\left( \left(\hat{z}^\prime     \left( \epsilon_{0,\beta} \eta_2  c \eta_2^{-1}\epsilon_{0,\beta}^{-1} \right) \hat{z}^{\prime-1} \right)\hat{z}^\prime \epsilon_{0,\beta} \eta_2    g\right) \psi_{\ell,a}^{-1}( \eta_2^{-1}\epsilon_{0,\beta}^{-1}\hat{z}^\prime \epsilon_{0,\beta}\eta_2) d\hat{z}^\prime \\
	 &=  \int_{Z_{\ell}^\prime(F)\backslash  Z_{\ell}^\prime(\A)} f_{\tau,\sigma,s}\left( \hat{z}^\prime \epsilon_{0,\beta} \eta_2    g\right) \psi_{\ell,a}^{-1}( \eta_2^{-1}\epsilon_{0,\beta}^{-1}\hat{z}^\prime \epsilon_{0,\beta}\eta_2) d\hat{z}^\prime \\
	 &=  f_{\tau,\sigma,s}^{Z_{\ell}^\prime, \psi^{-1}_{Z_{\ell}^\prime,a}	}(\epsilon_{0,\beta} \eta_2      g).
\end{split}
\end{equation*}
It follows that 
\begin{equation*}
\begin{split}
	\Phi( cg) &= \int_{N_{\ell}^{\eta_2}(\A)\backslash N_{\ell}(\A)}       f_{\tau,\sigma,s}^{Z_{\ell}^\prime, \psi^{-1}_{Z_{\ell}^\prime,a}	}(\epsilon_{0,\beta} \eta_2  u   cg) \psi_{\ell,a}^{-1}( u)  du \\
	&=\int_{N_{\ell}^{\eta_2}(\A)\backslash N_{\ell}(\A)}       f_{\tau,\sigma,s}^{Z_{\ell}^\prime, \psi^{-1}_{Z_{\ell}^\prime,a}	}(\epsilon_{0,\beta} \eta_2  cu   g) \psi_{\ell,a}^{-1}( u)  du \\
	&=\int_{N_{\ell}^{\eta_2}(\A)\backslash N_{\ell}(\A)}       f_{\tau,\sigma,s}^{Z_{\ell}^\prime, \psi^{-1}_{Z_{\ell}^\prime,a}	}(\epsilon_{0,\beta} \eta_2  u   g) \psi_{\ell,a}^{-1}( u)  du \\
	&=\Phi(g).
\end{split}	
\end{equation*}
This finishes the proof of the lemma.
\end{proof}

By \eqref{eq-quotient-space-isom} and Lemma~\ref{lemma-Phi-left-invariant-under-Cbetaeta}, we obtain that for $\Re(s)\gg 0$, we have 
\begin{equation}
\label{eq-unfolding-integral3-klel-3}
\begin{split}
\mathcal{Z}(\phi_\pi, f_{\tau,\sigma, s}) &=\int_{P_{\beta}^1(F) {G_0^\prime}(F) C_{\beta-1,\eta_2}(\A) C_G^0(\A) \backslash G(\A)} \Phi(g) \int_{C_{\beta-1,\eta_2}(F)\backslash C_{\beta-1,\eta_2}(\A) } \phi_\pi( cg) dc dg \\
&= \int_{P_{\beta}^1(F) {G_0^\prime}(F) C_{\beta-1,\eta_2}(\A) C_G^0(\A) \backslash G(\A)} \Phi(g) \phi_\pi^{C_{\beta-1,\eta_2}, 1}(g)dg,
\end{split}
\end{equation}
where
\begin{equation*}
	\phi_\pi^{C_{\beta-1,\eta_2}, 1}(g):= \int_{C_{\beta-1,\eta_2}(F)\backslash C_{\beta-1,\eta_2}(\A) } \phi_\pi( cg) dc.
\end{equation*}

Next, we will use the standard unfolding process of Piatetski-Shapiro and Shalika using partial Fourier expansion along the mirabolic subgroup $P_{\beta}^1$. Both $\Phi(g)$ and $\phi_\pi^{C_{\beta-1,\eta_2},1}(g)$ are automorphic functions on $P_{\beta}^1(\A)$, and $\phi_\pi^{C_{\beta-1,\eta_2},1}(g)$ is cuspidal due to the cuspidality of $\phi_\pi(g)$.  We have
\begin{equation}
\label{eq-phi-pi-C}
	\phi_\pi^{C_{\beta-1,\eta_2}, 1}(g)=\sum_{d\in Z_{\beta,\eta_2}(F)\backslash P_{\beta}^1(F)} \phi_\pi^{N_{\ell,\beta-1}^{\eta_2}, \psi_{\beta-1,-a}^{-1}} \left(d g\right),
\end{equation}
where
\begin{equation}
\phi_\pi^{N_{\ell,\beta-1}^{\eta_2}, \psi_{\beta-1,-a}^{-1}}(g)=	 \int_{ N_{\ell,\beta-1}^{\eta_2}(F)\backslash N_{\ell,\beta-1}^{\eta_2}(\A)} \phi_\pi(ug)  \psi_{\beta-1,-a}(u)du. 
\end{equation}
By \eqref{eq-phi-pi-C} and \eqref{eq-unfolding-integral3-klel-3}, we obtain that 
\begin{equation}
\label{eq-unfolding-integral3-klel-4}
\begin{split}
\mathcal{Z}(\phi_\pi, f_{\tau,\sigma, s})=  \int_{Z_{\beta,\eta_2}(F) {G_0^\prime}(F) C_{\beta-1,\eta_2}(\A) C_G^0(\A) \backslash G(\A)}  \Phi(g) \phi_\pi^{N_{\ell,\beta-1}^{\eta_2}, \psi_{\beta-1,-a}^{-1}}(g)dg.
\end{split}
\end{equation}
Since
 $\phi_\pi^{N_{\ell,\beta-1}^{\eta_2}, \psi_{\beta-1,-a}^{-1}}(g)$ is left-equivariant under $(Z_{\beta,\eta_2}(\A), \psi_{\beta-1,-a}^{-1})$, the expression in \eqref{eq-unfolding-integral3-klel-4} is equal to 
\begin{equation}
\label{eq-unfolding-integral3-klel-5}
\begin{split}
\int_{{G_0^\prime}(F)N_{\ell, \beta-1}^{\eta_2}(\A)  C_G^0(\A) \backslash G(\A)}  \phi_\pi^{N_{\ell,\beta-1}^{\eta_2}, \psi_{\beta-1,-a}^{-1}}(g) \int_{Z_{\beta,\eta_2}(F)\backslash Z_{\beta,\eta_2}(\A)} \Phi(zg) \psi_{\beta-1,-a}^{-1}(z)  dz dg.
\end{split}
\end{equation}

Now we focus on the inner integral in \eqref{eq-unfolding-integral3-klel-5}, which is given by 
\begin{equation}
\label{eq-unfolding-integral3-klel-6}
	\int_{Z_{\beta,\eta_2}(F)\backslash Z_{\beta,\eta_2}(\A)} \int_{N_{\ell}^{\eta_2}(\A)\backslash N_{\ell}(\A)}       f_{\tau,\sigma,s}^{Z_{\ell}^\prime, \psi^{-1}_{Z_{\ell}^\prime,a}	}(\epsilon_{0,\beta} \eta_2  u  z g) \psi_{\ell,a}^{-1}( u)  du  \psi_{\beta-1,-a}^{-1}(z)  dz.
\end{equation}
Recall that $f_{\tau,\sigma,s}^{Z_{\ell}^\prime, \psi^{-1}_{Z_{\ell}^\prime,a}	}$ is given in \eqref{eq-f-Z-l-prime}, and $Z_{\beta,\eta_2}$ is given in \eqref{eq-N-l-beta-1-eta2} with all entires 0 except $d$ and $d_1$. For $\Re(s)$ large enough, the integrals we have considered are absolutely convergent, which allow us to change the order of integration. 
Recall that
$$
N_{\ell}=\left\{ \begin{pmatrix}
c&x_{1}&x_{2}&x_{3}&y_{6}&x_{4}&x_{5}\\
&I_{\tilde{m}-k}&&&&&x'_{4}\\
&&I_{\beta}&&&&y'_{6}\\
&&&I_{m^\prime-2\tilde{m}}&&&x'_{3}\\
&&&&I_{\beta}&&x'_{2}\\
&&&&&I_{\tilde{m}-k}&x'_{1}\\
&&&&&&c^{*}
\end{pmatrix} \mid c\in Z_{\ell} \right\}
$$
and
$$
N_{\ell}^{\eta_2}=\left\{ \begin{pmatrix}
c&0&0&0&y_{6}&0&0\\
&I_{\tilde{m}-k}&&&&&0\\
&&I_{\beta}&&&&y'_{6}\\
&&&I_{m^\prime-2\tilde{m}}&&&0\\
&&&&I_{\beta}&&0\\
&&&&&I_{\tilde{m}-k}&0\\
&&&&&&c^{*}
\end{pmatrix} \mid c\in Z_{\ell} \right\}.
$$
We compute that
\begin{equation*}
\begin{split}
\epsilon_{0,\beta}\eta_2 (N_{\ell}^{\eta_2}Z_{\beta,\eta_2}) 	\eta_2^{-1} \epsilon_{0,\beta}^{-1} = \left\{ 
\begin{pmatrix}
\left(\begin{smallmatrix}
d & d_1\\&1
\end{smallmatrix}\right) &y_6^\prime & & &\\
&c^* & &&\\
&&I_{m^\prime-2k} &&\\
&&&c& y_6\left(\begin{smallmatrix} 1&d_1^\prime\\ & d^*\end{smallmatrix}\right)\\
&&&&I_{\beta}
\end{pmatrix} \mid c\in Z_{\ell}, d\in Z_{\beta-1} \right\}.
\end{split}
\end{equation*}
From the above computation we see that $N_{\ell}^{\eta_2}Z_{\beta,\eta_2}\cong Z_k$, where $Z_k$ is the maximal upper-triangular unipotent subgroup of $\GL_k$, which is regarded canonically as a subgroup of $P_k$. 
Define 
\begin{equation}
\label{eq-f-upper-Zk}
\begin{split}
 	f_{\mathcal{W}(\tau,\psi_{Z_k,a}^{-1}),\sigma,s}(g):=\int_{Z_k(F)\backslash Z_k(\A)} f_{\tau,\sigma,s}(zg) \psi_{Z_k,a}(z)dz,
\end{split}
\end{equation}
where the character $\psi_{Z_k,a}$ on $Z_k(\A)$ is given by
\begin{equation}
\label{eq-psi-Zka}
\psi_{Z_k,a}(z)=\psi(-z_{1,2}-\cdots - z_{\beta-1,\zeta} + (-1)^{m+1}\frac{a}{2} z_{\beta,\beta+1} + z_{\beta+1,\beta+2} +\cdots + z_{k-1,k})
\end{equation}
Then 
$f_{\mathcal{W}(\tau,\psi_{Z_k,a}^{-1}),\sigma,s}\in \Ind_{P_{k}(\A)}^{H(\A)}(\mathcal{W}(\tau,\psi_{Z_k,a}^{-1})|\cdot|^{s-\frac{1}{2}} \otimes \sigma)$, where $\mathcal{W}(\tau,\psi_{Z_k,a}^{-1})$ is the $(Z_k,\psi_{Z_k,a}^{-1})$-Whittaker model of $\tau$. 
Then the expression \eqref{eq-unfolding-integral3-klel-6} is equal to
\begin{equation*}
\int_{N_{\ell}^{\eta_2}(\A)\backslash N_{\ell}(\A)} f_{\mathcal{W}(\tau,\psi_{Z_k,a}^{-1}),\sigma,s} (\epsilon_{0,\beta}\eta_2 ug) \psi_{\ell,a}^{-1}(u)du.
\end{equation*}
The following subgroup realizes the quotient $N_{\ell}^{\eta_2}\backslash N_{\ell}$:
$$
\left\{ \begin{pmatrix}
I_{\beta}&x_{1}&x_{2}&x_{3}&0&x_{4}&x_{5}\\
&I_{\tilde{m}-k}&&&&&x'_{4}\\
&&I_{\beta}&&&&0\\
&&&I_{m^\prime-2\tilde{m}}&&&x'_{3}\\
&&&&I_{\beta}&&x'_{2}\\
&&&&&I_{\tilde{m}-k}&x'_{1}\\
&&&&&&I_{\beta}
\end{pmatrix}   \right\},
$$
whose image under the adjoint action of $\epsilon_{0,\beta}\eta_2$ is
\begin{equation}
\label{eq-U-k-eta2}
U_{k,\eta_2}^-:=  \left\{ \begin{pmatrix}
 I_{\beta} & &&&& \\
 &I_{\ell} &&&&\\
 &x_2^\prime &I_{m^\prime-2k} &&\\
 x_1& x_3&x_2&I_{\ell} &\\
 &x_1^\prime &&&I_{\beta}	
 \end{pmatrix}
 \right\},
\end{equation}
which is a subgroup of the unipotent radical $U_k^-$ of the parabolic opposite to $P_k$, 
and we denote the corresponding character over $U_{k,\eta_2}^-$ by
\begin{equation}
\psi_{U_{k,\eta_2}^-}(u)=\psi((x_1)_{\ell, \beta})	
\end{equation}
for $u\in U_{k,\eta_2}^-(\A)$ as in \eqref{eq-U-k-eta2}.
Thus the expression \eqref{eq-unfolding-integral3-klel-6} is equal to  
\begin{equation}
\label{eq-unfolding-integral3-klel-7}
	\int_{U_{k,\eta_2}^-(\A)} f_{\mathcal{W}(\tau,\psi_{Z_k,a}^{-1}),\sigma,s} (u \epsilon_{0,\beta}\eta_2 g) \psi_{U_{k,\eta_2}^-}(u)du.
\end{equation}
Denote 
\begin{equation}
\label{eq-unfolding-J-l-a}
	\mathcal{J}_{\ell,a}({f_{\mathcal{W}(\tau,\psi_{Z_k,a}^{-1}),\sigma,s}})(g):=\int_{U_{k,\eta_2}^-(\A)} f_{\mathcal{W}(\tau,\psi_{Z_k,a}^{-1}),\sigma,s} (u  g) \psi_{U_{k,\eta_2}^-}(u)du.
\end{equation}
Note that $\eta_2 {G_0^\prime} \eta_2^{-1}=\GSpin(W_k)=\GSpin_{m^\prime-2k}$. 
For a fixed $g\in G(\A)$, as a function of $x\in {G_0^\prime}(\A)$, the function 
\begin{equation*}
	\mathcal{J}_{\ell,a}({f_{\mathcal{W}(\tau,\psi_{Z_k,a}^{-1}),\sigma,s}})(\epsilon_{0,\beta} \eta_2 x g) =  \mathcal{J}_{\ell,a}(R( \epsilon_{0,\beta} \eta_2 g) {f_{\mathcal{W}(\tau,\psi_{Z_k,a}^{-1}),\sigma,s}}) (\epsilon_{0,\beta} \eta_2 x \eta_2^{-1} \epsilon_{0,\beta}^{-1})
\end{equation*}
belongs to the space of automorphic representation $\sigma^{\epsilon_{0,\beta}}$, where $R$ denotes the right translation, and $\sigma^{\epsilon_0,\beta}$ is the automorphic representation of $\GSpin_{m^\prime-2k}(\A)$ defined by $\sigma^{\epsilon_{0,\beta}}(h)=\sigma( \epsilon_{0,\beta} h \epsilon_{0,\beta}^{-1})$ for $h\in \GSpin_{m^\prime-2k}(\A)$. Denote $\sigma^\prime$ the automorphic representation of $G_0(\A)$ defined by
\begin{equation}
\label{eq-sigma-prime-1}
\sigma^\prime(h) =	\sigma( \epsilon_{0,\beta} h \epsilon_{0,\beta}^{-1}).
\end{equation}
Thus the pairing
\begin{equation*}
\begin{split}
	&\mathcal{B}^{\psi_{\beta-1,-a}}( \phi_\pi, \mathcal{J}_{\ell,a}(R( \epsilon_{0,\beta} \eta_2 g) {f_{\mathcal{W}(\tau,\psi_{Z_k,a}^{-1}),\sigma,s}}))\\
	=& \int_{C_{{G_0^\prime}}(\A){G_0^\prime}(F)\backslash {G_0^\prime}(\A)} \phi_\pi^{N_{\ell,\beta-1}^{\eta_2}, \psi_{\beta-1,-a}^{-1}}(x g) \mathcal{J}_{\ell,a}({f_{\mathcal{W}(\tau,\psi_{Z_k,a}^{-1}),\sigma,s}})(\epsilon_{0,\beta}\eta_2 x g)  dx
\end{split}
\end{equation*}
defines a Bessel period for the pair $(\pi(g)\phi_\pi, \mathcal{J}_{\ell,a}(R( \epsilon_{0,\beta} \eta_2 g) {f_{\mathcal{W}(\tau,\psi_{Z_k,a}^{-1}),\sigma,s}}))$.
Combining \eqref{eq-unfolding-integral3-klel-5} and  \eqref{eq-unfolding-integral3-klel-7}, we conclude that
\begin{equation*}
\begin{split}
\mathcal{Z}(\phi_\pi, f_{\tau,\sigma, s}) &=   \int_{R_{\ell,\beta-1}^{\eta_2}(\A)\backslash G(\A) }   \int_{C_{{G_0^\prime}}(\A){G_0^\prime}(F)\backslash {G_0^\prime}(\A)} \phi_\pi^{N_{\ell,\beta-1}^{\eta_2}, \psi_{\beta-1,-a}^{-1}}(x g) \mathcal{J}_{\ell,a}({f_{\mathcal{W}(\tau,\psi_{Z_k,a}^{-1}),\sigma,s}})(\epsilon_{0,\beta}\eta_2 x g)  dx dg.\\
&= \int_{R_{\ell,\beta-1}^{\eta_2}(\A)\backslash G(\A) } \mathcal{B}^{\psi_{\beta-1,-a}}(\pi(g) \phi_\pi, \mathcal{J}_{\ell,a}(R( \epsilon_{0,\beta} \eta_2 g) {f_{\mathcal{W}(\tau,\psi_{Z_k,a}^{-1}),\sigma,s}})) dg.
\end{split}
\end{equation*}
This finishes the proof of Theorem~\ref{thm-global-unfolding-kgreaterl}.

As a corollary, we deduce the following important vanishing result from Theorem~\ref{thm-global-unfolding-kgreaterl}.

\begin{corollary}
\label{corollary-vanishing-kgreaterl}
If the Bessel period $\mathcal{B}^{\psi_{\beta-1,-a}}$ for $(\pi, \sigma^\prime)$ is zero, then the global zeta integral $\mathcal{Z}(\phi_\pi, f_{\tau,\sigma, s})$ is zero for all choices of data. 
\end{corollary}

By the uniqueness of the Bessel models for $\GSpin$ groups proved in \cite{Yan2025}, the Bessel period $\mathcal{B}^{\psi_{\beta-1,-a}}( \pi(g) \phi_\pi, \mathcal{J}_{\ell,a}(R( \epsilon_{0,\beta} \eta_2 g) {f_{\mathcal{W}(\tau,\psi_{Z_k,a}^{-1}),\sigma,s}}))$ can be decomposed as
\begin{equation*}
\begin{split}
&\mathcal{B}^{\psi_{\beta-1,-a}}( \pi(g) \phi_\pi, \mathcal{J}_{\ell,a}(R( \epsilon_{0,\beta} \eta_2 g) {f_{\mathcal{W}(\tau,\psi_{Z_k,a}^{-1}),\sigma,s}}))\\
=&\prod_{\nu} \mathcal{B}_{\nu}^{\psi_{\beta-1,-a, \nu}}( \pi_{\nu}(g_{\nu})\phi_{\pi_{\nu}}, \mathcal{J}_{s, \nu} (R( \epsilon_{0,\beta} \eta_2 g_{\nu}) {f_{\mathcal{W}(\tau_{\nu},\psi_{Z_k,a, \nu}^{-1}),\sigma_{\nu},s}}))	
\end{split}
\end{equation*}
for factorizable vectors $\phi_\pi=\otimes_{\nu}\phi_{\pi_{\nu}}$, ${f_{\mathcal{W}(\tau,\psi_{Z_k,a}^{-1}),\sigma,s}}=\otimes_{\nu} {f_{\mathcal{W}(\tau_{\nu},\psi_{Z_k,a,\nu}^{-1}),\sigma_{\nu},s}}$, where at each local place $\nu$, $\mathcal{B}_{\nu}^{\psi_{\beta-1,-a, \nu}}$ is the unique local Bessel functional up to scalar in the space
\begin{equation}
\label{eq-Hom-Bessel}
\Hom_{R_{\ell, \beta-1}^{\eta_2}(F_\nu) }(\pi_{\nu}\otimes\sigma_{\nu}^\prime, \psi_{\beta-1,-a,\nu} ),
\end{equation}
and $\mathcal{J}_{s, \nu}$ is given by the following integration over $U_{k,\eta_2}^-(F_{\nu})$,
\begin{equation}
\label{eq-J-s-nu}
\mathcal{J}_{s, \nu}({f_{\mathcal{W}(\tau_{\nu},\psi_{Z_k,a, \nu}^{-1}),\sigma_{\nu},s}})(g):=\int_{U_{k,\eta_2}^-(F_{\nu})} f_{\mathcal{W}(\tau_{\nu},\psi_{Z_k,a,\nu}^{-1}),\sigma_{\nu},s} (u_{\nu}  g_{\nu}) \psi_{U_{k,\eta_2,\nu}^-}(u_{\nu})du_{\nu}.
\end{equation}

For $\mathrm{Re}(s)$ large enough, we define the local zeta integral for the case $k>\ell$ by
\begin{equation}
\label{eq-local-zeta-integral-k-larger-than-l}
\begin{split}
&\mathcal{Z}_{\nu}(\phi_{\pi_{\nu}},  f_{\mathcal{W}(\tau_{\nu},\psi_{Z_k,a}^{-1}),\sigma_{\nu},s}) :=\\
& \int_{R_{\ell,\beta-1}^{\eta_2}(F_{\nu})\backslash G(F_{\nu}) } \int_{U_{k,\eta_2}^-(F_{\nu})} \mathcal{B}_{\nu}^{\psi_{\beta-1,-a,\nu}}( \pi_{\nu}(g_{\nu})\phi_{\pi_{\nu}},   {f_{\mathcal{W}(\tau_{\nu},\psi_{Z_k,a,\nu}^{-1}),\sigma_{\nu},s}}(u_{\nu} \epsilon_{0,\beta} \eta_2 g_{\nu} ) ) \psi_{U_{k,\eta_2}^-}(u_{\nu}) du_{\nu}  dg_{\nu}.
\end{split}
\end{equation}
Let $\phi_\pi=\otimes_{\nu}\phi_{\pi_{\nu}}$, ${f_{\mathcal{W}(\tau,\psi_{Z_k,a}^{-1}),\sigma,s}}=\otimes_{\nu} f_{\mathcal{W}(\tau_{\nu},\psi_{Z_k,a,\nu}^{-1}), \sigma_{\nu},s}$ be factorizable vectors. Then for  $\mathrm{Re}(s)\gg 0$, we have the Euler product factorization
\begin{equation}
\label{eq-Euler-product-factorization-k-greater-than-l}
\mathcal{Z}(\phi_\pi, f_{\tau,\sigma, s})=\prod_{\nu} \mathcal{Z}_{\nu}(\phi_{\pi_{\nu}},  f_{\mathcal{W}(\tau_{\nu},\psi_{Z_k,a,\nu}^{-1}),\sigma_{\nu},s}).	
\end{equation}

\subsection{Euler product factorization: the case of $k\le \ell$}
\label{subsection-unfolding-klm}
We now prove Theorem~\ref{thm-global-unfolding-klel}.
We assume $k\le \ell$, so that $\beta=0$. Recall $N_{\ell}^{\eta_1}=N_{\ell}\cap \eta_1^{-1} P_{\ell}^{\epsilon_{0,0}} \eta_1$, which is explicitly given by
\begin{equation}
\label{eq-Nl-eta1}
N_{\ell}^{\eta_1}=\left\{\begin{pmatrix}
c&&&&\\
&b&y&z&\\
&&I_{m^\prime-2\ell}&y'&\\
&&&b^{*}&\\
&&&&c^*
\end{pmatrix} \mid b\in Z_{\ell-k}, c\in Z_{k} \right\}.
\end{equation} 
We note that $\epsilon_{0,0} \eta_1 N_{\ell}^{\eta_1} (\epsilon_{0,0} \eta_1)^{-1}= N_{\ell}^{\eta_1}$, and we can decompose $N_{\ell}^{\eta_1}$ as $Z_k N_{k,\ell-k}$, where $Z_k$ is the maximal unipotent subgroup of $\GL(V_k^+)$ which is identified as a subgroup of $H$, and 
\begin{equation}
\label{eq-N-k-l-k}
N_{k,\ell-k}= 	\left\{\begin{pmatrix}
I_k&&&&\\
&b&y&z&\\
&&I_{m^\prime-2\ell}&y'&\\
&&&b^{*}&\\
&&&&I_k
\end{pmatrix} \mid b\in Z_{\ell-k} \right\},
\end{equation}
which is the unipotent subgroup of $\GSpin_{m^\prime-2k}=\GSpin(W_k)$ defined similarly as in \eqref{eq-unipotent-Nl}. The restriction of the character $\psi_{\ell,a}$ to $N_{k,\ell-k}$ is the character defined as in \eqref{eq-psi-l-a} on the unipotent subgroup of $\GSpin_{m^\prime-2k}$, which we denote as $\psi_{m^\prime-2k;\ell-k,a}$. Since $\tau$ is generic, $\tau$ has a $\psi$-Whittaker model $\mathcal{W}(\tau,\psi)$ and we see that
\begin{equation*}
\int_{Z_k(F)\backslash Z_k(\A)} 	f_{\tau,\sigma,s}( cg)\psi^{-1}(c)dc=f_{\mathcal{W}(\tau,\psi),\sigma,s}(g),
\end{equation*}
where 
\begin{equation*}	
f_{\mathcal{W}(\tau,\psi),\sigma,s}\in  
\Ind_{P_{k}(\A)}^{H(\A)}(\mathcal{W}(\tau,\psi)|\cdot|^{s-\frac{1}{2}} \otimes \sigma).
\end{equation*}
We may view $f_{\mathcal{W}(\tau,\psi),\sigma,s}$ as a function
\begin{equation*}
f_{\mathcal{W}(\tau,\psi),\sigma,s}: H(\A)\times \GL_k(\A)\to \sigma,	
\end{equation*}
such that, for each $h\in H(\A)$, the function $a\mapsto f_{\mathcal{W}(\tau,\psi),\sigma,s}(h,a)$ lies in $W(\tau,\psi)\otimes \sigma$. We will often denote $f_{\mathcal{W}(\tau,\psi),\sigma,s}(h)=f_{\mathcal{W}(\tau,\psi),\sigma,s}(h, I_k)$ for simplicity.
Thus the inner integration in \eqref{eq-unfolding-integral3-klel} becomes
\begin{equation*}
\begin{split}
&	\int_{N_{\ell}^{\eta_1}(F)\backslash N_{\ell}^{\eta_1}(\A)}    f_{\tau,\sigma,s}(\epsilon_{0,0} \eta_1 u^\prime    g) \psi_{\ell,a}^{-1}(u^\prime )du^\prime \\
=& \int_{N_{\ell}^{\eta_1}(F)\backslash N_{\ell}^{\eta_1}(\A)}  f_{\tau,\sigma,s}( u^\prime \epsilon_{0,0} \eta_1     g) \psi_{\ell,a}^{-1}(u^\prime )du^\prime \\
=& \int_{N_{k,\ell-k}(F)\backslash N_{k,\ell-k}(\A)}  f_{\mathcal{W}(\tau,\psi),\sigma,s} ( u^\prime \epsilon_{0,0} \eta_1     g)  \psi_{m^\prime-2k;\ell-k,a}^{-1}(u^\prime)du^\prime\\
=& f_{\mathcal{W}(\tau,\psi),\sigma,s}^{N_{k,\ell-k}, \psi_{m^\prime-2k;\ell-k,a}} (  \epsilon_{0,0} \eta_1     g),
\end{split}
\end{equation*}
which is the Bessel coefficient of $f_{\mathcal{W}(\tau,\psi),\sigma,s}$ on the group $\GSpin_{m^\prime-2k}$ with respect to the unipotent subgroup $N_{k,\ell-k}$ and the character $\psi_{m^\prime-2k;\ell-k,a}$. We  denote by
 $$R^{\eta_1}_{k,\ell-k}=G\ltimes N_{k,\ell-k}$$ the Bessel subgroup of $\GSpin_{m^\prime-2k}$.
We deduce that 
\begin{equation}
\label{eq-unfolding-zeta2-1}
	\mathcal{Z}(\phi_\pi, f_{\tau,\sigma, s})= \int_{C_G^0(\A) G(F)\backslash G(\A)}  \phi_\pi( g) \int_{N_{\ell}^{\eta_1}(\A)\backslash N_{\ell}(\A)} f_{\mathcal{W}(\tau,\psi),\sigma,s}^{N_{k,\ell-k}, \psi_{m^\prime-2k;\ell-k,a}} (  \epsilon_{0,0} \eta_1   u  g) \psi_{\ell,a}^{-1}(u)dudg.
\end{equation}

Our next step is to reverse the order of integration. This is justified by the following lemma.
\begin{lemma}
\label{lemma-unfolding-moderate-growth}
The automorphic function on $G(\A)$ defined by
\begin{equation*}
g\mapsto 	\int_{N_{\ell}^{\eta_1}(\A)\backslash N_{\ell}(\A)} f_{\mathcal{W}(\tau,\psi),\sigma,s}^{N_{k,\ell-k}, \psi_{m^\prime-2k;\ell-k,a}} (  \epsilon_{0,0} \eta_1   u  g) \psi_{\ell,a}^{-1}(u)du
\end{equation*}
is of uniformly moderate growth (mod $C_G^0$) on $G(\A)$.
\end{lemma}

\begin{proof}
This follows from the special orthogonal group case; see \cite[Lemma 3.11]{JiangZhang2014}.	
\end{proof}

Because $\phi_\pi$ is rapidly decreasing (mod $C_G^0$), we can interchange the order of integration in \eqref{eq-unfolding-zeta2-1}, and after replacing $u$ by $gug^{-1}$, we obtain that
\begin{equation*}
\begin{split}
	\mathcal{Z}(\phi_\pi, f_{\tau,\sigma, s}) =  \int_{N_{\ell}^{\eta_1}(\A)\backslash N_{\ell}(\A)} \int_{C_G^0(\A) G(F)\backslash G(\A)}  \phi_\pi( g) f_{\mathcal{W}(\tau,\psi),\sigma,s}^{N_{k,\ell-k}, \psi_{m^\prime-2k;\ell-k,a}} ( g  \epsilon_{0,0} \eta_1   u  )dg \psi_{\ell,a}^{-1}(u)du.
\end{split}
\end{equation*}
Note that 
\begin{equation*}
\int_{C_G^0(\A) G(F)\backslash G(\A)}  \phi_\pi( g) f_{\mathcal{W}(\tau,\psi),\sigma,s}^{N_{k,\ell-k}, \psi_{m^\prime-2k;\ell-k,a}} ( g  \epsilon_{0,0} \eta_1   u  )dg= \mathcal{B}^{\psi_{m^\prime-2k;\ell-k,a}} (R(\epsilon_{0,0} \eta_1   u) f_{\mathcal{W}(\tau,\psi),\sigma,s}, \phi_\pi)	
\end{equation*}
is the Bessel period for $(R(\epsilon_{0,0} \eta_1   u) f_{\mathcal{W}(\tau,\psi),\sigma,s}, \phi_\pi)$. Thus we obtain 
\begin{equation}
\begin{split}
	\mathcal{Z}(\phi_\pi, f_{\tau,\sigma, s}) =  \int_{N_{\ell}^{\eta_1}(\A)\backslash N_{\ell}(\A)}  \mathcal{B}^{\psi_{m^\prime-2k;\ell-k,a}} (R(\epsilon_{0,0} \eta_1   u) f_{\mathcal{W}(\tau,\psi),\sigma,s}, \phi_\pi) \psi_{\ell,a}^{-1}(u)du.
\end{split}
\end{equation}
We have proved Theorem~\ref{thm-global-unfolding-klel}.

From Theorem~\ref{thm-global-unfolding-klel}, we immediately obtain the following vanishing result.

\begin{corollary}
\label{corollary-vanishing-klel}
If the Bessel period $\mathcal{B}^{\psi_{m^\prime-2k;\ell-k,a}} $ for $(\sigma, \pi)$ is zero, then the global zeta integral $\mathcal{Z}(\phi_\pi, f_{\tau,\sigma, s})$ is zero for all choices of data. 
\end{corollary}

To keep the notation consistent with Section~\ref{subsection-unfolding-kgm}, we denote
\begin{equation}
\label{eq-sigma-prime-2}
\sigma^\prime=\sigma.	
\end{equation}

Let $\phi_\pi=\otimes_{\nu} \phi_{\pi_{\nu}}$, $f_{\mathcal{W}(\tau,\psi),\sigma,s}=\otimes_{\nu} f_{\mathcal{W}(\tau_{\nu},\psi_{\nu}),\sigma_{\nu},s}$ be factorizable vectors. Then by the uniqueness of Bessel models proved in \cite{Yan2025}, we have
\begin{equation*}
\begin{split}
 \mathcal{B}^{\psi_{m^\prime-2k;\ell-k,a}} (\phi_\pi,  R(\epsilon_{0,0} \eta_1   u) f_{\mathcal{W}(\tau,\psi),\sigma,s}) 	= \prod_{\nu} \mathcal{B}_{\nu}^{\psi_{m^\prime-2k;\ell-k,a,\nu}} ( R(\epsilon_{0,0} \eta_1   u_{\nu} ) f_{\mathcal{W}(\tau_{\nu},\psi_{\nu}),\sigma_{\nu},s}, \phi_{\pi_{\nu}})
\end{split}
\end{equation*}
where $\mathcal{B}_{\nu}^{\psi_{m^\prime-2k;\ell-k,a,\nu}}$ is the unique local Bessel functional up to scalar.
For $\mathrm{Re}(s)\gg 0$, we define the local zeta integral for the case $k\le \ell$ by
\begin{equation}
\label{eq-local-zeta-integral-k-le-l}
\mathcal{Z}_{\nu}(\phi_{\pi_{\nu}},  f_{\mathcal{W}(\tau_{\nu},\psi_{\nu}),\sigma_{\nu},s}) := \int_{N_{\ell}^{\eta_1}(F_{\nu})\backslash N_{\ell}(F_{\nu})} \mathcal{B}_{\nu}^{\psi_{m^\prime-2k;\ell-k,a,\nu}} (R(\epsilon_{0,0} \eta_1   u_{\nu} ) f_{\mathcal{W}(\tau_{\nu},\psi_{\nu}),\sigma_{\nu},s}, \phi_{\pi_{\nu}}) \psi_{\ell, a, \nu}^{-1}(u_{\nu})du_{\nu}. 
\end{equation}
Then for $\mathrm{Re}(s)\gg 0$, we have the following Euler product decomposition
\begin{equation}
\label{eq-Euler-product-factorization-k-le-l}
\mathcal{Z}(\phi_\pi, f_{\tau,\sigma, s})=\prod_{\nu} \mathcal{Z}_{\nu}(\phi_{\pi_{\nu}},  f_{\mathcal{W}(\tau_{\nu},\psi_{\nu}),\sigma_{\nu},s}).	
\end{equation}

\section{The local zeta integrals}
\label{section-local-zeta-integrals}
In this section, we begin to develop the local theory for the local zeta integrals arising from the Euler product factorization of the global zeta integrals. 
 
 To ease notation, we drop the subscript $\nu$ throughout this section. Let $F$ be a local field and let $\psi$ be a non-trivial additive character of $F$.

\subsection{Basic properties of the local integrals}

Let $\pi$, $\sigma$ be irreducible admissible representations of $G(F)=\GSpin_{m^\prime-2\ell-1}(F)$ and $G_0(F)=\GSpin_{m^\prime-2k}(F)$ respectively and let $\tau$ be a generic representation of $\GL_k(F)$. We assume that $\omega_\pi \omega_\sigma=1$, where $\omega_\pi$ and $\omega_\sigma$ are the restriction of the central characters of $\pi$ and $\sigma$ to $F^\times$. 
Let $\sigma^\prime$ be representation defined by
\begin{equation*}
\sigma^\prime(h)=	\begin{cases} 
\sigma( \epsilon_{0,\beta} h \epsilon_{0,\beta}^{-1}) & \mbox{ if }  k>\ell, \\ 
\sigma( h) & \mbox{ if }  k\le \ell.  \\ 
\end{cases} 
\end{equation*}
We assume that 
\begin{equation}
\label{eq-local-Bessel}
\begin{cases}
\Hom_{R_{\ell,\beta-1}^{\eta_2}(F)} (\pi\otimes \sigma^\prime, \psi_{\beta-1,-a})\not=0 &\mbox{ if } k>\ell,\\
\Hom_{R_{k,\ell-k}^{\eta_1}(F)} (\sigma^\prime \otimes \pi, \psi_{m^\prime-2k;\ell-k,a})\not=0 &\mbox{ if } k \le \ell.
\end{cases}
\end{equation}
By \cite{Yan2025}, the Hom spaces are one-dimensional, and 
and let $\bb$ be a non-zero local Bessel functional in the Hom space \eqref{eq-local-Bessel} for each case. 
Consider the normalized induced representation
\begin{equation*}
\rho_{\tau,\sigma,s}=	\Ind_{P_{k}(F)}^{H(F)}(\mathcal{W}(\tau)|\cdot|^{s-\frac{1}{2}} \otimes \sigma)
\end{equation*}
where 
\begin{equation*}
\mathcal{W}(\tau)=
\begin{cases}
\mathcal{W}(\tau,\psi_{Z_k,a}^{-1}) & \mbox{ if } k>\ell,\\
\mathcal{W}(\tau,\psi^{-1}) & \mbox{ if } k\le \ell,
\end{cases}
\end{equation*}
is the Whittaker model of $\tau$ with respect to the character $\psi_{Z_k,a}$ defined analogously as in \eqref{eq-psi-Zka} if $k>\ell$ or the character $\psi^{-1}$ if $k\le \ell$. Let $f_{\mathcal{W}(\tau),\sigma, s} \in \rho_{\tau,\sigma,s}$ be a smooth section. We view  $f_{\mathcal{W}(\tau),\sigma, s} $ as a function 
$$
f_{\mathcal{W}(\tau),\sigma, s}: H(F)\times \GL_k(F) \to V_\sigma,
$$
such that, for each $h\in H(F)$, the function $a\mapsto f_{\mathcal{W}(\tau),\sigma, s}(h,a)$ lies in $\mathcal{W}(\tau)\otimes V_\sigma$. We will often denote $f_{\mathcal{W}(\tau),\sigma, s}(h)=f_{\mathcal{W}(\tau),\sigma, s}(h,I_k)$ for ease of notation. 

Let $v_\pi\in V_\pi$.
The local zeta integrals defined in \eqref{eq-local-zeta-integral-k-larger-than-l} and \eqref{eq-local-zeta-integral-k-le-l} are
\begin{equation}
\label{eq-local-zeta}
\mathcal{Z}(v_\pi,	f_{\mathcal{W}(\tau),\sigma, s}) :=
\begin{cases}
	\int_{R_{\ell,\beta-1}^{\eta_2}(F)\backslash G(F) } \int_{U_{k,\eta_2}^-(F)} \bb( \pi(g) v_\pi,   {f_{\mathcal{W}(\tau),\sigma,s}}(u \epsilon_{0,\beta} \eta_2 g ) ) \psi_{U_{k,\eta_2}^-}(u) du  dg &\mbox{ if } k>\ell,\\
	\int_{N_{\ell}^{\eta_1}(F)\backslash N_{\ell}(F)} \bb( v_\pi,  R(\epsilon_{0,0} \eta_1   u ) f_{\mathcal{W}(\tau),\sigma, s}) \psi_{\ell,a}^{-1}(u)du &\mbox{ if } k\le \ell.
\end{cases}
\end{equation}
Note that the local zeta integral $\mathcal{Z}(v_\pi,	f_{\mathcal{W}(\tau),\sigma, s})$, when it exists, is an element of 
\begin{equation}
\label{eq-Hom-local}
\Hom_{R_{\ell,a}(F)}	(\rho_{\tau,\sigma,s} \otimes \pi , \psi_{\ell,a}) . 
\end{equation}
Again, by the uniqueness of local Bessel models \cite{Yan2025}, the Hom space~\eqref{eq-Hom-local} is at most one-dimensional, when $\rho_{\tau,\sigma,s}$ is irreducible.

\begin{lemma}
The local zeta integral $\mathcal{Z}(v_\pi,	f_{\mathcal{W}(\tau),\sigma, s})$ converges absolutely in a right half plane, which depends only on the representations $\pi, \sigma, \tau$.
\end{lemma}

\begin{proof}
Since we mod out $R_{\ell,\beta-1}^{\eta_2}(F)$ which contains $C_G^0(F)$ in the case $k>\ell$, and unipotent subgroups for $\GSpin_{m^\prime}$ and $\SO_{m^\prime}$ are the same in the case $k\le  \ell$, the  absolute convergence follows from the the case of special orthogonal groups \cite{JiangSoudryZhang} (see also \cite[Lemma 4.1]{JiangZhang2014}, \cite[Theorem 3.1]{Soudry2017Israel}, \cite[Theorem 3.5]{Soudry2018JNT}).
\end{proof}

\subsection{Unramified representations and local $L$-functions of $\GSpin\times \GL$ groups}
From now on, we 
let $F$ be a nonarchimedean local field of characteristic zero with ring of integers $\mathcal{O}$. We fix a uniformizer $\varpi\in \mathcal{O}$ and let $q=|\mathcal{O}/\varpi\mathcal{O}|$ be the cardinality of the residue field. We normalize the absolute value $|\cdot |=|\cdot|_F$ on $F$ by $|\varpi|=q^{-1}$. Also, let $\psi$ be an additive character of $F$ which is unramified. 

Let $\GG(F)$ be either the $F$-points of a reductive group $\GG$ of interest, which is either a general linear group or a $\GSpin$ group, with $B_{\GG}=T_{\GG}U_{\GG}$ a fixed Borel subgroup and $T_{\GG}$ a maximal torus. Let $\Pi=\Ind_{B_{\GG}(F)}^{\GG(F)}(\mu)$ be an irreducible admissbile unramified representation of $\GG(F)$ where $\mu$ is a character of $T_{\GG}(F)$, and let $t_\Pi$ be the semi-simple conjugacy class in the complex dual group parametrizing $\Pi$. When $\GG=\GL_k$, we have $\mu=\chi_1\otimes\cdots \otimes \chi_k$, with $\chi_i$ unramified quasi-characters of $F^\times$. When $\GG=\GSpin_{2n+1}$ or $\GG(F)=\GSpin_{2n}$ (split), we have $\mu=\chi_0\otimes\chi_1\otimes\cdots \chi_n$, where $\chi_0$ is the restriction of the central character to $C_{\GG(F)}^0$. Then a representative $t_\Pi$ for the semi-simple conjugacy class associated to $\Pi$ is given by
\begin{equation}
\label{eq-t-Pi}
t_{\Pi} = \begin{cases} 
\operatorname{diag}\left( \chi_1(\varpi), \dots, \chi_k(\varpi)  \right) \in \GL_k(\C) & \mbox{ if } \GG = \GL_k, \\ 
\operatorname{diag}\left( \chi_1(\varpi), \dots, \chi_n(\varpi), \chi^{-1}_n \chi_0(\varpi),\dots, 
\chi^{-1}_1 \chi_0(\varpi)  \right) \in \GSp_{2n}(\C) & \mbox{ if } \GG = \GSpin_{2n+1}, \\ 
\operatorname{diag}\left( \chi_1(\varpi), \dots, \chi_n(\varpi), \chi^{-1}_n \chi_0(\varpi),\dots, 
\chi^{-1}_1 \chi_0(\varpi)  \right) \in \GSO_{2n}(\C) & \mbox{ if } \GG = \GSpin_{2n} \text{ (split)}.  \\ 
\end{cases} 
\end{equation}

 Now we assume that $\GG=\GSpin_{2n}^a$ with a non-square $a\in F^\times$  is quasi-split but not split over $F$, and let $E=F(\sqrt{a})$ be a quadratic extension of $F$ over which $\GG$ splits. Then the maximal torus of $\GG(F)$ is given by 
\[ T_{\GG} (F)  = F^\times \times \left( F^\times \right)^{n-1} \times \GSpin_2^a(F) \] 
with $\GSpin_2^a(F) = \left(\operatorname{Res}_{E/F} \GL_1\right)(F) = E^\times$. 
The unramified character $\mu$ of $T_{\GG}(F)$ can be written as 
$$\mu =  \chi_0 \otimes \chi_1 \otimes \cdots \chi_{n-1} \otimes \chi\circ\operatorname{Norm}_{E/F}, 
$$
and we may identify $t_\Pi \in {}^L{\GG} \cong \GSO_{2n}(\C) \rtimes \operatorname{Gal}(E/F)$ with 
\[ 
t_\Pi =  \operatorname{diag}\left( \chi_1(\varpi), \dots, \chi_{n-1}(\varpi), 
\begin{pmatrix} 
\alpha & \beta a \\ 
\beta & \alpha
\end{pmatrix}, 
\chi^{-1}_{n-1} \chi_0(\varpi),\dots, \chi^{-1}_1 \chi_0(\varpi)  \right) \in \GL_{2n}(\C), 
\]  
with $\alpha^2 - a \beta^2 = \chi_0(\varpi)$.

Let $\pi$, $\tau$, $\sigma$ be irreducible admissible unramified representations of $\GSpin_{m^\prime-2\ell-1}(F)$, $\GL_k(F)$, $G_{m^\prime-2k}(F)$ respectively. Then the local tensor product $L$-function associated to $\pi$ and $\tau$ is 
\begin{equation*}
L(s,\pi\times\tau)	= \det(1-q^{-s} t_{\pi}\otimes t_{\tau})^{-1}.
\end{equation*}
Similarly, the local tensor product $L$-function associated to $\sigma$ and $\tau$ is 
\begin{equation*}
L(s,\sigma\times\tau)	= \det(1-q^{-s} t_{\sigma}\otimes t_{\tau})^{-1}.
\end{equation*}
Moreover, for any finite dimensional representation $\kappa$ of $\GL_k(\C)$, define
\begin{equation*}
L(s,\tau,\kappa)=\det(1-q^{-s} \kappa(t_\tau))^{-1}.	
\end{equation*}

\subsection{Unramified computation}
\label{subsection-unramified-computation}
The goal of this section is to compute the local integral when all data are unramified.

Let $v_{\pi}^0\in V_\pi$, $v_\sigma^0\in V_\sigma$ be non-zero unramified vectors, and we normalize the local Bessel functional $\bb$ appearing in the local zeta integral by
\begin{equation*}
\bb(v_\pi^0, v_\sigma^0)=1. 	
\end{equation*}
Let $f_{\mathcal{W}(\tau),\sigma, s}^0$ be the unramified section in the space of $\rho_{\tau,\sigma,s}$ such that, for all $a\in \GL_k(F)$, 
$$
f_{\mathcal{W}(\tau),\sigma, s}^0(e,a)=W_{\tau}^0(a) v_\sigma^0,
$$ 
where $e\in H(F)$ is the identity element, and $W_{\tau}^0\in \mathcal{W}(\tau)$ is the unramified normalized Whittaker function such that $W_\tau^0(I_k)=1$.

\begin{theorem}
\label{thm-local-unramified-computation}
With all date being unramified and normalized as above, we have 
\begin{equation}
\label{eq-local-unramified-computation}
	\mathcal{Z}(v_\pi^0,	f_{\mathcal{W}(\tau),\sigma, s}^0)
	 =  \frac{L(s,\pi \times \tau)}{L(s+\frac{1}{2},\sigma  \times \tau \otimes\omega_\pi) L\left(2s, \tau, \rho \otimes \omega_{\pi}\right)}
\end{equation}
where
\begin{equation}
\label{eq-rho}
\rho=
\begin{cases} 
\wedge^2, & 
\mbox{ if $G$ is an odd $\GSpin$ group,}
\\
\Sym^2, & 
\mbox{ if $G$ is an even $\GSpin$ group.}
\end{cases} 
\end{equation}
\end{theorem}

\begin{proof}
We assume $k>\ell$. The case $k\le \ell$ is similar and omitted. 
Since $\omega_\pi$ and $\omega_\sigma$ are unramified, there exist unique unramified square roots $\omega_\pi^{1/2}$ and $\omega_\sigma^{1/2}$. Note that $\omega_\sigma=\omega_\pi^{-1}$. Denote
\begin{equation*}
\overline{\pi}=\pi\otimes \omega_\pi^{-1/2}\circ N, \quad 	\overline{\sigma}=\sigma\otimes \omega_\sigma^{-1/2}\circ N,
\end{equation*}
which have trivial central characters and hence can be viewed as representations of $\SO_{m^\prime-2\ell-1}(F)$ and $\SO_{m^\prime-2k}(F)$ respectively. Denote
$\overline{\tau}=\tau\otimes \omega_\pi^{1/2}$,
and let $\overline{f}_{\mathcal{W}(\overline{\tau}),\overline{\sigma}, s}^0$ be the unramified section in the space of $\Ind_{P_{k}(F)}^{H(F)}(\mathcal{W}(\overline{\tau})|\cdot|^{s-\frac{1}{2}} \otimes \overline{\sigma})$ such that, for all $a\in \GL_k(F)$, $\overline{f}_{\mathcal{W}(\overline{\tau}),\overline{\sigma}, s}^0(e,a)=W_{\tau}^0(a) v_\sigma^0$. Then
\begin{equation*}
\int_{U_{k,\eta_2}^-(F)} \bb( \pi(g) v_\pi^0,   {f_{\mathcal{W}(\tau),\sigma,s}^0}(u \epsilon_{0,\beta} \eta_2 g ) ) \psi_{U_{k,\eta_2}^-}(u) du = \int_{U_{k,\eta_2}^-(F)} \bb( \overline{\pi}(g) v_\pi^0,   \overline{f}_{\mathcal{W}(\overline{\tau}),\overline{\sigma}, s}^0 (u \epsilon_{0,\beta} \eta_2 g ) ) \psi_{U_{k,\eta_2}^-}(u) du.	
\end{equation*}
Now observe that $R_{\ell, \beta-1}^{\eta_2}$ is the Bessel subgroup of $G$ containing $C_G^0$, given by $R_{\ell, \beta-1}^{\eta_2}=G_0^\prime\ltimes N_{\ell, \beta-1}^{\eta_2}$, where $G_0^\prime=\eta_2^{-1}G_0\eta_2$. It follows that
\begin{equation*}
R_{\ell, \beta-1}^{\eta_2}\backslash G \cong \overline{R}_{\ell, \beta-1}	 \backslash \SO_{m^\prime-2\ell-1},
\end{equation*}
where $\overline{R}_{\ell, \beta-1}	 =  \pr(\eta_2)^{-1} \SO_{m^\prime-2k} \pr(\eta_2) \ltimes N_{\ell,\beta-1}^{\eta_2}$ is a Bessel subgroup of $\SO_{m^\prime-2\ell-1}$.
By the local unramified computation for special orthogonal groups \cite{JiangSoudryZhang} (see also \cite[Theorem 4.8]{JiangZhang2014}), we have
\begin{equation*}
\begin{split}
	\mathcal{Z}(v_\pi^0,	f_{\mathcal{W}(\tau),\sigma, s}^0) 
	& =  \int_{R_{\ell,\beta-1}^{\eta_2}(F)\backslash G(F) } \int_{U_{k,\eta_2}^-(F)} \bb( \pi(g) v_\pi^0,   {f_{\mathcal{W}(\tau),\sigma,s}^0}(u \epsilon_{0,\beta} \eta_2 g ) ) \psi_{U_{k,\eta_2}^-}(u) du  dg \\
	&=  \int_{R_{\ell,\beta-1}^{\eta_2}(F)\backslash G(F)}  \int_{U_{k,\eta_2}^-(F)} \bb( \overline{\pi}(g) v_\pi^0,   \overline{f}_{\mathcal{W}(\overline{\tau}),\overline{\sigma}, s}^0 (u \epsilon_{0,\beta} \eta_2 g ) ) \psi_{U_{k,\eta_2}^-}(u) du  \\
	&= \frac{L(s, \overline{\pi} \times \overline{\tau})}{L(s+\frac{1}{2},\overline{\sigma}  \times \overline{\tau}) L\left(2s, \overline{\tau}, \rho \right)} .
\end{split}
\end{equation*}
Finally, we compute that
\begin{equation*}
	L(s, \overline{\pi} \times \overline{\tau})=L(s, \pi\otimes \omega_\pi^{-1/2}\times \tau\otimes \omega_\pi^{1/2})= L(s, \pi\times\tau),
\end{equation*}
\begin{equation*}
L(s+\frac{1}{2},\overline{\sigma}  \times \overline{\tau}) =L(s+\frac{1}{2},\sigma\otimes \omega_\pi^{1/2}  \times \tau\otimes \omega_\pi^{1/2})=L(s+\frac{1}{2}, \sigma\times\tau\otimes \omega_\pi),
\end{equation*}
and for $\rho=\wedge^2$ (or $\Sym^2$), we have
\begin{equation*}
	L\left(2s, \overline{\tau}, \rho \right)=L(2s, \tau\otimes\omega_\pi^{\frac{1}{2}}, \rho)=L(2s, \tau, \rho\otimes\omega_\pi). 
\end{equation*}
This completes the proof of \eqref{eq-local-unramified-computation}.
\end{proof}

\begin{remark}
We will address some other properties of the local zeta integrals in Section~\ref{subsection-GGP-Normalization-Local-Zeta}. 
\end{remark}

\section{Application to the global Gan-Gross-Prasad conjecture}
\label{section-GGP}

The purpose of this section is to prove one direction of the global Gan-Gross-Prasad conjecture for generic cuspidal automorphic representations on $\GSpin$ groups. 
Throughout this section we let $k>\ell$ and $\beta=k-\ell$.

\subsection{Functoriality for generic representations of $\GSpin$ groups}
We start with the following definition.

\begin{definition}
\label{defn-twisted-self-dual}
Let $\eta$ be a Hecke character. We say that an irreducible cuspidal automorphic representation $\Pi$ of $\GL_n(\A)$ is $\eta$-self-dual if
\begin{equation*}
\wt\Pi =\Pi\otimes \eta,	
\end{equation*}
where $\wt\Pi$ is the contragredient of $\Pi$.
\end{definition}

The following description on the image of functoriality for generic representations on quasi-split $\GSpin$ groups is due to Asgari and Shahidi \cite{AsgariShahidi2006, AsgariShahidi2014}.

\begin{theorem}\cite[Theorem 4.26]{AsgariShahidi2014}
\label{thm-functoriality}
Let $G_n$ be either the split $\GSpin_{2n+1}$, the split $\GSpin_{2n}$, or one of the non-split quasi-split groups $\GSpin_{2n}^a$ (associated with a quadratic extension $F(a)/F$ for some non-square $a\in F^\times$). 
Let $\pi$ be an irreducible generic cuspidal automorphic representation of $G_n(\A)$.
Then
$\pi$ has a unique functorial transfer to an automorphic representation $\Pi$ of $\GL_{2n}(\A)$. The representation $\Pi$ satisfies $\Pi\cong \wt\Pi \otimes \omega_\pi$.  The representation $\Pi$ is an isobaric sum of the form
\begin{equation}
\label{eq-GGP-Pi}
\Pi= \Ind(\Pi_1\otimes \cdots  \otimes \Pi_k )=\Pi_1 \boxplus \cdots \boxplus \Pi_r
\end{equation}
where each $\Pi_i$ is an irreducible unitary cuspidal automorphic representation of $\GL_{n_i}(\A)$ such that for each $1\le i \le r$, the complete $L$-function $L(s,\Pi_i,\wedge^2\otimes\omega_\pi^{-1})$ has a pole at $s=1$ in the odd case and $L(s,\Pi_i,\Sym^2\otimes\omega_\pi^{-1})$ has a pole at $s=1$ in the even case. Moreover, $\Pi_i\cong \wt\Pi_i\otimes\omega_\pi$ for each $i$, $\Pi_i\not\cong \Pi_j$ if $i\not=j$, and $n_1+\cdots +n_r=2n$. Conversely, any automorphic representation $\Pi$ of $\GL_{2n}(\A)$ satisfying the above conditions is a functorial transfer of some irreducible generic cuspidal automorphic representation $\pi$ of $G_n(\A)$.
\end{theorem}

\begin{remark}
We remark that recently Cai, Friedberg and Kaplan in \cite{CaiFriedbergKaplan2024} established the Langlands functoriality for any cuspidal automorphic representations, generic or not, for split $\GSpin$ groups, using the generalized doubling method \cite{CaiFriedbergGinzburgKaplan2019} and the converse theorem \cite{CogdellPS1994, CogdellPS1999}. 	
\end{remark}

\subsection{Residue of the Eisenstein series}
Let $\sigma$ be irreducible generic cuspidal automorphic representations of $\GSpin_{m^\prime-2\ell-1}(\A)$.
Let $\tau=\tau_1\boxtimes \tau_2 \boxtimes \cdots \boxtimes \tau_{r}$ be an irreducible unitary generic isobaric automorphic representation of $\GL_k(\A)$ associated to distinct $\tau_1, \cdots, \tau_r$, such that $\tau_i$ is $\omega_\sigma^{-1}$-self-dual for each $1\le i \le r$.
The constant term of the Eisenstein series $E(g, f_{\tau,\sigma, s})$ along a standard parabolic subgroup $P$ is always zero unless $P=P_k$ (\cite[II.1.7]{MoglinWaldspurger1995}). In this case, we have
\begin{equation*}
\begin{split}
E_{P_k}(g, 	f_{\tau,\sigma, s}) &=  \int_{U_k(F)\backslash U_k(\A)} E(ug, 	f_{\tau,\sigma, s}) du\\
&= f_{\tau,\sigma,s}(g)+ \M(\omega_0,\tau\otimes\sigma,s)(f_{\tau,\sigma,s})(g),
\end{split}
\end{equation*}
where  $\M(\omega_0,\tau\otimes\sigma,s)$ is the intertwining operator from the induced representation 
\begin{equation}
\label{eq-induced-rep-global}
\Ind_{P_{k}(\A)}^{H(\A)}(\tau|\cdot|^{s-\frac{1}{2}} \otimes \sigma)	
\end{equation}
to 
the induced representation 
$$\Ind_{P_{k}(\A)}^{H(\A)}(\wt \tau \otimes \omega_\sigma^{-1}|\cdot|^{-s+\frac{1}{2}} \otimes \sigma)$$ 
given by the following integral
\begin{equation*}
\M(\omega_0,\tau\otimes\sigma,s)(f_{\tau,\sigma,s})(g) = \int_{U_k(F)\backslash U_k(\A)} 	 f_{\tau,\sigma,s} (\omega_0^{-1} u g) du.
\end{equation*}
Here $\omega_0$ is the Weyl group element which takes $U_k$ to its opposite $U_k^-$.
The term that carries the highest order of the pole at $s=1$ is given by $\M(\omega_0,\tau\otimes\sigma,s)(f_{\tau,\sigma,s})(g)$.
Since factorizable sections generate a dense subspace in the induced representation \eqref{eq-induced-rep-global}, it suffices to consider factorizable sections for the existence of poles of $E(g, 	f_{\tau,\sigma, s})$, or the existence of poles of $\M(\omega_0,\tau\otimes\sigma,s)(f_{\tau,\sigma,s})(g)$.

For a factorizable section $f_{\tau,\sigma,s}=\otimes f_{\tau_\nu,\sigma_\nu,s}$ where $f_{\tau_\nu,\sigma_\nu,s}$ is a section in $\Ind_{P_{k}(F_\nu)}^{H(F_\nu)}(\tau_\nu|\cdot|^{s-\frac{1}{2}} \otimes \sigma_\nu)$ and is unramified at almost all finite local places $\nu$, the term $\M(w_0,\tau\otimes\sigma,s)(f_{\tau,\sigma,s})(g)$ can be expressed as a product
\begin{equation*}
	\M(\omega_0,\tau\otimes\sigma,s)(f_{\tau,\sigma,s})(g)=\prod_{\nu} \M(\omega_0,\tau\otimes\sigma,s)_\nu (f_{\tau_\nu,\sigma_\nu,s}) (g_\nu)
\end{equation*}
where $\M(\omega_0,\tau\otimes\sigma,s)_\nu$ is the local intertwining operator at the place $\nu$.
Following the Langlands-Shahidi method (see \cite{Langlands1971EulerProducts} and \cite{Shahidi2010book}), the term $\M(\omega_0,\tau\otimes\sigma,s)(f_{\tau,\sigma,s})$ can be expressed as 
\begin{equation*}
	\M(\omega_0,\tau\otimes\sigma,s)_S (f_{\tau_S,\sigma_S,s})  \cdot \frac{L^S(s-\frac{1}{2}, \sigma\times\tau\otimes \omega_\sigma^{-1})L^S(2s-1, \tau, \rho\otimes\omega_\sigma^{-1})}{L^S(s+\frac{1}{2},\sigma\times\tau\otimes \omega_\sigma^{-1})  L^S(2s, \tau, \rho\otimes\omega_\sigma^{-1})} f_{\omega_0(\tau\otimes\sigma),s}^S
\end{equation*}
where $\rho$ is given by \eqref{eq-rho}, and 
\begin{equation*}
	\M(\omega_0,\tau\otimes\sigma,s)_S=\prod_{\nu\in S} \M(\omega_0,\tau\otimes\sigma,s)_\nu, \quad f_{\tau_S,\sigma_S,s}=\prod_{\nu\in S} f_{\tau_\nu,\sigma_\nu,s}, \quad f_{\omega_0(\tau\otimes\sigma),s}^S=\prod_{\nu\not\in S} f_{\omega_0(\tau_\nu\otimes\sigma_\nu),s}
\end{equation*}
with 
$$f_{\omega_0(\tau_\nu\otimes\sigma_\nu),s}\in \Ind_{P_{k}(F_\nu)}^{H(F_\nu)}(\wt \tau_\nu \omega_{\sigma_\nu}^{-1} |\cdot|^{-s+\frac{1}{2}} \otimes \sigma_\nu).$$
Then we define the normalized local intertwining operator $\NN(\omega_0,\tau\otimes\sigma,s)_{\nu}$ 
by the following formula
\begin{equation*}
\NN(\omega_0,\tau\otimes\sigma,s)_{\nu}=\beta_\nu(s,\tau,\sigma,\psi;\rho)\cdot \M(\omega_0,\tau\otimes\sigma,s)_\nu,
\end{equation*}
where the local normalization factor $\beta_\nu(s,\tau,\sigma,\psi;\rho)$ is defined to be
\begin{equation*}
	\frac{L(s+\frac{1}{2},\sigma_\nu\times\tau_\nu\otimes\omega_{\sigma_\nu}^{-1}) L(2s, \tau_\nu, \rho\otimes\omega_{\sigma_\nu}^{-1}) \epsilon(s-\frac{1}{2}, \sigma_\nu\times\tau_\nu\otimes\omega_{\sigma_\nu}^{-1},\psi_\nu) \epsilon(2s-1, \tau_\nu, \rho\times\omega_{\sigma_\nu}^{-1},\psi_v)} {L(s-\frac{1}{2},\sigma_\nu\times\tau_\nu\otimes\omega_{\sigma_\nu}^{-1}) L(2s-1, \tau_\nu, \rho\otimes\omega_{\sigma_\nu}^{-1})}.
\end{equation*}
Then we obtain 
\begin{equation}
\label{eq-GGP-M-intertwining-global}
\begin{split}
\M(w_0,\tau\otimes\sigma,s)= \frac{\NN(\omega_0,\tau\otimes\sigma,s) \cdot L(s-\frac{1}{2},\sigma\times\tau\otimes\omega_{\sigma}^{-1}) L(2s-1, \tau, \rho\otimes\omega_{\sigma}^{-1}) }{L(s+\frac{1}{2},\sigma\times\tau\otimes\omega_{\sigma}^{-1}) L(2s, \tau, \rho\otimes\omega_{\sigma}^{-1}) \epsilon(s-\frac{1}{2}, \sigma\times\tau\otimes\omega_{\sigma}^{-1}) \epsilon(2s-1, \tau, \rho\times\omega_{\sigma}^{-1})}. 
\end{split}
\end{equation}

\begin{proposition}
\label{prop-GGP-NormalizedIntertwining-HolomorphyNonvanishing}
Let $\sigma$ be irreducible generic cuspidal automorphic representations of $\GSpin_{m^\prime-2\ell-1}(\A)$.
Let $\tau=\tau_1\boxtimes \tau_2 \boxtimes \cdots \boxtimes \tau_{r}$ be an irreducible unitary generic isobaric automorphic representation of $\GL_k(\A)$ associated to distinct $\tau_1, \cdots, \tau_r$, such that $\tau_i$ is $\omega_\sigma^{-1}$-self-dual for each $1\le i \le r$.
 Then for each local place $\nu$  of $F$, the normalized local intertwining operator $\NN(\omega_0,\tau\otimes\sigma,s)_{\nu}$ from the induced representation $\Ind_{P_{k}(F_\nu)}^{H(F_\nu)}(\tau_\nu|\cdot|^{s-\frac{1}{2}} \otimes \sigma_\nu)$ to $\Ind_{P_{k}(F_\nu)}^{H(F_\nu)}(\wt \tau_\nu \omega_{\sigma_\nu}^{-1} |\cdot|^{-s+\frac{1}{2}} \otimes \sigma_\nu)$ is holomorphic and non-zero for $\Re(s)\ge 1$.
\end{proposition}

\begin{proof}
This follows from \cite[Proposition 3.6]{AsgariShahidi2006}.	
\end{proof}

\begin{proposition}
\label{prop-GGP-EisensteinSeriesPoles}
Let $\sigma$ be an irreducible generic cuspidal automorphic representation of $\GSpin_{m^\prime-2\ell-1}(\A)$.
Let $\tau=\tau_1\boxtimes \tau_2 \boxtimes \cdots \boxtimes \tau_{r}$ be an irreducible unitary generic isobaric automorphic representation of $\GL_k(\A)$ associated to distinct $\tau_1, \cdots, \tau_r$, such that $\tau_i$ is $\omega_\sigma^{-1}$-self-dual for each $1\le i \le r$.
Then we have the following.
\begin{enumerate}
\item[(i)] The $L$-function $L(s,\sigma\times\tau \otimes \omega_\sigma^{-1})$ is holomorphic at $s=\frac{1}{2}$.
\item[(ii)] The Eisenstein series $E(\cdot, f_{\tau,\sigma, s})$ can possibly have a pole at $s=1$ of order at most $r$.
\item[(iii)] The Eisenstein series $E(\cdot, f_{\tau,\sigma, s})$ has a pole at $s=1$ of order $r$ if and only if $L(s,\tau_i,\rho\otimes\omega_\sigma^{-1})$ has a pole at $s=1$ for $i=1, \cdots, r$ and $L(s,\sigma\times\tau \otimes \omega_\sigma^{-1})$ is non-zero at $s=\frac{1}{2}$. Here, we recall that $\rho$ is given by \eqref{eq-rho}.
\end{enumerate}
\end{proposition}

\begin{proof}
We first prove (i). By Theorem~\ref{thm-functoriality}, $\sigma$ has a functorial transfer to an isobaric automorphic representation $\Pi_\sigma$ of $\GL_N(\A)$ for some $N$. It follows that
\begin{equation*}
L(s,\sigma\times\tau \otimes \omega_\sigma^{-1})= L(s, \Pi_\sigma \times 	\tau \otimes \omega_\sigma^{-1}) 
\end{equation*}
is holomorphic at $s=\frac{1}{2}$. 	

Now we consider (ii) and (iii). By Proposition~\ref{prop-GGP-NormalizedIntertwining-HolomorphyNonvanishing}, the normalized global intertwining operator $\NN(\omega_0,\tau\otimes\sigma,s)$ is holomorphic and non-zero for $\Re(s)\ge 1$. The terms in the denominator of the right-hand side of \eqref{eq-GGP-M-intertwining-global} do not have a pole at $s=1$. Thus it suffices to consider the possible poles of $L(2s-1, \tau, \rho\otimes\omega_{\sigma}^{-1})$ at $s=1$. Note that
\begin{equation}
\label{eq-GGP-LTauRho}
	L(2s-1, \tau, \rho\otimes\omega_{\sigma}^{-1})=\prod_{i=1}^r L(2s-1, \tau_i, \rho\otimes\omega_{\sigma}^{-1}) \cdot \prod_{1\le i <j\le r} L(2s-1, \tau_i\times\tau_j\otimes \omega_{\sigma}^{-1}).
\end{equation}
Since each $\tau_j$ is $\omega_\sigma^{-1}$-self-dual and $\tau_i\not=\tau_j$ for $i\not=j$, we conclude that
\begin{equation*}
\prod_{1\le i <j\le r} L(2s-1, \tau_i\times\tau_j\otimes \omega_{\sigma}^{-1})=\prod_{1\le i <j\le r} L(2s-1, \tau_i\times \wt\tau_j)	
\end{equation*}
is holomorphic at $s=1$. The $L$-function $L(2s-1, \tau_i, \rho\otimes\omega_{\sigma}^{-1}) $ may have a simple pole at $s=1$ for each $1\le i \le r$. 
Thus, $L(2s-1, \tau, \rho\otimes\omega_{\sigma}^{-1})$ can possibly have a pole at $s=1$ of order at most $r$. This implies the global intertwining operator $\M(\omega_0,\tau\otimes\sigma,s)$ can possibly have a pole at $s=1$ of order at most $r$. Thus the Eisenstein series $E(\cdot, f_{\tau,\sigma, s})$ can possibly have a pole at $s=1$ of order at most $r$. This proves (ii). Note that $L(s-\frac{1}{2},\sigma\times\tau \otimes \omega_\sigma^{-1})$ is holomorphic at $s=1$, but may have zero at $s=1$. It follows that the Eisenstein series $E(\cdot, f_{\tau,\sigma, s})$ has a pole at $s=1$ of order exactly $r$ if and only if $L(s,\tau_i,\rho\otimes\omega_\sigma^{-1})$ has a pole at $s=1$ for $i=1, \cdots, r$ and $L(s,\sigma\times\tau \otimes \omega_\sigma^{-1})$ is non-zero at $s=\frac{1}{2}$.
This proves (iii).
\end{proof}

When the Eisenstein series $E(\cdot, f_{\tau,\sigma, s})$ has a pole at $s=1$ of order $r$, we denote by $\Eisen_{\tau\otimes\sigma}$ the $r$-th iterated residue at $s=1$ of $E(\cdot, f_{\tau,\sigma, s})$.

\subsection{Reciprocal non-vanishing of Bessel periods}
\label{subsection-Reciprocal-Nonvanishing-Bessel}

Let $\pi$ be irreducible generic cuspidal automorphic representations of $\GSpin_{m^\prime-2k}(\A)$ such that $\omega_\pi \omega_\sigma=1$. By Theorem~\ref{thm-functoriality}, $\pi$ has a unique functorial transfer $\Pi$ as in \eqref{eq-GGP-Pi}. Let $\tau=\Pi\otimes\omega_\pi^{-1}$. Then
\begin{equation*}
	\tau=\tau_1\boxtimes \tau_2 \boxtimes \cdots \boxtimes \tau_{r}
\end{equation*}
with $\tau_i=\Pi_i\otimes\omega_\pi^{-1}$, 
is an irreducible unitary generic isobaric automorphic representation of $\GL_k(\A)$ associated to distinct $\tau_1, , \cdots, \tau_r$, with $k=m^\prime-2k$ if $m^\prime-2k$ is even and $k=m^\prime-2k-1$ if $m^\prime-2k$ is odd. Then for each $1\le i \le r$, we have
\begin{equation*}
\wt \tau_i=	\wt \Pi_i\otimes\omega_\pi =\Pi_i =\tau_i\otimes \omega_\pi= \tau_i\otimes \omega_{\sigma}^{-1}.
\end{equation*}
Hence $\tau_i$ is $\omega_\sigma^{-1}$-self-dual. 
Moreover, we have
\begin{equation}
\label{eq-GGP-L-SigmaPi}
L(s, \sigma\times\pi)=L(s, \sigma\times \Pi)=L(s, \sigma\times \tau\otimes \omega_\pi).	
\end{equation}

We have the following result on the reciprocal non-vanishing of Bessel periods for the pair $(\Eisen_{\tau\otimes\sigma}, \pi)$ and the pair $(\pi, \sigma^\prime)$, which is an analogue of \cite[Theorem 5.3]{JiangZhang2020Annals} for generic representations.
\begin{theorem}
\label{thm-GGP-Reciprocal}	
Let $\pi$, $\sigma$ be irreducible generic cuspidal automorphic representations of $\GSpin_{m^\prime-2\ell-1}(\A)$ and $\GSpin_{m^\prime-2k}(\A)$ respectively such that $\omega_\pi \omega_\sigma=1$. Let $\Pi$ be the functorial transfer of $\pi$ and let $\tau=\Pi\otimes\omega_\pi^{-1}$. Assume that the residue $\Eisen_{\tau\otimes\sigma^\prime}$ is non-zero. Then the Bessel period $\mathcal{B}^{\psi_{\ell},a}$ for $(\Eisen_{\tau\otimes\sigma^\prime}, \pi)$ is non-zero for some choice of data if and only if the Bessel period $\mathcal{B}^{\psi_{k-\ell-1},-a}$ for the pair $(\pi, \sigma)$ is non-zero for some choice of data. 
\end{theorem}

\begin{proof}[Proof of Theorem~\ref{thm-GGP-Reciprocal}: forward direction]
Suppose that the Bessel period $\mathcal{B}^{\psi_{\ell},a}$ for $(\Eisen_{\tau\otimes\sigma^\prime}, \pi)$ is non-zero. Then by replacing the residue $\Eisen_{\tau\otimes\sigma^\prime}$ with the Eisenstein series $E(\cdot, f_{\tau, \sigma^\prime, s})$, we obtain that the global zeta integral $\mathcal{Z}(\phi_\pi, f_{\tau,\sigma^\prime, s})$ is not identically zero for $\Re(s)\gg 0$. By Corollary~\ref{corollary-vanishing-kgreaterl}, the Bessel period $\mathcal{B}^{\psi_{k-\ell-1},-a}$ for the pair $(\pi, \sigma)$ is non-zero for some choice of data.
\end{proof}

The proof of the converse direction of Theorem~\ref{thm-GGP-Reciprocal} is more involved, and requires that we know enough analytic properties of the local zeta integral at the ramified and archimedean places. 
We will study the local zeta integrals in the next section, and prove the converse direction of Theorem~\ref{thm-GGP-Reciprocal} in Section~\ref{subsection-GGP-Reciprocal-converse}.

\subsection{Normalization of local zeta integrals}
\label{subsection-GGP-Normalization-Local-Zeta}

In this section we continue our discussion on the global and local zeta integrals from Section~\ref{section-global-zeta-integrals} and Section~\ref{section-local-zeta-integrals}. Recall that $k>\ell$, so we are in the situation of Section~\ref{subsection-unfolding-kgm}.
Recall that
\begin{equation*}
	\mathcal{Z}(\phi_\pi, f_{\tau,\sigma^\prime, s})=\int_{R_{\ell,k-\ell-1}^{\eta_2}(\A)\backslash G(\A) } \mathcal{B}^{\psi_{k-\ell-1,-a}}(\pi(g) \phi_\pi, \mathcal{J}_{\ell,a}(R( \epsilon_{0,\beta} \eta_2 g) {f_{\mathcal{W}(\tau,\psi_{Z_k,a}^{-1}),\sigma^\prime,s}})) dg 
\end{equation*}
and
\begin{equation*}
	\mathcal{Z}(\phi_\pi, f_{\tau,\sigma^\prime, s})=  \prod_{\nu} \mathcal{Z}_{\nu}(\phi_{\pi_{\nu}},  f_{\mathcal{W}(\tau_{\nu}),\sigma^\prime_{\nu},s})= \LL^S(s, \tau,\pi,\sigma;\rho) \cdot \prod_{\nu\in S} \mathcal{Z}_{\nu}(\phi_{\pi_{\nu}},  f_{\mathcal{W}(\tau_{\nu}),\sigma^\prime_{\nu},s}).
\end{equation*}
where we denote
\begin{equation*}
\LL(s,\tau_\nu,\pi_\nu,\sigma_\nu;\rho):=	 \frac{L(s,\pi_\nu \times \tau_\nu)}{L(s+\frac{1}{2},\sigma_\nu  \times \tau_\nu \otimes\omega_{\pi_\nu}) L\left(2s, \tau_\nu, \rho \otimes \omega_{\pi_\nu}\right)}
\end{equation*}
at each place $\nu$, and 
\begin{equation*}
\LL^S(s, \tau,\pi,\sigma;\rho):=\prod_{\nu\not\in S} \LL(s,\tau_\nu,\pi_\nu,\sigma_\nu;\rho).
\end{equation*}
Here we recall that $S$ is a finite set of places consisting of all ramified places of relevant data and all archimedean places, and the data is chosen such that for $\nu\not\in S$, all data are unramified and normalized as in Section~\ref{subsection-unramified-computation}.

We normalize the local zeta integral by
\begin{equation}
\label{eq-GGP-NormalizedLocalIntegral}
	\mathcal{Z}^*_{\nu}(\phi_{\pi_{\nu}},  f_{\mathcal{W}(\tau_{\nu}),\sigma^\prime_{\nu},s})= \frac{\mathcal{Z}_{\nu}(\phi_{\pi_{\nu}},  f_{\mathcal{W}(\tau_{\nu}),\sigma^\prime_{\nu},s})}{\LL(s,\tau_\nu,\pi_\nu,\sigma_\nu;\rho)}.
\end{equation}
Then
\begin{equation*}
\prod_{\nu\in S} \mathcal{Z}_{\nu}(\phi_{\pi_{\nu}},  f_{\mathcal{W}(\tau_{\nu}),\sigma^\prime_{\nu},s}) = \prod_{\nu\in S}	\mathcal{Z}^*_{\nu}(\phi_{\pi_{\nu}},  f_{\mathcal{W}(\tau_{\nu}),\sigma^\prime_{\nu},s}) \LL(s,\tau_\nu,\pi_\nu,\sigma_\nu;\rho).
\end{equation*}
Put 
\begin{equation}
\label{eq-GGP-L-global}
\LL(s, \tau,\pi,\sigma;\rho)= \frac{L(s,\pi \times \tau)}{L(s+\frac{1}{2},\sigma  \times \tau \otimes \omega_\pi) L\left(2s, \tau, \rho \otimes \omega_{\pi}\right)} .
\end{equation}
Then we have
\begin{equation}
\label{eq-GGP-NormalizedGlobalIntegral}
\begin{split}
	\mathcal{Z}(\phi_\pi, f_{\tau,\sigma^\prime, s}) =\LL(s, \tau,\pi,\sigma;\rho) \cdot \mathcal{Z}_S^*(\phi_\pi, f_{\tau,\sigma^\prime, s}) 
\end{split}
\end{equation}
where
\begin{equation*}
\mathcal{Z}^*_S(\phi_\pi, f_{\tau,\sigma^\prime, s}) =\prod_{\nu\in S} \mathcal{Z}^*_{\nu}(\phi_{\pi_{\nu}},  f_{\mathcal{W}(\tau_{\nu}),\sigma^\prime_{\nu},s}).	
\end{equation*}

\begin{lemma}
\label{lemma-GGP-L-function-pole}
Let $\pi$, $\sigma$ be irreducible generic cuspidal automorphic representations of $\GSpin_{m^\prime-2\ell-1}(\A)$ and $\GSpin_{m^\prime-2k}(\A)$ respectively such that $\omega_\pi \omega_\sigma=1$. Let $\Pi$ be the functorial transfer of $\pi$ and let $\tau=\Pi\otimes\omega_\pi^{-1}$. Then 
$\LL(s, \tau,\pi,\sigma;\rho)$ has a pole at $s=1$ of order $r$.
\end{lemma}

\begin{proof}
Since  $\Pi$ is the functorial transfer of $\pi$ and $\tau=\Pi\otimes\omega_\pi^{-1}$, we have
\begin{equation*}
L(s, \pi\times\tau)=L(s,\Pi\times\tau)=L(s, \tau\otimes\omega_\pi \times\tau)=L(s, \tau, \wedge^2\otimes\omega_\pi) L(s, \tau, \Sym^2\otimes\omega_\pi).	
\end{equation*}
Using the identity~\eqref{eq-GGP-LTauRho} again, we see that $L(s, \pi\times\tau)$ has a pole at $s=1$ of order $r$. Let $\Pi_\sigma$ be the functorial transfer of $\sigma$.
Then
\begin{equation*}
L(s+\frac{1}{2}, \sigma\times\tau\otimes \omega_\pi) = \prod_{i=1}^r L(s+\frac{1}{2}, \Pi_\sigma\times\tau_i\otimes \omega_\pi)	
\end{equation*}
is holomorphic and non-zero at $s=1$. It is clear that
\begin{equation*}
	L(2s, \tau, \rho\otimes\omega_{\sigma}^{-1})=\prod_{i=1}^r L(2s, \tau_i, \rho\otimes\omega_{\sigma}^{-1}) \cdot \prod_{1\le i <j\le r} L(2s, \tau_i\times\tau_j\otimes \omega_{\sigma}^{-1}).
\end{equation*}
is holomorphic and non-zero at $s=1$. By \eqref{eq-GGP-L-global}, $\LL(s, \tau,\pi,\sigma;\rho)$ has a pole at $s=1$ of order $r$. This proves the lemma.
\end{proof}

\begin{proposition}
\label{prop-GGP-NormalizedLocalIntegral}
Let $\pi, \tau, \sigma$ be as in Theorem~\ref{thm-GGP-Reciprocal}. Then the following statements hold.
\begin{enumerate}
\item[(i)] $\mathcal{Z}^*_S(\phi_\pi, f_{\tau,\sigma^\prime, s})$ is meromorphic in $s\in \C$ for any choice of the smooth sections $f_{\tau,\sigma^\prime,s}$.
\item[(ii)] $\mathcal{Z}^*_S(\phi_\pi, f_{\tau,\sigma^\prime, s})$ is holomorphic at $s=1$ for any choice of the smooth sections $f_{\tau,\sigma^\prime,s}$.
\end{enumerate}
\end{proposition}

\begin{proof}
By Lemma~\ref{lemma-global-integral-convergence}, $\mathcal{Z}(\phi_\pi, f_{\tau,\sigma^\prime, s})$ is meromorphic for any choice of smooth sections $f_{\tau,\sigma^\prime, s}$. Also, since $\pi$ and $\sigma$ are generic, they have functorial transfer to the general linear groups and hence  the $L$-functions $L(s,\pi \times \tau)$ and $L(s+\frac{1}{2},\sigma  \times \tau \otimes \omega_\pi)$ are meromorphic. It follows that $\LL(s, \tau,\pi,\sigma;\rho)$ is meromorphic. By \eqref{eq-GGP-NormalizedGlobalIntegral}, we conclude that $\mathcal{Z}^*_S(\phi_\pi, f_{\tau,\sigma^\prime, s})$ is meromorphic in $s\in \C$ for any choice of the smooth sections $f_{\tau,\sigma^\prime,s}$. This proves (i).

Recall that  $\Pi$ is the functorial transfer of $\pi$ and $\tau=\Pi\otimes\omega_\pi^{-1}$. Then
\begin{equation*}
L(s, \pi\times\tau)=L(s,\Pi\times\tau)=L(s, \tau\otimes\omega_\pi \times\tau)=L(s, \tau, \wedge^2\otimes\omega_\pi) L(s, \tau, \Sym^2\otimes\omega_\pi).	
\end{equation*}
Using the identity~\eqref{eq-GGP-LTauRho} again, we see that $L(s, \pi\times\tau)$ has a pole at $s=1$ of order $r$. Let $\Pi_\sigma$ be the functorial transfer of $\sigma$.
Then
\begin{equation*}
L(s+\frac{1}{2}, \sigma\times\tau\otimes \omega_\pi) = \prod_{i=1}^r L(s+\frac{1}{2}, \Pi_\sigma\times\tau_i\otimes \omega_\pi)	
\end{equation*}
is holomorphic and non-zero at $s=1$. It is clear that
\begin{equation*}
	L(2s, \tau, \rho\otimes\omega_{\sigma}^{-1})=\prod_{i=1}^r L(2s, \tau_i, \rho\otimes\omega_{\sigma}^{-1}) \cdot \prod_{1\le i <j\le r} L(2s, \tau_i\times\tau_j\otimes \omega_{\sigma}^{-1}).
\end{equation*}
is holomorphic and non-zero at $s=1$. By \eqref{eq-GGP-L-global}, $\LL(s, \tau,\pi,\sigma;\rho)$ has a pole at $s=1$ of order $r$. This proves (iii).

Using the identity \eqref{eq-GGP-NormalizedGlobalIntegral}, in order to show that $\mathcal{Z}^*_S(\phi_\pi, f_{\tau,\sigma^\prime, s})$ is holomorphic at $s=1$ for any choice of the smooth sections $f_{\tau,\sigma^\prime,s}$, it suffices to show that $\mathcal{Z}(\phi_\pi, f_{\tau,\sigma^\prime, s})$ has a pole at $s=1$ of order at most $r$ for any smooth sections. By Proposition~\ref{prop-GGP-EisensteinSeriesPoles}, the Eisenstein series $E(\cdot, f_{\tau,\sigma, s})$ can possibly have a pole at $s=1$ of order at most $r$. Thus $\mathcal{Z}(\phi_\pi, f_{\tau,\sigma^\prime, s})$ can have a pole at $s=1$ of order at most $r$. This proves (ii).
\end{proof}

As an assumption in the converse direction of Theorem~\ref{thm-GGP-Reciprocal}, we now assume that the Bessel period $\mathcal{B}^{\psi_{k-\ell-1},-a}$ for $(\pi, \sigma)$ is non-zero. We recall from Section~\ref{subsection-unfolding-kgm} that for $\phi_\pi\in V_\pi$ and $f_{\tau,\sigma^\prime, s}\in \Ind_{P_{k}(\A)}^{H(\A)}(\tau|\cdot|^{s-\frac{1}{2}} \otimes \sigma^\prime)$, the global zeta integral $\mathcal{Z}(\phi_\pi, f_{\tau,\sigma^\prime, s})$ unfolds to 
$$\int_{R_{\ell,k-\ell-1}^{\eta_2}(\A)\backslash G(\A) } \mathcal{B}^{\psi_{k-\ell-1,-a}}(\pi(g) \phi_\pi, \mathcal{J}_{\ell,a}(R( \epsilon_{0,\beta} \eta_2 g) {f_{\mathcal{W}(\tau,\psi_{Z_k,a}^{-1}),\sigma^\prime,s}})) dg,$$
where $\mathcal{B}^{\psi_{k-\ell-1,-a}}$ is a Bessel period for $(\pi, \sigma)$, and  $\mathcal{J}_{\ell,a}$ denotes a certain Fourier coefficient as given in \eqref{eq-unfolding-J-l-a}. 
For each place $\nu$, we let $\bb_\nu$ be a non-zero functional in 
$$\Hom_{R_{\ell,k-\ell-1}^{\eta_2}(F)} (\pi\otimes \sigma, \psi_{k-\ell-1,-a}),$$
which is unique, up to scalar. Let 
$$f_{\mathcal{W}(\tau_\nu),\sigma^\prime_\nu,s}\in \rho_{\tau_\nu,\sigma_\nu^\prime,s}=	\Ind_{P_{k}(F_\nu)}^{H(F_\nu)}(\mathcal{W}(\tau_\nu)|\cdot|^{s-\frac{1}{2}} \otimes \sigma^\prime_\nu).$$
Recall that the local zeta integral at $\nu$ is 
\begin{equation}
\label{eq-RamifiedNonvanishing-integral}
	\mathcal{Z}(v_{\pi_\nu},	f_{\mathcal{W}(\tau_\nu),\sigma_\nu^\prime, s}) =
	\int_{R_{\ell,k-\ell-1}^{\eta_2}(F_\nu)\backslash G(F_\nu) } \int_{U_{k,\eta_2}^-(F_\nu)} \bb_\nu( \pi_\nu(g) v_{\pi_\nu},   {f_{\mathcal{W}(\tau_\nu),\sigma_\nu^\prime,s}}(u \epsilon_{\beta} \eta_2 g ) ) \psi_{U_{k,\eta_2}^-}(u) du  dg.
\end{equation}
At an unramified finite place $\nu\not\in S$, we normalize $\bb_\nu$ such that 
\begin{equation*}
\bb_\nu	(v_{\pi_\nu}, v_{\sigma_\nu})=1,
\end{equation*}
where $v_{\pi_\nu}\in V_{\pi_\nu}, v_{\sigma_\nu}\in V_{\sigma_\nu}$ are spherical vectors, and we normalize $f_{\mathcal{W}(\tau_\nu),\sigma^\prime_\nu,s}$ such that for all $a\in \GL_k(F_\nu)$, $f_{\mathcal{W}(\tau_\nu),\sigma_\nu^\prime, s}(e,a)=W_{\tau_\nu}(a) v_{\sigma_\nu^\prime}$, where $e\in H(F_\nu)$ is the identity element, and $W_{\tau_\nu}\in \mathcal{W}(\tau_\nu)$ is the unramified normalized Whittaker function such that $W_{\tau_\nu}(I_k)=1$.
The normalization of the data at a place $\nu\in S$ will be given in the proof of Proposition~\ref{prop-GGP-local-integral-non-zero} below (see \eqref{eq-RamifiedNonvanishing-3}). 
Denote the finite product of the local zeta integrals over places in $S$ by
\begin{equation*}
\mathcal{Z}_S(\phi_\pi, f_{\tau,\sigma^\prime,s})	= \prod_{\nu\in S} \mathcal{Z}(v_{\pi_\nu},	f_{\mathcal{W}(\tau_\nu),\sigma_\nu^\prime, s}).
\end{equation*}

\begin{proposition}
\label{prop-GGP-local-integral-non-zero}
Let $\pi, \tau, \sigma$ be as in Theorem~\ref{thm-GGP-Reciprocal}. Let $s=s_0\in \C$ be fixed. If for every $\nu\in S$, the local pairing $\bb_\nu(v_{\pi_\nu}, v_{\sigma_\nu} )$ is non-zero for some $v_{\pi_\nu}\in V_{\pi_\nu}, v_{\sigma_\nu}\in V_{\sigma_\nu}$, then there exists a collection of local sections $f_{\mathcal{W}(\tau_\nu),\sigma^\prime_\nu, s}$ with $\nu$ running in $S$, such that the finite product of local zeta integrals over $S$, $\mathcal{Z}_S(\phi_\pi, f_{\tau,\sigma^\prime,s})$, is non-zero at $s=s_0$.
\end{proposition}

\begin{proof}
We follow the method in \cite[Appendix A]{JiangZhang2020Annals}.  
Let $s=s_0$ be given and let $\nu$ be a place in $S$. 
It suffices to prove that there exists a choice of $f_{\mathcal{W}(\tau_\nu),\sigma^\prime_\nu, s}$ such that the local zeta integral $\mathcal{Z}(v_{\pi_\nu},	f_{\mathcal{W}(\tau_\nu),\sigma_\nu^\prime, s})$ is non-zero at $s=s_0$.
We choose $f_{\mathcal{W}(\tau_\nu),\sigma^\prime_\nu, s}$ such that 
\begin{equation}
\label{eq-RamifiedNonvanishing-1}
	f_{\mathcal{W}(\tau_\nu),\sigma^\prime_\nu, s}(ah u \overline{n} \epsilon_{\beta} \eta_2 )=|\det(a)|^{s+\rho_k} W_{\tau_\nu}(a)f_{\nu}(\overline{n}) \sigma_{\nu}(h)v_{\sigma_\nu},
\end{equation}
where $a\in \GL(V_{k}^+(F_\nu))\cong\GL_k(F_\nu)$, $h\in G_0(F_\nu)$, $u\in U_k(F_\nu)$, $\overline{n}\in U_{k,\eta_2}^-(F_\nu)$. Here, $W_{\tau_\nu}$ is a Whittaker function in $\mathcal{W}(\tau_\nu)$, $f_\nu(\overline{n})$ is a smooth, compactly supported function defined on $U_{k}^-(F_{\nu})$, and $\rho_k$ is a number such that $ah\mapsto |\det(a)|^{2\rho_k}$, for $a\in \GL_k(F_\nu), h\in G_0(F_\nu)$, is the modular character of the parabolic subgroup $P_k(F_\nu)$. Over archimedean places, we may also take $f_\nu$ to be a positive real-valued function. Additionally, we may assume $f_\nu(e)=1$ where $e\in U_{k,\eta_2}^-(F_\nu)$ is the identity element. Then we have
\begin{equation}
\label{eq-RamifiedNonvanishing-2}
	f_{\mathcal{W}(\tau_\nu),\sigma^\prime_\nu, s}( \epsilon_{\beta} \eta_2 )= W_{\tau_\nu}(I_k)v_{\sigma_\nu},
\end{equation}
where $I_k$ is the identity matrix in $\GL_k(F_\nu)$. It follows that
\begin{equation}
\label{eq-RamifiedNonvanishing-2-3}
	\bb_\nu(v_{\pi_\nu}, f_{\mathcal{W}(\tau_\nu),\sigma^\prime_\nu, s}( \epsilon_{\beta} \eta_2 ) )= W_{\tau_\nu}(I_k) \cdot \bb_{\nu}(v_{\pi_\nu}, v_{\sigma_\nu}).
\end{equation}
We may assume $W_{\tau_\nu}(I_k)\not=0$, and normalize $\bb_{\nu}$ so that
\begin{equation}
\label{eq-RamifiedNonvanishing-3}
	W_{\tau_\nu}(I_k) \cdot \bb_{\nu}(v_{\pi_\nu}, v_{\sigma_\nu}) = 1.
\end{equation}
This gives the normalization of the data at a place $\nu\in S$.

Recall that $G$ is identified as a subgroup of the Levi subgroup $L_\ell$ of $H$. Since $P_k(F_\nu)U_k^-(F_\nu)$ is an open dense subset of $H(F_\nu)$ whose complement has Haar measure 0, an integral over the domain $R_{\ell,k-\ell-1}^{\eta_2}(F_\nu)\backslash G(F_\nu) $ is the same as an integral over the domain
\begin{equation}
\label{eq-RamifiedNonvanishing-4}
	R_{\ell,k-\ell-1}^{\eta_2}(F_\nu)\backslash ( G(F_\nu) \cap (\epsilon_\beta \eta_2)^{-1} P_k(F_\nu)U_k^-(F_\nu) (\epsilon_\beta \eta_2)).
\end{equation}
Let $P_{m^\prime-2\ell, k-\ell}=\GSpin({W_{\ell}})\cap \epsilon_{\beta}^{-1}P_k \epsilon_\beta$ be the standard parabolic subgroup of $\GSpin({W_{\ell}})=\GSpin_{m^\prime-2\ell}$ with Levi decomposition $(\GL_{k-\ell}\times \GSpin(W_{k}) ) \ltimes U_{m^\prime-2\ell, k-\ell}$, where $U_{m^\prime-2\ell, k-\ell}$ is the unipotent radical of $P_{m^\prime-2\ell, k-\ell}$. Let $U_{m^\prime-2\ell, k-\ell}^-$ be the opposite of $U_{m^\prime-2\ell, k-\ell}$.
Recall that $R_{\ell,k-\ell-1}^{\eta_2}\subset G^{\eta_2}$ where (see \eqref{eq-G-eta2})
$$
G^{\eta_2}=G\cap \eta_2^{-1} M_{\ell}^{\epsilon_{\beta}} \eta_2=(\GL(V_{\tilde{m}-\beta,\beta-1}^+)\times \eta_2^{-1}\GSpin_{m^\prime-2k}\eta_2  ) \ltimes 	U_{\beta-1,\eta_2}
$$
with $M_{\ell}^{\epsilon_{\beta}}=\epsilon_{\beta}^{-1}P_k \epsilon_{\beta}\cap M_{\ell}$, we see that
\begin{equation*}
	R_{\ell,k-\ell-1}^{\eta_2}(F_\nu) \subset G(F_\nu)\cap \eta_2^{-1}P_{m^\prime-2\ell, k-\ell}(F_\nu)\eta_2, 
\end{equation*}
hence the quotient \eqref{eq-RamifiedNonvanishing-4} is well-defined. Recall that $U_k^-$ is the opposite of $U_k$.  
Using the fact that
\begin{equation*}
 G(F_\nu)\cap \epsilon_{\beta}^{-1}U_k^-(F_\nu) \epsilon_\beta \subset U_{m^\prime-2\ell, k-\ell}^-(F_\nu)= 	\left\{\begin{pmatrix}
I_\ell &&&&\\
&I_{k-\ell} &&&\\
&x&I_{m^\prime-2k}&&\\
&y&x^\prime&I_{k-\ell}&\\
&&&&I_\ell
\end{pmatrix}  \right\},
\end{equation*} 
we see that an integral over the domain $R_{\ell,k-\ell-1}^{\eta_2}(F_\nu)\backslash G(F_\nu) $ is the same as an integral over the domain
\begin{equation}
\label{eq-RamifiedNonvanishing-5}
	R_{\ell,k-\ell-1}^{\eta_2}(F_\nu)\backslash ( G(F_\nu) \cap \Ad(\eta_2^{-1}) (P_{m^\prime-2\ell, k-\ell} (F_\nu) U_{m^\prime-2\ell, k-\ell}^-(F_\nu)) ) .
\end{equation}
 Moreover, by a simple change of variable, we obtain that $\mathcal{Z}(v_{\pi_\nu},	f_{\mathcal{W}(\tau_\nu),\sigma_\nu^\prime, s})$ is equal to 
 \begin{equation}
\label{eq-RamifiedNonvanishing-6}
\begin{split}
	&\int_{\Ad(\eta_2)R_{\ell,k-\ell-1}^{\eta_2}(F_\nu)\backslash \Ad(\eta_2)G(F_\nu)\cap (P_{m^\prime-2\ell, k-\ell} (F_\nu) U_{m^\prime-2\ell, k-\ell}^-(F_\nu))} \int_{U_{k,\eta_2}^-(F_\nu)} \\
	&\quad \quad \bb_\nu( \pi_\nu(\eta_2^{-1}g\eta_2) v_{\pi_\nu},   {f_{\mathcal{W}(\tau_\nu),\sigma_\nu^\prime,s}}(u \epsilon_{\beta}  g \eta_2) ) \psi_{U_{k,\eta_2}^-}(u) du  dg.
\end{split}
\end{equation}

 Note that by \eqref{eq-Bessel-subgroup-of-G}, we have
\begin{equation*}
\Ad(\eta_2) R_{\ell,k-\ell-1}^{\eta_2}(F_\nu)=\GSpin_{m^\prime-2k}(F_\nu) \cdot  Z_{\ell, \beta-1}(F_\nu)	\cdot U_{\beta,m^\prime-2\ell}^{\eta_2}(F_\nu).
\end{equation*}
We now consider  the set $\Ad(\eta_2)G(F_\nu)\cap (P_{m^\prime-2\ell, k-\ell} (F_\nu) U_{m^\prime-2\ell, k-\ell}^-(F_\nu))$. 
 We take 
 \begin{equation}
 \label{eq-RamifiedNonvanishing-7}
 p=\begin{pmatrix}
I_{k-\ell} &Y&Z&\\
&I_{m^\prime-2k}&Y^\prime&\\
& &I_{k-\ell}&\\
\end{pmatrix} \cdot \widehat{\diag(I_{\ell},A)} \cdot h \in P_{m^\prime-2\ell, k-\ell} (F_\nu),
 \end{equation}
\begin{equation}
\label{eq-RamifiedNonvanishing-8}
\overline{n}= 	\begin{pmatrix}
I_{k-\ell} &&&\\
X^\prime&I_{m^\prime-2k}&&\\
W&X&I_{k-\ell}&\\
\end{pmatrix} \in  U_{m^\prime-2\ell, k-\ell}^-(F_\nu),
\end{equation}
where $A\in \GL_{k-\ell}(F_\nu)$, $h\in \GSpin(W_k(F_\nu))$. Here, for $g\in \GL(V_k^+(F_\nu))\cong \GL_k(F_\nu)$, we denote by $\hat{g}$ the lift of $g$ into the Levi subgroup $\GL(V_k^+)\times \GSpin(W_k)$ of $P_k$. 
Since $G$ fixes the vector $w_0$, it follows that $\Ad(\eta_2)G$ fixes the vector
\begin{equation}
\label{eq-RamifiedNonvanishing-9}
\pr(\eta_2)\cdot  w_0 = \begin{pmatrix}
 E_1\\
 0_{(m^\prime-2k)\times 1}\\
 E_2	
 \end{pmatrix}_{(m^\prime-2\ell)\times 1},
\end{equation}
where
\begin{equation*}
E_1= ( 0,  \cdots, 0, 1)^t, \quad E_2=((-1)^{m^\prime+1} \frac{a}{2}, 0, \cdots, 0)^t \text{ in } \Mat_{(k-\ell)\times 1}.
\end{equation*}
Here we view $\pr(\eta_2)\cdot  w_0$ as an anisotropic vector in the space $W_{\ell}$. Then 
$$p\cdot \overline{n}\in P_{m^\prime-2\ell, k-\ell} (F_\nu)U_{m^\prime-2\ell, k-\ell}^-(F_\nu)$$ is in $\Ad(\eta_2)G(F_\nu)$ if and only if $\pr(p\cdot \overline{n})$ fixes the vector \eqref{eq-RamifiedNonvanishing-9}.
This means that if we write $p$ as in \eqref{eq-RamifiedNonvanishing-7} and $\overline{n}$ as in \eqref{eq-RamifiedNonvanishing-8}, then we get the following equations:
 \begin{equation}
\label{eq-RamifiedNonvanishing-10}
(A^{-1}-I_{k-\ell})E_1=Z E_2, \quad X^\prime E_1=\pr(h) \cdot Y^\prime E_2, \quad WE_1=(A^*-I_{k-\ell})E_2.
\end{equation}
Because the unipotent part of $\Ad(\eta_2) R_{\ell,k-\ell-1}^{\eta_2}(F_\nu)$ is given by \eqref{eq-U-beta-mprime-2l-eta2} and \eqref{eq-Z-ell-beta-1}, we may identify the domain of the outer integration in \eqref{eq-RamifiedNonvanishing-6} by setting $h=e$, $A\in Z_{k-\ell}(F_\nu)\backslash \GL_{k-\ell}(F_\nu)$, 
\begin{equation*}
Y=\begin{pmatrix}
0_{(k-\ell-1)\times (m^\prime-2k)} \\
y	
\end{pmatrix}, \quad Z= \begin{pmatrix}
 0_{(k-\ell-1)\times 1} &0_{(k-\ell-1)\times (k-\ell-1)}\\
 z &0_{1\times (k-\ell-1)}	
 \end{pmatrix}.
\end{equation*}
Due to \eqref{eq-RamifiedNonvanishing-10}, the vector $y$ in $Y$ and $z$ in $Z$ are determined by $X$ and $A$ respectively. Thus we may write $Y_X$ and $Z_A$ for $Y$ and $Z$, respectively.  

We now choose $f_\nu$ so that
\begin{equation}
\label{eq-RamifiedNonvanishing-11}
f_{\nu} \left(\begin{pmatrix}
 I_{k-\ell} & &&&& \\
 &I_{\ell} &&&&\\
 X^\prime &x_2^\prime &I_{m^\prime-2k} &&\\
 x_1& x_3&x_2&I_{\ell} &\\
 W&x_1^\prime &X&&I_{k-\ell}	
 \end{pmatrix}\right) = f_1(x_1, x_2, x_3)f_2(X, W),
\end{equation}
where $f_1$ and $f_2$ are smooth, compactly supported functions, and the sizes of $x_1, x_2, x_3, X, W$ are indicated by the size of the matrix in \eqref{eq-RamifiedNonvanishing-11}. 
With this choice of $f_\nu$, we are able to write down a more explicit formula for the local zeta integral $\mathcal{Z}(v_{\pi_\nu},	f_{\mathcal{W}(\tau_\nu),\sigma_\nu^\prime, s})$. We decompose $\Ad(\eta_2)g$ as $p\cdot \overline{n}$ where $p$ and $\overline{n}$ are given in \eqref{eq-RamifiedNonvanishing-7} and \eqref{eq-RamifiedNonvanishing-8} subject to the above conditions. Then 
\begin{equation*}
\epsilon_\beta p \epsilon_\beta^{-1}= n(Y_X, Z_A) \cdot  \widehat{\diag(A, I_{\ell})}, \quad n(Y_X, Z_A)= \begin{pmatrix}
I_{k-\ell} &&Y_X&&Z_A\\
&I_{\ell}&&&\\
&&I_{m^\prime-2k}&&Y_X^\prime\\
&&&I_{\ell}&\\
&&&&I_{k-\ell}
\end{pmatrix}. 	
\end{equation*}
We also have
\begin{equation*}
\epsilon_\beta \overline{n}\epsilon_\beta^{-1} = 	\begin{pmatrix}
I_{k-\ell} && && \\
&I_{\ell}&&&\\
X^\prime&&I_{m^\prime-2k}&& \\
&&&I_{\ell}&\\
W&&X&&I_{k-\ell}
\end{pmatrix}.
\end{equation*}
With such a decomposition, the inner integration in \eqref{eq-RamifiedNonvanishing-6} is equal to 
\begin{equation*}
\int_{U_{k,\eta_2}^-(F_\nu)}     {f_{\mathcal{W}(\tau_\nu),\sigma_\nu^\prime,s}}(u \cdot    n(Y_X, Z_A) \cdot \widehat{\diag(A, I_{\ell})} \cdot \epsilon_\beta \overline{n}\epsilon_\beta^{-1}  \cdot \epsilon_\beta \eta_2)  \psi_{U_{k,\eta_2}^-}(u) du	.
\end{equation*}
Recall the group $U_{k,\eta_2}^-$ is defined in \eqref{eq-U-k-eta2}, and $\GL(V_{k-\ell}^+)$  lies in its normalizer. For
\begin{equation*}
u(x_1, x_2, x_3)=   \begin{pmatrix}
 I_{k-\ell} & &&&& \\
 &I_{\ell} &&&&\\
 &x_2^\prime &I_{m^\prime-2k} &&\\
 x_1& x_3&x_2&I_{\ell} &\\
 &x_1^\prime &&&I_{k-\ell}	
 \end{pmatrix}
\in U_{k,\eta_2}^-(F_\nu),
\end{equation*}
we have
\begin{equation*}
\begin{split}
& u(x_1, x_2, x_3) \cdot    n(Y_X, Z_A) \cdot \widehat{\diag(A, I_{\ell})} \\
= & \widehat{\diag(A, I_{\ell})} \cdot \widehat{n_k(B_1) }  \cdot u_0   \cdot u(x_1A, x_2+x_1Y_X, x_3+B_2 x_1^\prime)
\end{split}
\end{equation*}
where $u_0\in U_k(F_\nu)$, $n_k(B_1)=\left(\begin{smallmatrix} I_{k-\ell} & B_1\\ 0 &I_\ell\end{smallmatrix}\right)$, $B_1=-A^{-1}Y_Xx_2^\prime-Z_A A^\prime x_1^\prime$, $B_2=x_2Y_X^\prime+x_1AZ_A A^\prime$.
Now we use \eqref{eq-RamifiedNonvanishing-1}, \eqref{eq-RamifiedNonvanishing-2} and \eqref{eq-RamifiedNonvanishing-11} and a change of variable to obtain that, for a representative $p\cdot \overline{n}$ chosen above,  the inner integration in \eqref{eq-RamifiedNonvanishing-6} is equal to 
\begin{equation}
\begin{split}
\label{eq-RamifiedNonvanishing-12}
& |\det(A)|^{s+\rho_k} W_{\tau_\nu}\left( \begin{pmatrix} A&\\&I_{\ell}\end{pmatrix} \right) f_2(X, W)v_\sigma \\
&\times \int_{U_{k,\eta_2}^-(F_\nu)}     f_1(x_1, x_2,x_3)  \psi((x_1A^{-1})_{\ell,k-\ell}) \psi_{Z_k,a}^{-1}(n_k(B_1)) |\det(A)|^{-\ell} dx_i, 
\end{split}
\end{equation}
where the matrices $x_1, x_2, x_3$ define the element $u(x_1, x_2, x_3)$ in $U_{k,\eta_2}^-(F_\nu)$. Since the smooth function $f_1$ is chosen to be compactly supported and is independent of the complex variable $s$, we see that the expression \eqref{eq-RamifiedNonvanishing-12} is well-defined over the entire complex plane for the above choice of data, and so is the local zeta integral $\mathcal{Z}(v_{\pi_\nu},	f_{\mathcal{W}(\tau_\nu),\sigma_\nu^\prime, s})$.

By plugging \eqref{eq-RamifiedNonvanishing-12} into \eqref{eq-RamifiedNonvanishing-6}, we obtain that the local zeta integral $\mathcal{Z}(v_{\pi_\nu},	f_{\mathcal{W}(\tau_\nu),\sigma_\nu^\prime, s})$ is equal to
\begin{equation}
\begin{split}
\label{eq-RamifiedNonvanishing-13}
&\int_{X,W}\int_{A} |\det(A)|^{s+\rho_k} W_{\tau_\nu}\left( \begin{pmatrix} A&\\&I_{\ell}\end{pmatrix} \right) f_2(X, W) \bb_\nu( \pi_\nu(\eta_2^{-1} p(A, X) \overline{n}(X, W) \eta_2) v_{\pi_\nu}, v_\sigma) \\
&\times \int_{U_{k,\eta_2}^-(F_\nu)}     f_1(x_1, x_2,x_3)  \psi((x_1A^{-1})_{\ell,k-\ell}) \psi_{Z_k,a}^{-1}(n_k(B_1)) |\det(A)|^{-\ell} dx_i dA dX dW.  
\end{split}
\end{equation}
Here, we write $p$ as $p(A, X)$ and $\overline{n}$ as $\overline{n}(X, W)$, to indicate their dependence on $A, X$ and $X, W$ respectively. The integration $\int_A$ is over the set 
\begin{equation}
\label{eq-RamifiedNonvanishing-14}
	Z_{k-\ell}(F_\nu)\backslash \{ A\in \GL_{k-\ell}(F_\nu): WE_1=(A^*-I_{k-\ell})E_2 \}
\end{equation}
with constraints given in \eqref{eq-RamifiedNonvanishing-10}. The integration $\int_{X, W}$ is over the set $U_{m^\prime-2\ell, k-\ell}^-(F_\nu)$ with the condition $WE_1\not=-E_2$, due to the constraint $WE_1=(A^*-I_{k-\ell})E_2$ in \eqref{eq-RamifiedNonvanishing-10} and $\det(A)\not=0$. Indeed, if $WE_1=-E_2$, then we deduce from $WE_1=(A^*-I_{k-\ell})E_2$ that $A^* E_2=0$, which would imply $\det(A)=0$.

Due to the above expression \eqref{eq-RamifiedNonvanishing-13} for the local zeta integral $\mathcal{Z}(v_{\pi_\nu},	f_{\mathcal{W}(\tau_\nu),\sigma_\nu^\prime, s})$, we will use proof by contradiction to finish the proof of Proposition~\ref{prop-GGP-local-integral-non-zero}.
Suppose that the local zeta integral $\mathcal{Z}(v_{\pi_\nu},	f_{\mathcal{W}(\tau_\nu),\sigma_\nu^\prime, s})$ is identically zero for all choices of data $f_1$ and $f_2$ at the given $s=s_0$. 
We can vary the function $f_2(X, W)$ first, and treat the rest of the expression in \eqref{eq-RamifiedNonvanishing-13} as a continuous function of $X$ and $W$. Since the integral over $\overline{n}(X, W)\in U_{m^\prime-2\ell, k-\ell}^-(F_\nu)$ is identically zero for any choice of data, the remaining expression as a function of $\overline{n}$ is identically zero; that is,
\begin{equation}
\begin{split}
\label{eq-RamifiedNonvanishing-15}
&\int_{A} |\det(A)|^{s+\rho_k-\ell} W_{\tau_\nu}\left( \begin{pmatrix} A&\\&I_{\ell}\end{pmatrix} \right)  \bb_\nu( \pi_\nu(\eta_2^{-1} p  \overline{n}  \eta_2) v_{\pi_\nu}, v_\sigma) \\
&\times \int_{U_{k,\eta_2}^-(F_\nu)}     f_1(x_1, x_2,x_3)  \psi((x_1A^{-1})_{\ell,k-\ell}) \psi_{Z_k,a}^{-1}(n_k(B_1))  dx_i dA  \equiv 0.  
\end{split}
\end{equation}
In particular, the integral on the left-hand side of \eqref{eq-RamifiedNonvanishing-15} is equal to zero at $\overline{n}=I_{m^\prime}$, or equivalently at $X=0_{(k-\ell)\times (m^\prime-2k)}$ and $W=0_{(k-\ell)\times (k-\ell)}$. Due to $X^\prime E_1=Y^\prime E_2$ in \eqref{eq-RamifiedNonvanishing-10}, we must have $Y_X=0_{(k-\ell)\times (m^\prime-2k)}$. Similarly, $Z_A=0_{(k-\ell)\times (k-\ell)}$. It follows that 
$$B_1=-A^{-1}Y_Xx_2^\prime-Z_A A^\prime x_1^\prime=0, \quad \psi_{Z_k,a}^{-1}(n_k(B_1))=1.$$
Moreover, since the integral over $A$ is modulo $Z_{k-\ell}(F_\nu)$ subject to the condition that $(A^*-I_{k-\ell})E_2=WE_1=0_{(k-\ell)\times 1}$ (see \eqref{eq-RamifiedNonvanishing-14}), we see that the integral over $A$ is over $Z_{k-\ell}(F_\nu) \backslash P_{k-\ell}^{\GL_{k-\ell}}(F_\nu)\cong Z_{k-\ell-1}(F_\nu)\backslash \GL_{k-\ell-1}(F_\nu)$, where $P_{k-\ell}^{\GL_{k-\ell}}$ is the standard mirabolic subgroup of $\GL_{k-\ell}$. Since $A\in  P_{k-\ell}^{\GL_{k-\ell}}(F_\nu)$, we have $(x_1A^{-1})_{\ell,k-\ell}=(x_1)_{\ell,k-\ell}$ and hence $\psi((x_1A^{-1})_{\ell,k-\ell})=\psi((x_1)_{\ell,k-\ell})$.
Thus, by plugging $\overline{n}=I_{m^\prime}$ into \eqref{eq-RamifiedNonvanishing-15} we obtain that
\begin{equation}
\begin{split}
\label{eq-RamifiedNonvanishing-16}
&\int_{Z_{k-\ell-1}(F_\nu)\backslash \GL_{k-\ell-1}(F_\nu)} |\det(A)|^{s+\rho_k-\ell} W_{\tau_\nu}\left( \begin{pmatrix} A&\\&I_{\ell+1}\end{pmatrix} \right)  \bb_\nu( \pi_\nu(  \widehat{\diag(A, I_{\ell})} ) v_{\pi_\nu}, v_\sigma) dA \\
&\times \int_{U_{k,\eta_2}^-(F_\nu)}     f_1(x_1, x_2,x_3)  \psi((x_1)_{\ell,k-\ell})  dx_i  = 0.  
\end{split}
\end{equation}
Now we choose a suitable smooth, compactly supported function $f_1$ such that 
\begin{equation}
\label{eq-RamifiedNonvanishing-17}
\int_{U_{k,\eta_2}^-(F_\nu)}     f_1(x_1, x_2,x_3)  \psi((x_1)_{\ell,k-\ell})  dx_i=1.	
\end{equation}
By combining \eqref{eq-RamifiedNonvanishing-16} and \eqref{eq-RamifiedNonvanishing-17}, we conclude that 
\begin{equation}
\begin{split}
\label{eq-RamifiedNonvanishing-18}
\int_{Z_{k-\ell-1}(F_\nu)\backslash \GL_{k-\ell-1}(F_\nu)} |\det(A)|^{s+\rho_k-\ell} W_{\tau_\nu}\left( \begin{pmatrix} A&\\&I_{\ell+1}\end{pmatrix} \right)  \bb_\nu( \pi_\nu(  \widehat{\diag(A, I_{\ell})} ) v_{\pi_\nu}, v_\sigma) dA  = 0.  
\end{split}
\end{equation}
We can now apply the same inductive argument in \cite[Sections 6 and 7]{Soudry1993} and the Dixmier-Marlliavin Lemma \cite{DixmierMalliavin1978}
to conclude that 
\begin{equation*}
	W_{\tau_\nu}(I_k) \cdot \bb_{\nu}(v_{\pi_\nu}, v_{\sigma_\nu}) = 1.
\end{equation*}
This contradicts \eqref{eq-RamifiedNonvanishing-3}. Therefore, there must exist a choice of data such that the local zeta integral $\mathcal{Z}(v_{\pi_\nu},	f_{\mathcal{W}(\tau_\nu),\sigma_\nu^\prime, s})$ is non-zero at $s=s_0$. This completes the proof of Proposition~\ref{prop-GGP-local-integral-non-zero}. 
\end{proof}

\begin{proposition}
\label{prop-GGP-local-integral-non-zero-at-center}
Let $\pi, \tau, \sigma$ be as in Theorem~\ref{thm-GGP-Reciprocal}. Assume further that $(\pi, \sigma)$ has a non-zero Bessel period  $\mathcal{B}^{\psi_{k-\ell-1,-a}}$. Then there exist factorizable data $\phi_\pi\in V_\pi, \phi_{\sigma}\in V_{\sigma}, f_{\tau,\sigma^\prime, s}\in f_{\tau,\sigma^\prime, s}\in \Ind_{P_{k}(\A)}^{H(\A)}(\tau|\cdot|^{s-\frac{1}{2}} \otimes \sigma^\prime)$ such that $\mathcal{Z}_S^*(\phi_\pi, f_{\tau,\sigma^\prime,s})$ at $s=1$ and the Bessel period $\mathcal{B}^{\psi_{k-\ell-1,-a}}(\phi_\pi, \phi_\sigma)$ for $(\phi_\pi, \phi_\sigma)$ are simultaneously non-zero.
\end{proposition}

\begin{proof}
By assumption, there exist factorizable vectors $\phi_\pi=\otimes_{\nu} v_{\pi_\nu}\in V_\pi, \phi_{\sigma}=\otimes_{\nu} v_{\sigma_\nu} \in V_{\sigma}$ such that the Bessel period $\mathcal{B}^{\psi_{k-\ell-1,-a}}(\phi_\pi, \phi_\sigma)$ is non-zero. By the uniqueness of Bessel models, we have
\begin{equation*}
	\mathcal{B}^{\psi_{k-\ell-1,-a}}(\phi_\pi, \phi_\sigma) = c_{\pi,\sigma} \prod_{\nu} \bb_\nu(v_{\pi_\nu}, v_{\sigma_\nu} )
\end{equation*}
 for some constant $c_{\pi,\sigma}$. Since $\mathcal{B}^{\psi_{k-\ell-1,-a}}(\phi_\pi, \phi_\sigma)\not=0$, it follows that 
 $$
 c_{\pi,\sigma}\not=0, \quad \bb_\nu(v_{\pi_\nu}, v_{\sigma_\nu} )\not=0 \text{ for all } \nu.
 $$ 
 By Proposition~\ref{prop-GGP-local-integral-non-zero}, there exists a smooth factorizable section $f_{\tau,\sigma^\prime, s}$ such that $\mathcal{Z}_S(\phi_\pi, f_{\tau,\sigma^\prime,s})$ is a non-zero constant at $s=1$. 
 
 Recall from \eqref{eq-GGP-NormalizedLocalIntegral} that
 \begin{equation*}
\begin{split}
	\mathcal{Z}_S(\phi_\pi, f_{\tau,\sigma^\prime,s}) =\LL_S(s, \tau,\pi,\sigma;\rho) \cdot \mathcal{Z}_S^*(\phi_\pi, f_{\tau,\sigma^\prime,s}),
\end{split}
\end{equation*}
where
\begin{equation*}
\LL_S(s, \tau,\pi,\sigma;\rho)=\prod_{\nu\in S} \LL(s,\tau_\nu,\pi_\nu,\sigma_\nu;\rho).
\end{equation*}
For each $\nu\in S$, the $L$-function part $\LL(s,\tau_\nu,\pi_\nu,\sigma_\nu;\rho)$ is holomorphic for $\Re(s)\ge 1$. Thus $\LL_S(s, \tau,\pi,\sigma;\rho)$ is holomorphic for $\Re(s)\ge 1$. This proves that $\mathcal{Z}_S^*(\phi_\pi, f_{\tau,\sigma^\prime,s})$ is non-zero at $s=1$.
\end{proof}

\vspace{0.5in}
\subsection{Proof of Theorem~\ref{thm-GGP-Reciprocal}: converse direction}
\label{subsection-GGP-Reciprocal-converse}

The goal of this section is to finish the proof of Theorem~\ref{thm-GGP-Reciprocal}. Recall that we have proved the forward direction of Theorem~\ref{thm-GGP-Reciprocal}. It remains to prove the converse direction. 

\begin{proof}[Proof of Theorem~\ref{thm-GGP-Reciprocal}: converse direction]
Assume that the Bessel period $\mathcal{B}^{\psi_{k-\ell-1},-a}(\phi_\pi, \phi_\sigma)$ for the pair $(\pi, \sigma)$ is non-zero for some choice of data $\phi_\pi\in V_\pi$ and $\phi_\sigma\in V_\sigma$. Recall from \eqref{eq-GGP-NormalizedGlobalIntegral} that we have
\begin{equation*}
\begin{split}
	\mathcal{Z}(\phi_\pi, f_{\tau,\sigma^\prime, s}) =\LL(s, \tau,\pi,\sigma;\rho) \cdot \mathcal{Z}_S^*(\phi_\pi, f_{\tau,\sigma^\prime, s}).
\end{split}
\end{equation*}
By Proposition~\ref{prop-GGP-NormalizedLocalIntegral}, $\mathcal{Z}_S^*(\phi_\pi, f_{\tau,\sigma^\prime, s})$ is meromorphic in $s$ and is holomorphic at $s=1$ for any choice of the smooth sections $f_{\tau,\sigma^\prime, s}$. By Proposition~\ref{prop-GGP-local-integral-non-zero-at-center}, there exists a choice of factorizable data 
$$\phi_\pi\in V_\pi, \quad \phi_{\sigma}\in V_{\sigma}, \quad f_{\tau,\sigma^\prime, s}\in f_{\tau,\sigma^\prime, s}\in \Ind_{P_{k}(\A)}^{H(\A)}(\tau|\cdot|^{s-\frac{1}{2}} \otimes \sigma^\prime)$$ 
such that $\mathcal{Z}_S^*(\phi_\pi, f_{\tau,\sigma^\prime,s})$ at $s=1$ is non-zero and the Bessel period $\mathcal{B}^{\psi_{k-\ell-1,-a}}(\phi_\pi, \phi_\sigma)$ for $(\phi_\pi, \phi_\sigma)$ is non-zero.

With such a choice of data, and with a factorizable section $f_{\tau,\sigma^\prime, s}=\otimes_{\nu} f_{\tau_\nu,\sigma_\nu^\prime, s}$ that correspond to the above factorizable section $\otimes_\nu f_{\mathcal{W}(\tau_\nu),\sigma_\nu^\prime, s}$, by Lemma~\ref{lemma-GGP-L-function-pole}, $\LL(s, \tau,\pi,\sigma;\rho)$ has a pole at $s=1$ of order $r$. Hence the right-hand side of \eqref{eq-GGP-NormalizedGlobalIntegral} has a pole at $s=1$ of order $r$. It follows that the left-hand side of \eqref{eq-GGP-NormalizedGlobalIntegral}, i.e., $\mathcal{Z}(\phi_\pi, f_{\tau,\sigma^\prime, s})$, has a pole at $s=1$ of order $r$. Since we know from Proposition~\ref{prop-GGP-EisensteinSeriesPoles} that the Eisenstein series $E(\cdot, f_{\tau,\sigma^\prime,s})$ has a pole at $s=1$ of order at most $r$, it follows that the Eisenstein series $E(\cdot, f_{\tau,\sigma^\prime,s})$ must have a pole at $s=1$ of order exactly equal to $r$. Thus 
$$\mathcal{Z}(\phi_\pi, f_{\tau,\sigma^\prime, s})=\mathcal{B}^{\psi_{\ell,a}}(E(\cdot, f_{\tau,\sigma^\prime, s}), \phi_\pi)$$ 
has a pole at $s=1$ of order $r$.  Since $\phi_\pi$ is cuspidal, by taking the iterated residue of $\mathcal{B}^{\psi_{\ell,a}}(E(\cdot, f_{\tau,\sigma^\prime, s}), \phi_\pi)$, we see that the Bessel period  $\mathcal{B}^{\psi_{\ell,a}}(\Eisen_{\tau\otimes\sigma^\prime}, \phi_\pi)$ is non-zero, where we recall that $\Eisen_{\tau\otimes\sigma^\prime}$ is the $r$-th iterated residue at $s=1$ of $E(\cdot, f_{\tau,\sigma^\prime, s})$. This completes the proof of Theorem~\ref{thm-GGP-Reciprocal}.
\end{proof}

\subsection{Global Gan-Gross-Prasad conjecture for generic representations}
\label{subsection-GGP-proof}
We now state the main result on one direction of the global Gan-Gross-Prasad conjecture for generic representations of $\GSpin$ groups. 

\begin{theorem}
\label{thm-global-GGP}
Let $k>\ell$.
Let $\pi$, $\sigma$ be irreducible generic cuspidal automorphic representations of $\GSpin_{m^\prime-2\ell-1}(\A)$ and $\GSpin_{m^\prime-2k}(\A)$ respectively such that $\omega_\pi \omega_\sigma=1$. 
If the Bessel period $\mathcal{B}^{\psi_{k-\ell-1},-a}(\phi_\pi, \phi_\sigma)$ is non-zero for some choice of data $\phi_\pi\in V_\pi, \phi_\sigma\in V_\sigma$, then the tensor product $L$-function $L(s,\sigma\times\pi)$ is holomorphic and non-zero at $s=\frac{1}{2}$.
\end{theorem}

\begin{proof}
Let $\Pi$ be the functorial transfer of $\pi$ and let $\tau=\Pi\otimes\omega_\pi^{-1}$. By the same proof as in Section~\ref{subsection-GGP-Reciprocal-converse}, we see that the Eisenstein series $E(\cdot, f_{\tau,\sigma^\prime,s})$ must have a pole at $s=1$ of order exactly equal to $r$. By Proposition~\ref{prop-GGP-EisensteinSeriesPoles}, we conclude that $L(s,\sigma^\prime\times\tau \otimes \omega_\sigma^{-1})$ is holomorphic and non-zero at $s=\frac{1}{2}$. Note that
\begin{equation*}
L(s, \sigma\times\pi)=L(s, \sigma\times \Pi)=L(s, \sigma\times \tau\otimes \omega_\pi)=L(s,\sigma\times\tau \otimes \omega_\sigma^{-1})=L(s,\sigma^\prime\times\tau \otimes \omega_\sigma^{-1}).	
\end{equation*}
Thus $L(s,\sigma\times\pi)$ is holomorphic and non-zero at $s=\frac{1}{2}$, as desired. This finishes the proof of the theorem.
\end{proof}

\begin{remark}
The assumption that $\pi$ and $\sigma$ are generic in Theorem~\ref{thm-global-GGP} is used to apply the result of Asgari and Shahidi \cite{AsgariShahidi2014} on the image of the Langlands functorial transfer for generic representations of $\GSpin$ groups. Recently the Langlands functorial transfer for any cuspidal representations, generic or not, of the split $\GSpin$ groups has been established by Cai, Friedberg and Kaplan \cite{CaiFriedbergKaplan2024}. We expect that Theorem~\ref{thm-global-GGP} can be extended to non-generic representations on split $\GSpin$ groups. 
\end{remark}

\begin{remark}
The other direction of the global Gan-Gross-Prasad conjecture for $\GSpin$ groups is more delicate. We hope to report our progress on the other direction in the future.	
\end{remark}

\bibliographystyle{alpha}
\bibliography{References}

\end{document}